\documentclass[12pt,a4paper,reqno]{amsart}  
\title{Braided Hopf algebras and gauge transformations}

\date{March 2022}
\author{Paolo Aschieri, Giovanni Landi, Chiara Pagani}

\usepackage{amsmath, amssymb, amsthm} 
\usepackage[english]{babel} 
\usepackage{fullpage}  
\usepackage{enumerate}
\usepackage[all]{xy}
\usepackage{graphicx}
\usepackage{mathtools}
\usepackage{mathrsfs} 
\usepackage{hyperref}
\usepackage[normalem]{ulem}
\usepackage{etoolbox}  
\numberwithin{equation}{section}

\usepackage{pifont, dsfont, bbm}
 \newcommand{\II}{\mathds{1}}

\address[]{\textit{Paolo Aschieri}  \newline \indent 
Universit{\`a} del Piemonte Orientale, 
\newline \indent 
Dipartimento di Scienze e Innovazione Tecnologica
\newline \indent   viale T.~Michel~11,~15121~Alessandria,~Italy,
\newline \indent  and INFN Torino, via P.~Giuria~1, 10125~Torino,~Italy}
\email{paolo.aschieri@uniupo.it}

\address[]{\textit{Giovanni Landi }
\newline \indent  Universit\`a di Trieste, Via A. Valerio 12/1, 34127 Trieste, Italy,
\newline \indent Institute for Geometry and Physics (IGAP) Trieste, Italy, 
\newline \indent and INFN, Sezione di Trieste, Trieste, Italy. }
\email{landi@units.it}

\address[]{\textit{Chiara Pagani} 
\newline \indent     Universit\`a di Trieste, Via A. Valerio, 12/1, 34127  Trieste, Italy. }
\email{cpagani@units.it}

\let\oldtocsection=\tocsection
\let\oldtocsubsection=\tocsubsection

 \renewcommand{\tocsection}[2]{\hspace{0em}\oldtocsection{#1}{#2}}
\renewcommand{\tocsubsection}[2]{\hspace{1em}\oldtocsubsection{#1}{#2}}

\theoremstyle{plain}
\newtheorem{thm}{Theorem}[section]
\newtheorem{lem}[thm]{Lemma}
\newtheorem{prop}[thm]{Proposition}

\newtheorem{cor}[thm]{Corollary}
\newtheorem{defi}[thm]{Definition}

\theoremstyle{definition}
\newtheorem{ex}[thm]{Example}

\AtEndEnvironment{ex}{\null\hfill \tiny{$\blacksquare$}}

\theoremstyle{remark}

\newtheorem{Q}{Question/Comment}
\newtheorem{rem}[thm]{Remark}

\newcommand{\finex}{ \hfill {\tiny $\blacksquare $ } \end{ex} }

\newcommand{\al}{\alpha}

\newcommand{\dd}{\mathcal{D}}

\newcommand{\nn}{\nonumber}
\newcommand{\ot}{\otimes}
\newcommand{\beq}{\begin{equation}}
\newcommand{\eeq}{\end{equation}}

\newcommand{\btrl}{\blacktriangleright}
\newcommand{\trl}{\vartriangleright}

\newcommand{\trr}{\triangleleft}
\newcommand{\bbK}{\mathds{k}}
\newcommand{\id}{\mathrm{id}}

\newcommand{\A}{\mathcal{A}}

\newcommand{\M}{\mathcal{M}}
\newcommand{\C}{\mathcal{C}}
\renewcommand{\O}{\mathcal{O}}

\newcommand{\IR}{\mathbb{R}}

\newcommand{\can}{\chi}

\newcommand{\flip}{\tau}

\newcommand{\ur}{\mathsf{u}_\r}
\newcommand{\bur}{\bar{\mathsf{u}}_\r}
\newcommand{\uf}{\mathsf{u}_\F}
\newcommand{\buf}{\bar{\mathsf{u}}_\F}
\renewcommand{\r}{\mathsf{R}}
\newcommand{\br}{{\overline{\r}}}
\newcommand{\T}{T}

\newcommand{\rF}{{\r_\F}}
\newcommand{\brF}{{\br_\F}}
\newcommand{\RA}{\A^H}
\newcommand{\KM}{{}_{K}{\M}}
\newcommand{\KA}{{}_{K}{\A}}
\newcommand{\KC}{{}_{K}{\mathcal{C}}}
\newcommand{\KF}{{K_\F}}
\newcommand{\KFM}{{}_{\KF}{\M}}
\newcommand{\KFA}{{}_{\KF}{\A}}
\newcommand{\KFC}{{}_{\KF}{\mathcal{C}}}

\def\hsr#1#2{\langle #1,#2\rangle}
\def\hsrF#1#2{\langle #1,#2\rangle_{\F}}

\newcommand{\cun}{\varepsilon}
\newcommand{\K}{K}
\newcommand{\ad}{\mathsf{ad}}
\newcommand{\tad}{\trl_{\ad}}
\newcommand{\tadF}{\trl_{\ad_\F}}
\newcommand{\x}{\mathsf{x}}

\newcommand{\tadr}{\btrl_{\ad_\r}}
\newcommand{\tDelta}{\underline{\Delta}}
\newcommand{\tS}{\underline{S}}
\newcommand{\tK}{\underline{\K}}

\renewcommand{\cot}{\gamma}

\newcommand{\F}{\mathsf{F}}

\newcommand{\bF}{{\overline{\mathsf{F}}}}
\newcommand{\otF}{\ot_\F}

\def\dotF{{\,\raisebox{1.5pt}{\scalebox{0.55}[0.55]{\mbox{$\bullet_{\mbox{$\F\,$}}$}}}}}
 \def\dott{{\,\raisebox{1.5pt}{\scalebox{0.55}[0.55]{\mbox{$\bullet_{\mbox{$\theta\!\:$}}$}}}}}
\def\dotcot{{\,\raisebox{1.5pt}{\scalebox{0.55}[0.55]{\mbox{$\bullet_{\mbox{$\cot\,$}}$}}}}}
\newcommand{\zero}[1]{{#1}_{\scriptscriptstyle{(0)}}}
\newcommand{\one}[1]{{#1}_{\scriptscriptstyle{(1)}}}
\newcommand{\two}[1]{{#1}_{\scriptscriptstyle{(2)}}}
\newcommand{\three}[1]{{#1}_{\scriptscriptstyle{(3)}}}
\newcommand{\four}[1]{{#1}_{\scriptscriptstyle{(4)}}}

\newcommand{\fone}[1]{{#1}_{\scriptscriptstyle{[1]}}}
\newcommand{\ftwo}[1]{{#1}_{\scriptscriptstyle{[2]}}}

\newcommand{\tone}[1]{{#1}_{\scriptscriptstyle{\underline{(1)}}}}
\newcommand{\ttwo}[1]{{#1}_{\scriptscriptstyle{\underline{(2)}}}}

\newcommand{\tuno}[1]{{#1}^{\scriptscriptstyle{<1>}}}
\newcommand{\tdue}[1]{{#1}^{\scriptscriptstyle{<2>}}}
\newcommand{\Aut}[1]{\mathrm{Aut}_B(#1)}
\newcommand{\auto}[1]{\mathrm{aut}_B(#1)}
\newcommand{\aut}[1]{\mathrm{aut}^\r_B(#1)}

\newcommand{\autF}[1]{\mathrm{aut}^{\r_\F}_{B_\F}(#1)}

\newcommand{\Der}[1]{\mathrm{Der}{(#1)}}

\newcommand{\FP}{\mathsf{G}}

\newcommand{\g}{\mathfrak{g}}
\newcommand{\Hom}{{\rm{Hom}}}
\newcommand{\lie}{{\mathcal{L}}}
\newcommand{\Ad}{\mathrm{Ad}}
\newcommand{\U}{\mathcal{U}}

\newcommand{\stwo}{\tfrac{1}{\sqrt{2}}}
\newcommand{\parn}[1]{\partial_{\scriptscriptstyle{#1}}}

\newcommand{\Ru}[2]{R_{\!}\left({#1}\ot{#2}\right)}

\newcommand{\tpr}{\,{\raisebox{.28ex}{\rule{.6ex}{.6ex}}}\,}
\renewcommand{\bot}{\boxtimes}
\newcommand{\bs}{\underline{S}}
\newcommand{\bh}{{\underline{H}}}

\renewcommand{\bot}{\boxtimes}
\newcommand{\hpr}{{\:\cdot_{^{_{{{\:\!\!\!\!\!-}}}}}}}

\newcommand{\wdg}[2]{\, \textsf{#1}  \wedge \textsf{#2}}

\begin{document}

\begin{abstract}
We study infinitesimal gauge transformations of  an equivariant
  noncommutative principal bundle as a braided Lie algebra of  derivations. 
For this, we analyse general $K$-braided  Hopf and Lie algebras, for $K$ a (quasi)triangular Hopf algebra of symmetries,
and study their representations as
braided derivations. We then study Drinfeld twist deformations of braided Hopf algebras and of Lie algebras of infinitesimal gauge transformations. We give examples coming from  deformations of abelian and Jordanian type. 
In particular we explicitly describe the braided Lie algebra of gauge
transformations of the instanton bundle and of the orthogonal
bundle on the quantum sphere $S^4_\theta$.\\[-3em]
\end{abstract}
\maketitle

\tableofcontents 
\parskip=.75 ex

\allowdisplaybreaks[4]

\section[Intro]{Introduction}

In this paper we study different notions of gauge group and 
infinitesimal gauge transformations of
noncommutative principal bundles (Hopf--Galois extensions).  
This is a key ingredient for a differential
geometry study of the theory of connections and of their moduli
spaces.
The group of gauge transformations was considered in \cite{brz-tr}
(see also \cite{brz-maj} and \cite{strong}).
A general feature of these works is that the group 
there defined is bigger than one would expect. In \cite{pgc} we
studied the gauge group of noncommutative bundles with commutative
base space algebra and noncommutative Hopf algebra as structure group.
In particular, we proved that noncommutative principal bundles arising via twist
deformation of the structure Hopf algebra of commutative principal bundles have gauge group
isomorphic to the gauge group of the initial classical
bundle.

On one hand in the present paper we improve those results and study the case of principal
bundles with noncommutative base space algebra.
 For any Hopf--Galois extension
$B=A^{co H}\subseteq A$, with total space algebra $A$, base space
algebra $B$ and structure group Hopf algebra $H$, we define the gauge group $\Aut{A}$
of right $H$-comodule algebra morphisms that restrict to the identity on
the subalgebra $B$. This definition encompasses the above mentioned
Drinfeld twist deformations case where the gauge group remains
classical. However, for arbitrary Hopf--Galois extensions it usually gives very
small groups because in general there are few algebra maps on a noncommutative space. 
For example on Galois objects (Hopf--Galois extensions with $B$
 the ground field) the gauge group is just made of characters. 

 On the other hand we study a second definition of gauge group 
 for Hopf--Galois extensions with structure Hopf algebra $H$ and
 equivariant under a triangular Hopf algebra $K$. 
As expected, we show that in this case the natural notion of Lie algebra of gauge
transformations is that internal to the representation category of the
triangular Hopf algebra $K$. A Lie algebra object in that category is
a $K$-module with a braided antisymmetric bracket compatible with the
$K$-action, that is, it is a $K$-braided Lie algebra. Similarly, the gauge group is a
$K$-braided Hopf algebra.
The framework of braided Hopf algebras as Hopf algebras internal in the
representation category of a quasitriangular Hopf algebra has been
proposed in \cite{maj-braid, majid-bla}.  A notion of braided gauge
group was also considered in \cite{durdevic}, albeit in a different braided
context.

The advantage of considering braided Hopf and Lie algebras is clear
when noncommutative principal bundles, with now noncommutative base
space, are obtained via deformation quantization of classical ones.  In this braided approach the (infinite dimensional) vector space underlying infinitesimal gauge
transformations is the same as that of the undeformed bundle.
This follows from the general categorical framework mentioned before, which is well
adapted to Drinfeld twist deformations. It is also illustrated in the key
examples of the instanton bundle and the orthogonal bundle on the
4-sphere $S_\theta^4$. In this latter both the base space and the structure group
are noncommutative. These results (for the instanton bundle) suggest a
correspondence with those on noncommutative instanton moduli spaces
whose dimensions survives the $\theta$-deformation see for
example \cite{LS} (see also \cite{BL, BLS13}).

Another relevant application of this braided gauge Lie algebra approach
is for noncommutative coset spaces --not necessarily related to
twist deformations-- whose total space $A$ is a triangular Hopf
algebra. The braided Lie algebra of infinitesimal gauge transformations
is the subalgebra of vertical vector fields of the braided Lie algebra
of vector fields. This latter  is
generated (over the base space algebra) by the right-invariant vector fields defining the
bicovariant differential calculus \`a la Woronowicz on $A$. This
relationship between differential calculi and braided infinitesimal
gauge transformations supports their relevance for a theory of connections.

Infinitesimal gauge transformations similar to those in the present
paper have appeared in the 
mathematical physics literature. We just mention the recent paper
\cite{Szabo21} where only trivial principal bundles are considered using formal
deformation quantization with $\star$-products.
A further independent argument in favour of a theory of noncommutative gauge
groups that, like the braided one we develop, does not drastically depart from the 
classical one comes from the Seiberg-Witten map between 
commutative and noncommutative gauge theories
\cite{Seiberg:1999vs}. This map 
establishes a one-to-one correspondence between the corresponding gauge transformations (yet of a different kind from the
braided gauge transformations we consider) and hence points to  
noncommutative  gauge equivalence classes that are a deformation of the classical ones.
\\

The $K$-braided Lie algebra of infinitesimal gauge transformations we
consider is an example of Lie algebra in the representation category
of $K$-modules. The first part of the  paper provides a self
contained introduction to  the theory of braided Hopf algebras, these
are Hopf algebras in the representation category of a
quasitriangular Hopf algebra $K$.
 In Section 4 we twist with a Drinfeld 2-cocycle $\F\in
K\otimes K$ the quasitriangular Hopf algebra
$K$  to the quasitriangular Hopf algebra
$K_\F$ and recall the monoidal equivalence betweeen $K$-modules and
$K_F$-modules.
$K$-braided Hopf algebras are twisted to
$K_F$-braided Hopf algebras. In Section 5 we consider  $K$  triangular,  we discuss braided Lie
algebras, that were pioneered by \cite{Gurevich} as generalized Lie
algebras. We twist these braided Lie
algebras and their
universal enveloping algebras, which are braided Hopf algebras.  This leads to
isomorphisms (denoted ${\mathcal{D}}$) of $K_F$-braided Lie algebras that
we study in detail in particular in the category of relative
$K$-modules $B$-bimodules, that is, when there is a $K$-module algebra
$B$ with a $K$-compatible action on the braided Lie algebra. 
Indeed, infinitesimal gauge transformations of a Hopf--Galois extension
$B=A^{co H}\subseteq A$ form a $K$-module and
a compatible $B$-module: they are a $K$-braided Poisson algebra.

In Section 6 the gauge group of a Hopf--Galois extension $B=A^{co
  H}\subseteq A$ as $H$-equivariant algebra endomorphisms of $A$ is
considered, this does not require braided geometry notions. On the contrary, Section 7
applies the braided geometry of the initial sections to study the
braided Lie algebra of infinitesimal gauge transformations of 
$K$-equivariant Hopf--Galois extensions $B=A^{co
  H}\subseteq A$. As discussed at the beginning of this introduction,
this notion is natural in the $K$-representation category, and it is
supported  by the relationship beween infinitesimal gauge
transformations and differential calculi.  The last section applies the Drinfeld twist deformation
machinery to provide $K_\F$-braided gauge Lie algebras 
from $K$-braided gauge Lie algebras. Thus gauge transformations of
twisted principal bundles arising from commutative principal bundles
are presented. The theory is first illustrated for a trivial principal 
bundle. Then the two main examples are those of the instanton and
orthogonal bundles on the $S^4_\theta$-spheres ($\O(S^4_\theta)\subseteq
\O(S^7_\theta)$ and $\O(S_\theta^4)\subset
\mathcal{O}(SO_\theta(5,\mathbb{R}))$ respectively)
where the braided Lie algebra of the gauge group is explicitly
presented. Two further examples considering Jordanian twist deformations
(rather than abelian $\theta$-deformations via actions of tori)
are also considered. While all these examples are on the ground field
$\mathbb{C}$ a last example considers the general case in the
context of formal deformation quantization.

\section{Basic definitions and notations} ~\\
On this preliminary section we collect basic definitions and fix notation. 
All linear spaces are $\bbK$-modules, where 
$\bbK$ is a commutative field with unit $1_\bbK$, or the ring of formal power series in a variable $\hbar$ over a field. 
Much of what follows can be generalised to $\bbK$ a commutative unital ring. 
We denote the tensor product over $\bbK$  by $\otimes$.  For $\bbK$-modules $V,W$, we denote by $\flip$ the flip operator $\flip: V \ot W \to W \ot V, ~ v \ot w \mapsto w \ot v$.

All algebras (coalgebras) are over $\bbK$ and assumed to be unital and associative (counital and coassociative). Morphisms of algebras (coalgebras) are unital (counital).
The product in an algebra $A$ is denoted $m_A: A \ot A \to A$, 
and 
 the unit map $\eta_A: \bbK \to A$, with $1_A:= \eta_A(1_\bbK)$ the unit element.  
The  counit and coproduct of a coalgebra $C$ are denoted $\cun_C: C \to \bbK$ and $\Delta_C: C \to C \ot C$ respectively.
We use the standard Sweedler notation for the coproduct:  $\Delta_C(c)= \one{c} \ot \two{c}$ (with sum understood) 
and for its iterations, $\Delta_C^n=(\id \ot  \Delta_C) \circ\Delta_C^{n-1}: c \mapsto \one{c}\ot \two{c} \ot 
\cdots \ot c_{\scriptscriptstyle{(n+1)\;}}$, for all $ c \in C$.
Moreover, for a Hopf algebra $H$, we denote $S_H: H \to H$ its antipode.
We denote by $*$ the convolution product in the  dual $\bbK$-module
$C':=\mathrm{Hom}(C,\bbK)$, $(f * g) (c):=f(\one{c})g(\two{c})$, for all  $f,g \in C'$, $c \in C$.
For a coalgebra $C$, we denote $C^{cop}$ the co-opposite coalgebra: $C^{cop}$ is the coalgebra structure on the $\bbK$-module $C$ with comultiplication $\Delta_{C^{cop}}:=\flip \circ \Delta_C$
and counit $\cun_{C^{cop}}:= \cun_C$. If $C$ is a bialgebra, such is $C^{cop}$ with  multiplication and unit the same of $C$.

\subsection{Hopf algebra modules and comodules} Given a bialgebra (or Hopf algebra) $H$, 
a left \textbf{$H$-module} is a $\bbK$-module $V$ with a left
$H$-action: a $\bbK$-module map
$\trl_V: H \ot V\to V$, $h\ot v \mapsto h \trl_V v$, such that 
\beq \label{eqn:Hmodule}
(hk) \trl_V v = h \trl_V (k \trl_V v) \; , \quad 1 \trl_V v = v \, ,
\eeq
for all $v \in V$, $h,k \in H$.
Given two $H$-modules $(V, \trl_V),~(W, \trl_W)$,  a map $\psi:V\to W$ such that 
$\psi (h \trl_V v) = h \trl_W (\psi(v))$, for all $h\in H$, $v\in V$,
is called  
a morphism of left $H$-modules. This condition will also be referred to as $H$-equivariance. 
The tensor product (as $\bbK$-modules) 
$V\otimes W$ is an $H$-module with action  
\begin{align}\label{trlVW}
 \trl_{V\otimes W} : H \ot V\otimes W \longrightarrow  V\otimes W~, \quad 
h \ot  v\otimes w  \longmapsto (\one{h} \trl_V v ) \ot (\two{h} \trl_W w).
\end{align}
We denote by {${}_H \M$ the category of left $H$-modules. 
It is a monoidal category, whose  
unit object is $\bbK$ with (trivial) action  given by the counit map $\trl_{\bbK}=\cun_H: H \simeq H \ot \bbK\to \bbK$.
For $H$ a Hopf algebra, 
the linear space $\Hom(V,W)$ of   $\bbK$-linear maps $\psi: V \to W$ is a left $H$-module with action
\begin{align}\label{action-hom}
\trl_{\Hom(V,W)}:  H \ot \Hom(V,W) &\to \Hom(V,W) 
\nn\\
h \ot \psi & \mapsto  h \trl_{\Hom(V,W)} \psi : \; v \mapsto \one{h} \trl_W \psi(S(\two{h})\trl_V v) .
\end{align}
This action is trivial on the  subspace $\Hom_{{}_H \M}(V,W)$
consisting of $H$-module morphisms: $h \trl_{\Hom(V,W)} \psi= \varepsilon(h) \psi$. 

An algebra $A$ which is a left $H$-module, $(A, \trl_A)$, is called a  left
$H$-\textbf{module algebra} if its multiplication $m_A: A \ot A\to A$ and unit $\eta_A: \bbK \to A$ are morphisms of $H$-modules:
\beq\label{mod-a}
h \trl_A (ab)= (\one{h}\trl_A a)(\two{h}\trl_A b) \, ,  \quad h \trl_A 1= \cun(h)1
\eeq
for all $a,b \in A$, $h \in H$. 
We denote by \textbf{${}_H \A$} the category of
left $H$-module algebras, with morphisms in it being $H$-module
morphisms which are also algebra maps. 
The algebra $(\Hom(V,V), \circ)$ of  $\bbK$-linear maps from an $H$-module $V$ to itself, with multiplication given by map composition  and action as in \eqref{action-hom}, is a  left $H$-module algebra.
A module $V$ that is
a left $H$-module and a left $A$-module, with $A$-action denoted
$\cdot_V$,  is called  
$(H,A)$-\textbf{relative Hopf
module} if for all $h\in H$, $a\in A$, and $v\in V$,
\beq\label{eqn:KArelmod}
h \trl_V (a\cdot_V v)=
(\one{h}\trl_V a)\cdot_V(\two{h}\trl v) ~. 
\eeq
A morphism of 
  $(H,A)$-relative Hopf modules is a morphism of
right $H$-modules  which is also a morphism of left $A$-modules. 

Furthermore, a left $H$-\textbf{module coalgebra}  
is a coalgebra $C$ which is a left $H$-module and its 
coproduct and counit are morphisms of $H$-modules:
\beq\label{mod-co}
\one{(h \trl_C c)} \ot \two{(h \trl_C c)} = (\one{h} \trl_C \one{c}) \ot (\two{h} \trl_C \two{c})  \, , 
\quad  
\cun(h \trl_C {c}) = \cun(h) \cun(c)\, 
\eeq
 for all $c \in C$, $h \in H$. We denote by  \textbf{${}_H \C$} 
 the category of left $H$-module coalgebras; 
morphisms in $_H{\C}$ 
are $H$-module maps which are also coalgebra maps.

Analogous definitions (and notations) are given for a right $H$-module $(V, \trr_V)$, and corresponding categories of right $H$-module algebras $\A_H$ and coalgebras $\C_H$. Recall that if $H$ is a Hopf algebra, any 
left $H$-module $(V, \trl_V)$ can be made into a right $H$-module $(V, \trr_V)$ with right action
$\trr_V : V \ot H \to V , \,\, v \ot h \mapsto v \trr_V h := S(h) \trl_V v $. Moreover if $(V,\trl_V)  \in \A_H$, then 
$(V, \trr_V) \in \A_{H^{cop}}$, with ${H^{cop}}$ the co-opposite Hopf algebra.

Dually,
a right \textbf{$H$-comodule} is a $\bbK$-module $V$ with a $\bbK$-linear map
$\delta^V:V\to V\otimes H$ (a right $H$-coaction) such that 
\beq \label{eqn:Hcomodule}
(\id\otimes \Delta)\circ \delta^V = (\delta^V\otimes \id)\circ \delta^V~,\quad 
(\id\otimes \cun) \circ \delta^V =\id \, . 
\eeq
In Sweedler notation we write $\delta^V:V\to V\otimes H$, 
$v\mapsto \zero{v}\otimes \one{v}$ (with sum understood), and the right $H$-comodule properties
\eqref{eqn:Hcomodule}  read
$$
\zero{v} \otimes \one{(\one{v})}\otimes \two{(\one{v})} = \zero{(\zero{v})}\otimes \one{(\zero{v})} \otimes \one{v}=: \zero{v} \ot\one{v} \ot \two{v} \, ,\quad
\zero{v} \,\cun (\one{v}) = v~
$$
for all $v\in V$.
A morphism between  $H$-comodules $V, W$ is a $\bbK$-linear map $\psi:V\to W$ which is an $H$-comodule map, that is $\delta^W\circ \psi=(\psi\otimes\id)\circ \delta^V$, or
$\zero{\psi(v)}\otimes \one{\psi(v)} = \psi(\zero{v})\otimes \one{v}$
for all $v\in V$.
Also, the tensor product (as $\bbK$-modules) 
 $V\otimes W$ is an $H$-comodule with  the right $H$-coaction  
\begin{align}\label{deltaVW}
 \delta^{V\otimes W} :V\otimes W \longrightarrow  V\otimes W\otimes H \, , \quad 
 v\otimes w & \longmapsto \zero{v}\otimes \zero{w} \otimes 
 \one{v}\one{w} ~.
\end{align}

We denote by \textbf{$\M^H$} the category of right $H$-comodules; it is a monoidal category. The
unit object is $\bbK$ with coaction $\delta^{\bbK}$
given by the unit map $\eta_H:\bbK\to \bbK\otimes H \simeq H$.

A right $H$-\textbf{comodule algebra} is an algebra  $A$  which is 
a right $H$-comodule and has multiplication and unit which are morphisms of $H$-comodules. 
This is equivalent to requiring the coaction $\delta^A: A\to A\otimes H$ to be
a morphism of unital algebras (where $A\otimes H$ has the usual tensor
product algebra structure):  for all $a,a^\prime\in A \, $,
$$
\delta^A(a\,a^\prime) =\delta^A(a)\,\delta^A(a') \, , \quad
\delta^A(1_A) = 1_A\otimes 1_H \, .
$$
We denote by \textbf{$\A^H$} the category of right $H$-comodule algebras with 
morphisms just $H$-comodule maps which are also algebra maps.
 
Finally, a right $H$-\textbf{comodule coalgebra}  
is a coalgebra $C$ which is a right $H$-comodule and  has 
coproduct and counit which are morphisms of $H$-comodules, 
that is,
\beq\label{com-co}
\zero{(\one{c})} \ot \zero{(\two{c})} \ot \one{(\one{c})} \one{(\two{c})} 
=
\one{(\zero{c})} \ot \two{(\zero{c})} \ot \one{c} \, , 
\quad  
\cun(\zero{c}) \one{c}=\cun(c) 1_H \, 
\eeq
 for each $c \in C$.
We denote by  \textbf{${\C}^H$} the category of right $H$-comodule coalgebras; 
its morphisms are $H$-comodule maps which are also coalgebra maps.
Analogous notions and notation can be introduced for left comodules (algebras and coalgebras).

\subsection{Relative modules and K-equivariance}\label{sub:Kequiv}

Let $H$ be a bialgebra and let $A\in\A^H$.
An $(H,A)$-\textbf{relative Hopf module}
$V$ is a right $H$-comodule with a left  $A$-action $\trl_V$ which is a morphism of $H$-comodules: for all $a\in A$ and $v\in V$,
\beq\label{eqn:modHcov} 
\zero{(a \trl_V v)} \ot \one{(a \trl_V v)} = \zero{a} \trl_V \zero{v} \ot \one{a}\one{v} ~. 
\eeq
A morphism of 
 $(H,A)$-relative Hopf modules is a morphism of
right $H$-comodules  which is also a morphism of left $A$-modules.
We denote by $ {}_{A}\M^H$ the category of 
 $(H,A)$-relative Hopf modules.
In a similar way one defines the categories of relative Hopf modules ${{\mathcal M}_A}^{\!H}$ for $A$ acting on the right, 
 and ${{}_{A}{\mathcal M}_A}^{\!H}\,$ for right and left  $A$
 compatible actions.

For two Hopf algebras
$K$ and $H$, a \textbf{$K$-equivariant $H$-comodule} $V$ is 
an
$(H,K)$-relative Hopf module with $K\in \A^H$ with a trivial coaction
$K\to H\ot K, k\mapsto k\otimes 1_H$. The left $K$-action $\trl_V : K\otimes V\to V$ and
the right $H$-coaction $\delta^V : V\to V\otimes H$ on $V$ thus
satisfy $\delta\circ \trl_V = (\trl_V\ot \id)\circ (\id \ot \delta)$ that is,
\beq\label{compatib}
\zero{(k \trl_V v)} \ot \one{(k \trl_V v)} = k \trl_V \zero{v} \ot \one{v}  , 
\eeq 
for all $k\in K, v\in V$.

We denote by ${}_K{\M}{}^{H}$  the category of  $K$-equivariant $H$-comodules.
It is a monoidal category: the tensor product 
of $V,W\in{}_K{\M}{}^{H}$ is the object $V\otimes W\in{}_K{\M}{}^{H}$
with tensor $K$-module structure  in \eqref{trlVW} and $H$-coaction  in \eqref{deltaVW}.
We  also consider the category ${}_K{\A}{}^{H}$
of  \textbf{$K$-equivariant $H$-comodule algebras}, where objects and morphisms are  in
${}_K{\A}{}^{H}$ if they are in  ${\A}{}^{H}$, ${}_K{\A}$ and ${}_K{\M}{}^{H}$.

\section{Quasitriangular structures and associated braidings}

\subsection{Quasitriangular Hopf algebras} ~\\
We recall some basic definitions and properties of quasitriangular Hopf algebras referring 
to  \cite[Ch.~8]{KS} or \cite[Ch.~2]{Majid} for more details and  proofs.  
This class of algebras is characterized by being  
cocommutative up to conjugation by a suitable element.

\begin{defi}\label{def:qt}
A bialgebra $\K$, with coproduct $\Delta$, is called \textbf{quasitriangular}  
if there exists an  element $\r:=\r^\alpha \ot \r_\alpha \in \K \ot \K$ (with an implicit sum) which is invertible:
$\exists \, \br \in \K \ot \K$, such that  $\r \br = \br \r = 1 \ot 1$;  $\Delta$ is quasi-cocommutative with respect to $\r$:  
\beq\label{iiR} 
\Delta^{cop} (k) = \r \Delta(k) \br
\eeq for each $k \in \K$, with $\Delta^{cop} := \flip \circ \Delta$ and $\tau$ the flip; and such that
\beq \label{iii} (\Delta \ot \id) \r=\r_{13} \r_{23} \qquad \mbox{and} \qquad 
(\id \ot \Delta)\r=\r_{13} \r_{12}.
 \eeq
Here $\r_{12}=\r^\alpha \ot \r_\alpha \ot 1$  and similarly for $\r_{13}$ and $\r_{23}$. 
Then conditions \eqref{iii}  become
\beq\label{iiiR}
\one{\r^\alpha} \ot \two{\r^\alpha} \ot \r_\alpha = \r^\alpha \ot \r^\beta \ot \r_\alpha \r_\beta
\quad ; \quad
\r^\alpha \ot \one{\r_\alpha} \ot \two{\r_\alpha}  = \r^\alpha  \r^\beta \ot \r_\beta \ot \r_\alpha 
\; .
\eeq
The element $\r$ is called a \textbf{universal $R$-matrix} of
$\K$.  
A quasitriangular bialgebra $(\K,\r)$ is \textbf{triangular} if $\br= {\r}_{21},$ with ${\r}_{21}= \flip(\r)=\r_\alpha \ot \r^\alpha$.
\end{defi}

The $\r$-matrix  of a quasitriangular bialgebra $(\K,\r)$ is unital,  
\beq\label{Rnormalized}
  (\cun \ot \id) \r = 1 =  (\id \ot \cun) \r   
\eeq 
with $\cun$ the counit of $K$, and satisfies the quantum Yang--Baxter equation
$$
\r_{12} \r_{13} \r_{23}= \r_{23} \r_{13}\r_{12} ,
$$ 
or in components, 
\beq\label{YB}
\r^\alpha \r^\beta \ot \r_\alpha \r^\gamma \ot \r_\beta \r_\gamma=
\r^\beta \r^\gamma \ot \r^\alpha \r_\gamma  \ot \r_\alpha \r_\beta \; .
\eeq

By the defining properties for a $\r$-matrix  it follows that if $\r$ is a $\r$-matrix  of $\K$, then
$\br_{21}$ is also a $\r$-matrix  of $\K$.
Moreover, $\K$ is cocommutative, $\Delta^{cop} = \Delta$, if and only if $\r=1 \ot 1$ is a $\r$-matrix  of $\K$.

A  Hopf algebra which is quasitriangular as a bialgebra is called a \textbf{quasitriangular Hopf algebra}. 
Its antipode is such that 
\beq\label{ant-R}
(S \ot \id) (\r)= \br  \; ; \quad
(\id \ot S) (\br)= \r \; ; \quad
(S \ot S) (\r)= \r \; .
\eeq
and is invertible. Indeed, consider the invertible element
$
\ur:= S(\r_\alpha)\r^{\alpha} 
$
with inverse $\bur = \r_\alpha S^2(\r^{\alpha})$. Then,
$
S^2(k)=  \ur \, k \, \bur ,
$
for all $k \in \K$. This relation and the fact that $S(\ur) \ur=\ur S(\ur)$
is central, yields the following expression for the inverse of $S$:
\beq\label{Sinv}
S^{-1}(k)=\bur \, S(k) \, \ur \, ,\qquad  \forall k \in \K \,. 
\eeq

When $(\K,\r)$  is a quasitriangular bialgebra, the monoidal category $\KM$ of left
$\K$-modules  is braided monoidal: if $(V, \trl_V), (W, \trl_W)$ are two left $\K$-modules, then the $\K$-modules $V \ot W$ and $W \ot V$, with actions 
as in \eqref{trlVW},
are isomorphic via the map
\beq\label{braiding}
\Psi^{\r}_{V,W}: V \ot W  \longrightarrow W \ot V ~, \quad v\ot w \longmapsto  
(\r_\alpha \trl_W w) \ot (\r^\alpha \trl_V v) \, .
\eeq 
\begin{rem}
In order to show that  
the braiding $\Psi^{\r}_{V,W}$ is an isomorphism of left $\K$-modules, only the condition $\br \Delta^{cop}= \Delta \br$ in \eqref{iiR} in Definition \ref{def:qt} is needed. Thus, since both $\r$ and $\br_{21}$ are $\r$-matrices for $\K$, one could  use
 $\br$ instead of $\r_{21}$  in the definition of the braiding.
\end{rem}
Additionally, the category of  left $\K$-module algebras $\KA$ is monoidal as well:

\begin{prop} \label{BTPalgK}
Let $(\K,\r)$ be a quasitriangular bialgebra. Let  $(A, \trl_A), (C, \trl_C)$ be left $\K$-module algebras, 
then the $\K$-module $A \ot C$ with tensor product action as in \eqref{trlVW}, 
 is a left $\K$-module algebra when endowed with the product
\beq\label{tpr}
(a \ot c) \tpr (a' \ot c'):= a~  \Psi^{\r}_{C,A} (c \ot a') c'=
a(\r_\alpha \trl_A a') \ot (\r^\alpha \trl_C c) c'\, .
\eeq
Moreover, if $\phi: A\to E$ and $\psi: C\to F$ 
are morphisms of $\K$-module algebras,  
then so is 
the map $\phi\otimes\psi : A \ot C \to E \ot F$, $a\otimes c\mapsto
\phi(a)\otimes \psi(c)$,
where $A\otimes C$ and $E\otimes F$ are both endowed with the $\tpr$-products.
\end{prop}
The $\K$-module algebra $(A \ot C, \tpr)$ is the
{\bf{braided tensor product algebra}} of $A$ and $C$. 
We
denote it by
$A \bot  C$ and write $a \bot c\in A \bot
 C$ for $a\in A$, $c\in C$, and 
similarly by 
 $\phi\bot\psi:=\phi\otimes\psi: A\bot C\to E\bot F$ the 
 tensor product of two morphisms. 
We denote by $({}_{(K,\r)}\A,\bot)$ or simply $({}_{K}\A,\bot)$ the monoidal category of left
 $K$-module algebras with respect to the quasitriangular structure
 $\r$.

\subsection{Braided Hopf algebras} \label{sec:braidedHa}~\\
We can next introduce braided Hopf algebras (see  \cite[ \S
9.4.2]{Majid} or \cite[\S 10.3.3]{KS}):
\begin{defi} \label{def:bbK}
Let  $(\K,\r)$ be a quasitriangular Hopf algebra. 
Let $(L ,\trl_L)$ be both a $\K$-module algebra 
$(L, m_L, \eta_L, \trl_L)$ and a $\K$-module coalgebra $(L, \Delta_L,\cun_L, \trl_L)$. Then $L$ is  
a \textbf{braided bialgebra associated with $(\K,\r)$} 
if it is a bialgebra
in the braided monoidal category $({}_\K\M,\otimes, \Psi^{\r})$ of $\K$-modules: 
\beq\label{cop-braid}
\cun_L\circ m_L  = m_\bbK \circ (\cun_L \ot \cun_L) \, , \quad
\Delta_L\circ m_L  = m_{L \bot L} \circ (\Delta_L\otimes\Delta_L) \, .
\eeq
That is, $\cun_L: L\to \bbK$ and  $\Delta_L: L \to L\bot L$ are algebra maps 
with respect to the product $m_L$ in $L$ and 
the product $m_{L \bot L} = (m_L \otimes m_L)  \circ (\id_L\otimes
\Psi^{\r}_{L,L} \otimes \id_L)$ in $L\bot L$ (as in
\eqref{tpr}).
We frequently term $L$ a $K$-braided bialgebra omitting to mention the choice 
 of quasitriangular structure on $K$. 
 
 A braided biagebra  $L$ is a braided Hopf algebra if there is a $\K$-module map
$S_L : L\to L $ which is
the convolution inverse of the identity
$\id: L\to L$:
\beq 
m_L\circ (\id_L\otimes S_L)\circ \Delta_L=
\eta_L\circ \cun_L =
m_L\circ (S_L\otimes\id_L)\circ \Delta_L ~. \label{bantipode}
\eeq
Such a map is called a (braided) antipode.
\end{defi}
  
 For later use we recall that the antipode $S_L: L\to L$ of a braided Hopf algebra $L$ is a  braided anti-algebra map
and a braided anti-coalgebra map:
\beq\label{S-algmap}
S_L\circ m_L=m_L\circ  \Psi^{\r}_{L,L}\circ (S_L\ot S_L)\, , \quad 
\Delta_L\circ S_L=(S_L\ot S_L)\circ \Psi^{\r}_{L,L}\circ \Delta_L \, .
\eeq
 
\begin{ex}\label{ex:braided}
For any  Hopf algebra $\K$,  the underlying algebra $(\K, m_\K, \eta_\K)$ becomes a left $\K$-module algebra 
with the left adjoint action 
\beq\label{ad}
\tad : \K \ot \K \to \K \; , \quad 
 h \ot k \mapsto h\tad k := \one{h} k S(\two{h}) .
\eeq
 In general, the coproduct $\Delta : \K \to \K \ot \K$ is not a morphism of $\K$-modules, for
 $\K \ot \K$ carrying the tensor product action as in \eqref{trlVW}. 
However, if  
$\K$ is quasitriangular with $\r$-matrix  $\r=\r^\alpha \ot \r_\alpha$, then the `braided' coproduct
\beq
\tDelta (k):= \one{k} S(\r_\alpha) \ot (\r^\alpha \tad \two{k})
=\one{k} \r_\beta S(\r_\alpha) \ot \r^\alpha  \two{k} \r^\beta 
\eeq
is now a morphism between the $\K$-modules $\K$ and
$\K\ot\K$.
This braided coproduct is coassociative and unital; together with  the
counit $\cun_\K$ it defines a coalgebra structure on
$\K$. We use the notation $\tDelta (k)= \tone{k} \ot \ttwo{k}$ to distinguish it from the coproduct in $K$.
The compatibility with the adjoint $\K$-action
implies that
$(\K,  \tDelta, \cun_\K, \tad)$ is a $\K$-module coalgebra in $\textbf{${}_\K \C$}$.  Moreover, the map 
\beq
\tS: \K \to \K \; , \quad  k \mapsto \tS(k):= \r_\alpha S( \r^\alpha \tad k) = 
\br_\alpha \bur S(k)  \br^{\alpha} 
\eeq
is a left $\K$-module morphism and
$$
\tK :=(\K, m_\K, \eta_\K, \tDelta, \cun_\K, \tS, \tad)
$$
is a braided Hopf algebra associated  with the quasitriangular Hopf algebra $(\K,\r)$.
\end{ex}

In the category of $K$-modules an action of a $K$-module algebra $L$
on a $K$-module $M$ is an action $\btrl_M : L\ot M\to M$ of the algebra
$L$ on $M$ which is $K$-equivariant:
\beq\label{action-L}
k \trl_M (\ell \btrl_M m)= (\one{k} \trl_L \ell) \btrl_M (\two{k} \trl_M m).
\eeq

Braided Hopf algebras are symmetry objects within the category of $K$-modules. 
\begin{lem}
Let ($K, \r)$ be a quasitriangular Hopf algebra and $L$ a
braided bialgebra associated with $K$ acting on $K$-modules $M$ and
$N$  as in \eqref{action-L}.  There is an action 
$$
\btrl_{M\ot N}:L\ot M\ot N\to M\ot N \, 
$$
on the tensor product $M\ot N$ given by
$$
\ell \btrl_{M\ot N}(m\ot n)=\one{\ell}\btrl_M (\r_\alpha\trl_M m) \ot  (\r^\alpha\trl_L \two{\ell})\btrl_N n~.
$$
\end{lem}
\begin{proof}
We show that this is an action: $(\ell\ell') \btrl_{M\ot N}(m\ot n)=\ell \btrl_{M\ot N} \ell'\btrl_{M\ot N}(m\ot n)$. For ease of notations in the proof we drop subscripts.
\begin{align*}
\ell\btrl (\ell'\btrl & (m\ot n)) = \ell \btrl (\one{\ell'}\btrl  (\r_\alpha\trl m) \ot  (\r^\alpha\trl \two{\ell'})\btrl n)
\\
&= \one{\ell}\btrl (\r_\beta\trl  (\one{\ell'}\btrl  (\r_\alpha \trl m)) \ot  (\r^\beta\trl \two{\ell})\btrl ((\r^\alpha\trl \two{\ell'})\btrl n)
\\
&= \one{\ell}\btrl 
(\one{\r_\beta} \trl \one{\ell'}) \btrl  (\two{\r_\beta}\r_\alpha\trl m)
\ot  (\r^\beta\trl \two{\ell})\btrl (\r^\alpha\trl \two{\ell'})\btrl n
\\
&
= (\one{\ell}
(\one{\r_\beta} \trl \one{\ell'})) \btrl  (\two{\r_\beta}\r_\alpha\trl m)
\ot  ((\r^\beta\trl \two{\ell}) (\r^\alpha\trl \two{\ell'})) \btrl n
\\
&=
 (\one{\ell}
({\r_\beta} \trl \one{\ell'})) \btrl  ({\r_\gamma}\r_\alpha\trl m)
\ot  ((\r^\gamma\r^\beta\trl \two{\ell}) (\r^\alpha\trl \two{\ell'})) \btrl n
\end{align*}
where we have used condition \eqref{action-L} on $M$ for the third equality and the  quasitriangular condition \eqref{iiiR} for the last one. 
On the other hand
\begin{align*}
 (\ell\ell')\btrl(m\ot n) &=  (\one{\ell} \r_\beta \trl \one{\ell'}) \btrl (\r_\alpha\trl m) \ot  (\r^\alpha\trl ((\r^\beta \trl \two{\ell})
\two{\ell'})
)\btrl n
\\
&=  (\one{\ell} \r_\beta \trl  \one{\ell'}) \btrl (\r_\alpha\trl  m) \ot  
((\one{\r^\alpha}\r^\beta \trl  \two{\ell}) (\two{\r^\alpha} \trl 
\two{\ell'})
)\btrl n
\\
&=  (\one{\ell} \r_\beta \trl  \one{\ell'}) \btrl (\r_\gamma \r_\alpha\trl  m) \ot  
((\r^\gamma \r^\beta \trl  \two{\ell}) ({\r^\alpha} \trl 
\two{\ell'})
)\btrl n
\end{align*}
where again we have used the quasitriangular condition for the last equality. 

\noindent
Next, we show $K$-equivariance of $\btrl_{M\ot N}$.  For all $k\in K, \ell\in L, m\in M, n\in N$, \begin{align*}
k\trl \big(\ell\btrl (m\ot n)\big)
&=
k\trl\big(\one{\ell}\btrl(\r_\alpha\trl  m) \ot (\r^\alpha\trl \two{\ell})\btrl n\big)\\
&=
\one{k}\trl\big(\one{\ell}\btrl(\r_\alpha\trl m)\big) \ot \two{k}\trl\big((\r^\alpha\trl\two{\ell})\btrl n\big)\\
&=
(\one{k}\trl\one{\ell})\btrl(\two{k}\trl (\r_\alpha\trl m)) \ot (\three{k}\trl (\r^\alpha\trl\two{\ell}))\btrl(\four{k}\trl n)\\
&=
(\one{k}\trl\one{\ell})\btrl(\r_\alpha\trl (\three{k}\trl m)) \ot (\r^\alpha\trl (\two{k}\trl\two{\ell}))\btrl(\four{k}\trl n)\\
&=(\one{k}\trl \ell)\btrl ( \two{k}\trl m\ot \three{k}\trl n)\\
&=(\one{k}\trl \ell) \btrl\big( \two{k}\trl (m\ot n)\big)~.
\end{align*}
In the second line we used that $L$ is a $K$-module algebra, in
the third that the action of $L$ on $M$ and on $N$ is $K$-equivariant. 
Then quasi-cocommutativity yields the result. 
\end{proof}

An action of a braided bialgebra 
$L$ on a $K$-module algebra $A$ is a  $K$-equivariant action $\btrl_A:L\ot A\to A$
(cf. \eqref{action-L}) which satisfies the condition} 
\beq\label{LAbraidedaction}
\ell\btrl_A (ac)=(\one{\ell}\btrl_A (\r_\alpha \trl_A  a))\, ((\r^\alpha \trl_L \two{\ell})\btrl_A c) ,
\eeq
for all $a,c\in A$.

An important example of the above construction is given by the adjoint action of a braided Hopf algebra $L$ on itself.
\begin{prop}
The  $\bbK$-linear map 
\begin{equation}\label{bLad}
\tadr : L\otimes L\to L , \quad
\ell\otimes \ell'\mapsto\ell \tadr \ell':=\one{\ell}(\r_\alpha \trl_L\ell') \, \r^\alpha \trl_L S_L(\two{\ell}), 
\end{equation} 
is a $K$-equivariant action of $L$ on the $K$-module $L$: 
\begin{align*}
(\ell {\ell'})\tadr x &= \ell\tadr ({\ell'}\tadr x)\, \\ 
k\trl_L(\ell \tadr x) & =(\one{k}\trl_L\ell)\tadr (\two{k}\trl_L x)
\end{align*}
for all $k\in K$ and $\ell, \ell', x  \in L$.
\end{prop}
\begin{proof}
To lighten notation we drop all subscripts. The right hand side reads  
\begin{align*}
\ell\btrl ({\ell'}\btrl x)&=
\one{\ell}  
\r_\beta\trl \big(\one{{\ell'}} 
(\r_\alpha \trl x)
\, \r^\alpha\trl S(\two{{\ell'}}) \big) \, 
\r^\beta \trl S(\two{\ell})
\\ &=
\one{\ell} 
(\r_\beta\trl \one{{\ell'}}) \, 
\r_\gamma\trl \big( (\r_\alpha\trl x) \, \r^\alpha\trl  S(\two{{\ell'}}) \big)
\, \r^{\gamma}{\r^\beta} \trl S(\two{\ell})
\\ &=
\one{\ell}
( \r_\beta\trl\one{{\ell'}})
\, (\r_{\gamma}{\r_\alpha} \trl x)
\, (\r_\delta\r^\alpha \trl S(\two{{\ell'}}) )
\, \r^{\delta}\r^\gamma\r^{\beta} \trl S(\two{\ell})~,
\end{align*}
were we used the quasitriangular condition in \eqref{iiiR}. 
The left hand side reads 
\begin{align*}
(\ell {\ell'})\btrl x&=\one{\ell} 
(\r_\beta\trl\one{{\ell'}})
\, (\r_\alpha\trl x)
\, \r^{\alpha}\trl S\big(({}^{}\r^{\beta}\trl\two{\ell}) \, \two{{\ell'}}\big)
\\ &=\one{\ell}
\, (\r_\beta\trl\one{{\ell'}})
\, (\r_\alpha\trl x)
\, \r^{\alpha}\trl\big( (\r_\gamma\trl S(\two{{\ell'}}))
\,  \r^\gamma \trl S(\r^{\beta}\trl \two{\ell}) \big)
\\ &=\one{\ell}\ (\r_\beta\trl\one{{\ell'}}) \, (\r_{\alpha}{\r_\delta} \trl x) 
\, (\r^{\alpha}\r_\gamma \trl S(\two{{\ell'}}) )
\, \r^{\delta}{\r^\gamma} \trl S({}^{}\r^{\beta}\trl\two{
  \ell})~,
\end{align*}
where we used the braided algebra map property of the coproduct \eqref{cop-braid}, the braided
antialgebra map property of the antipode  \eqref{S-algmap} and again
the quasitriangular condition in \eqref{iiiR}. 
The left hand side equals the right hand side using the Yang-Baxter
equation and the fact that the antipode 
is $K$-equivariant.
 We are left to show $K$-equivariance:
\begin{align*}
k\trl(\ell\btrl x)&=
k\trl\big(\one{\ell}(\r_\alpha \trl x) \, \r^\alpha \trl S(\two{\ell})\big)\\
&=(\one{k}\trl\one{\ell})(\two{k}\trl(\r_\alpha \trl x))\, (\three{k}\trl(\r^\alpha \trl S(\two{\ell}))) \\
&=(\one{k}\trl\one{\ell})(\r_\alpha\trl(\three{k} \trl x))\, (\r^\alpha\trl(\two{k} \trl S(\two{\ell}))) \\
&=\one{(\one{k}\trl{\ell})}(\r_\alpha\trl(\two{k} \trl x))\, (\r^\alpha\trl S(\two{(\one{k} \trl {\ell})})) \\
&=(\one{k}\trl\ell)\btrl (\two{k}\trl x) . 
\end{align*}
In the third line we used quasi-cocommutativity, in the fourth line the $K$-equivariance of the antipode
and that $L$ is a $K$-module coalgebra. 
\end{proof}

\begin{prop}
The adjoint action $\tadr$ in \eqref{bLad}  is an action of $L$ on the $K$-module algebra $L$ 
\begin{equation}\label{bralg}
\ell\tadr (x y)=(\one{\ell}\tadr (\r_\alpha\trl_L
  x)) \, ((\r^{\alpha}\trl_L\two{\ell})\tadr y) 
 \end{equation}
 for all $\ell, x,y\in L$; it is 
henceforth called braided adjoint action.
 It satisfies the Jacobi identity (it is a braided action with respect to the nonassociative product \eqref{bLad}), 
\begin{equation}\label{bralgAD}
\ell\tadr (x\tadr y)
=(\one{\ell}\tadr (\r_\alpha\trl_L x))\tadr 
((\r^{\alpha}\trl_L\two{\ell})\tadr y)~
\end{equation} 
for all $\ell, x,y\in L$. 
\end{prop}
\begin{proof}
We again omit all subscripts.
We start with the braided algebra map property \eqref{bralg}
(cf. \cite[Ex. 2.7]{maj-braid}, \cite[App.]{majid-bla}).
The left hand side reads:
$$
\ell\btrl (x y)=
\one{\ell} \, (\r_\alpha\trl (x y)) \, \r^\alpha \trl S(\two{\ell})=
\one{\ell} (\r_\alpha \trl x) \, (\r_\beta\trl y) 
\, \r^\beta \r^\alpha \trl S(\two{\ell})~.
$$
It equals the right hand side: 
\begin{align*}
(\one{\ell}\btrl & (\r_\alpha  \trl x)) \, ((\r^{\alpha} \trl \two{\ell})\btrl y) \\ &=
\one{\ell}\, (\r_{\beta}\r_{\alpha}\trl x) (\r^\beta\trl S(\two{\ell})) \, \one{(\r^\alpha\trl\three{\ell})\!\:\!} 
(\r_\gamma \trl y) \, (\r^\gamma \trl S((\two{\r^\alpha \trl \three{\ell})\!\:\!}))
\\ &=
\one{\ell}\, (\r_{\beta}\r_{\rho}\r_{\delta} \trl x) \,
(\r^\beta \trl  S(\two{\ell})) \, (\r^\rho \trl \three{\ell}) \, (\r_\gamma  \trl y) \, (\r^\gamma  \trl S(\r^{\delta} \trl \four{\ell}))
\\ &= \one{\ell}\, (\r_{\beta}\r_{\delta} \trl x) \, \big(\r^{\beta} \trl (S(\two{\ell})\three{\ell}) \big) \, (\r_\gamma \trl y) 
     \, (\r^\gamma  \trl S(\r^{\delta} \trl \four{\ell}))
\\ &=
\one{\ell}\, (\r_{\delta} \trl x) \, (\r_\gamma \trl y) \,
(\r^{\gamma}\r^{\delta} \trl S(\two{\ell}))
\end{align*}
where in the second and third lines we used the quasitriangular condition in \eqref{iiiR}.

The proof of property  \eqref{bralgAD}  follows recalling that $\tadr$ is an action, indeed
 \begin{align*}
(\one{\ell}\btrl (\r_\alpha\trl x))\btrl 
((\r^{\alpha}\trl\two{\ell})\btrl y)
&=
\big((\one{\ell}\btrl (\r_\alpha  \trl x)) \, \r^{\alpha} \trl \two{\ell}\big)\btrl y
\\ &=
\big(\one{\ell} (\r_{\beta}\r_{\alpha} \trl  x) \, (\r^\beta  \trl S(\two{\ell})) \, \r^{\alpha} \trl \three{\ell}\big)\btrl y
\\ &=
\big(\one{\ell} (\r_{\beta}  \trl x) \, \r^\beta  \trl (S(\two{\ell}) \three{\ell})\big)\btrl y
\\ &=
(\ell x)\btrl y
\\ &=
\ell\btrl (x\btrl y) 
\end{align*}
where for the third equality we used again the equality $(\Delta \ot \id)\r=\r_{13} \r_{12}$.
\end{proof}

\begin{lem}\label{lem:psiell}
The map 
\begin{align}\label{lem:psi}
\lie : (L, \trl_L) &\to (\Hom(L,L), \trl_{\Hom(L,L)})  \nn
\\
\ell &\mapsto \lie_\ell , \quad \lie_\ell (x):= \ell \tadr x =\one{\ell}(\r_\alpha \trl_L x) \, \r^\alpha \trl_L S_L(\two{\ell}) , \quad x \in L ,
\end{align}
is a morphim of $K$-modules.
\end{lem}
\begin{proof}
We show that
$\lie_{k \trl_L \ell}= k \trl_{\Hom(L,L)} \lie_\ell$, for each $k \in K$, $\ell \in L$. 
Let $x \in L $, then
\begin{align*}
\lie_{k \trl \ell}(x) 
&=\one{(k \trl \ell)}(\r_\alpha \trl x) \, \r^\alpha \trl S_L(\two{(k \trl \ell)})
\\
&=(\one{k} \trl \one{\ell})(\r_\alpha \trl x) \, \r^\alpha \trl S_L(\two{k} \trl \two{\ell})
\\
&=(\one{k} \trl \one{\ell})(\r_\alpha \trl x) \, (\r^\alpha \two{k}  \trl S_L( \two{\ell}))
\end{align*}
where for the last equality we used that $S_L$ is a morphism of $H$-modules.
On the other hand, recalling the definition of $\trl_{\Hom(L,L)} $ from \eqref{action-hom},
\begin{align*}
(k \trl  \lie_\ell)(x) 
&= \one{k} \trl \lie_\ell (S(\two{k})\trl x)
\\
&= \one{k} \trl (\one{\ell} \, (\r_\alpha \trl (S(\two{k}) \trl x)) \, (\r^\alpha \trl S_L(\two{\ell})) 
\\
&= (\one{k} \trl \one{\ell} )\, (\two{k} \r_\alpha S(\four{k}) \trl x)) \, (\three{k} \r^\alpha \trl S_L(\two{\ell})) 
\\
&= (\one{k} \trl \one{\ell} )\, (\r_\alpha \three{k} S(\four{k}) \trl x)) \, ( \r^\alpha \two{k}  \trl S_L(\two{\ell})),
\end{align*}
using the quasi-cocommutativity of $K$ for the last equality. Thus, the equality holds.
\end{proof}

\begin{prop}
Let $(\tK,\tad)$ be the braided Hopf algebra associated with
the quasitriangular Hopf algebra 
$(\K,\r)$ as is Example \ref{ex:braided}. Then,
the braided adjoint action \eqref{bLad} coincides with the adjoint action \eqref{ad}: $\tadr = \tad $,
That is for all $h,k \in \K$, 
$$
 \tone{h}(\r_\alpha \tad k) (\r^\alpha \tad \tS(\ttwo{h})) = \one{h} k S(\two{h}) .
$$
\end{prop}

\begin{proof}
From the explicit expressions of the braided coproduct and antipode in $\tK$, we can rewrite the braided adjoint action in the left hand side as
$$ h \tadr k = \one{h} \r_\nu S(\r_\mu)\one{\r_\alpha}  k S(\two{\r_\alpha}) \, \one{\r^\alpha}   \br_\sigma \bur S (\r^\mu  \two{h} \r^\nu ) \br^{\sigma} S(\two{\r^\alpha}) .
$$
By using the quasi triangularity of $\K$ in \eqref{iiiR}, and then the properties \eqref{ant-R} relating $\r$ to $\br$ via the altipode, this is written as
\begin{align*}
\one{h} \r_\nu S(\r_\mu)
\r_\gamma \r_\delta
& k S(\r_\alpha \r_\beta) \, \r^\alpha  \r^\gamma  \br_\sigma \bur S (\r^\mu  \two{h} \r^\nu ) \br^{\sigma} S(\r^\beta \r^\delta)
\\
& =
\one{h} \r_\nu \r_\mu \r_\gamma \br_\delta
k \r_\beta \ur  \r^\gamma  \br_\sigma \bur S(\r^\nu)  S(\two{h}) \r^\mu 
\br^{\sigma} \br^\delta\r^\beta 
\\
&=
\one{h} \r_\nu \r_\mu \br_\gamma \br_\delta
k \r_\beta S( \br^\gamma)  S(\r_\sigma)  S(\r^\nu)  S(\two{h}) \r^\mu 
\r^{\sigma} \br^\delta\r^\beta .
\end{align*}
For the last equality we  used 
$$
 \r_\gamma
\ot \ur  \r^\gamma  \br_\sigma \bur \ot \br^\sigma 
=
\br_\gamma
\ot S(\br^\gamma)  S(\r_\sigma)  \ot \r^\sigma 
$$
which is obtained from the expression \eqref{Sinv} for the inverse of the antipode.
Next the quantum Yang-Baxter equation (on the indices $\nu, \mu,\sigma$) gives
$$
\one{h} \r_\mu \r_\sigma \br_\gamma \br_\delta
k \r_\beta S( \br^\gamma)  S(\r_\nu \r^\sigma)  S(\two{h}) \r^\nu 
\r^{\mu} \br^\delta\r^\beta 
$$
and we finally obtain
\begin{align*}
\one{h} \r_\mu  \br_\delta
k \r_\beta  S(\r_\nu)  S(\two{h}) \r^\nu 
\r^{\mu} \br^\delta\r^\beta 
& =
\one{h} 
k \r_\beta  S(\r_\nu)  S(\two{h}) \r^\nu 
\r^\beta 
\\
& =
\one{h} 
k   S(\r_\nu \br_\beta)  S(\two{h}) \r^\nu 
\br^\beta 
\\
& =
\one{h} 
k     S(\two{h})  
\end{align*}
thus establishing that $h \tadr k = h \tad k$.
\end{proof}

\subsection{Dual structures}\label{DUALS} ~\\
For later use we briefly recall the dual notion of
coquasitriangular Hopf algebra and that of
associated  braided Hopf algebra, that was used in  
\cite{pgc}.

 \begin{defi}\label{defcoqt}
  A bialgebra $H$ is called \textbf{coquasitriangular} (or dual
quasitriangular) if it is endowed with a linear form $R:H \ot H \to \bbK$ such that:

\noindent
(i) $R$ is  invertible for the convolution product in $H\otimes H$, with inverse denoted $\bar R$ ,  \\ 
(ii) $m_{op} = R * m * \bar R$ , 
\\
(iii) $R\circ (m \ot \id)=R_{13}* R_{23}$ and $R\circ
(\id \ot m)=R_{13}* R_{12}$, \\ where $R_{12}=R\ot \varepsilon: H\ot
H\ot H\to \bbK$  and similarly
 for $R_{13}$ and $R_{23}$.

The linear form $R$ is called a \textbf{universal $R$-form} of $H$. 
A coquasitriangular bialgebra $(H,R)$ is called \textbf{cotriangular} if $R= \bar{R}_{21}$.
\end{defi}

As for the quasitriangular case, 
tensor products of comodule algebras are comodule algebras and 
tensor products of comodule algebra maps are again comodule
algebra maps:
\begin{prop}
Let $(H,R)$ be a coquasitriangular bialgebra. Let  $(A,\delta^A), (C,\delta^C) \in \A^H$
be right $H$-comodule algebras. 
Then the $H$-comodule $A \ot C$, with tensor product coaction 
as in \eqref{deltaVW}, is a right $H$-comodule algebra when endowed with the product
\beq\label{tpr'}
(a \ot c) \tpr (a' \ot c'):=
a\zero{a'} \ot \zero{c} c' ~\Ru{\one{c}}{\one{a'}} \, .
\eeq
Moreover, when $\phi: A\to E$ and $\psi: C\to F$ 
are morphisms of $H$-comodule algebras, then so is 
the map $\phi\otimes\psi : A \ot C \to E \ot F$, $a\otimes c\mapsto
\phi(a)\otimes \psi(c)$,
where $A\otimes C$ and $E\otimes F$ are endowed with the $\tpr$-products in \eqref{tpr'}.
\end{prop}
The $H$-comodule algebra $(A \ot C, \tpr)$ is the
{\bf{braided tensor product algebra}} of $A$ and $C$. We
denote it by $A \bot  C$, and write $a \bot c\in A \bot
 C$ for $a\in A$, $c\in C$ and by
 $\phi\bot\psi:=\phi\otimes\psi: A\bot C\to E\bot F$ the 
 tensor product of two morphisms.
With the tensor product $\bot$ the category $(\RA,\bot)$ of 
$H$-comodule algebras is a 
monoidal category.

The definition of a braided Hopf algebra associated with a
coquasitriangular Hopf algebra $(H,R)$ is the same as that in Definition \ref{def:bbK} 
for the quasitriangular case, with the braided monoidal category of comodules for $(H,R)$ replacing that of modules for $(K,\r)$.

In particular for any Hopf algebra $H$, the data  $(H, \Delta_H,\cun_H, \Ad)$ is an $H$-comodule coalgebra, with the right adjoint coaction  
$\Ad: H\to H\ot H, ~h \mapsto \two{h} \ot S(\one{h})\three{h}$. 
When $(H, R)$ is  coquasitriangular $\Ad$ is an algebra map for the 
`braided' product in $H$ given by
\beq\label{hpr}
h \hpr k := \two{h}\two{k} 
R\big( S(\one{h})\three{h} \ot {S(\one{k})} \big)\; , \quad h,k \in H \, .
\eeq
Then $\bh:= (H, \hpr, \eta_H,  \Delta_H, \cun_H, \Ad)$ is a
braided bialgebra associated with $(H,R)$. Furthermore $\bh$
is a braided Hopf algebra with antipode $\bs$ defined  by
\beq\label{bs}
\bs(h):= S(\two{h}) R \big({S^2(\three{h})S(\one{h})} \ot {\four{h}})
\; , \quad h\in H \, .
\eeq

\section{Twisting braided Hopf algebras}\label{TbHa}
We recall some basic definitions and properties of the theory of
Drinfel'd twists \cite{drin1,drin2}, see also \cite{drin}.

\begin{defi}
Let $\K$ be a bialgebra (or Hopf algebra). A {\bf twist} for $\K$  is an invertible  element 
$\F \in \K \ot \K$ which is unital, 
$
(\cun \ot \id) (\F)= 1= (\id \ot \cun)(\F) 
$,
and satisfies 
\beq\label{twist}
(\F \ot 1)[(\Delta \ot \id)(\F)]= (1 \ot \F)[(\id \ot \Delta) (\F)] ~,
\eeq
referred to as the twist condition.
\end{defi}
If $\bF$ denotes the inverse of the twist, $\F \bF=1 \ot 1 =  \bF \F$, 
the condition \eqref{twist} can be written equivalently as 
\beq\label{btwist}
[(\Delta \ot \id)(\bF)](\bF \ot 1)=[(\id \ot \Delta) (\bF)]  (1 \ot \bF)~.
\eeq
We use the notation $\F=\F^{ \alpha} \ot \F_\alpha$ and   $\bF=: \bF^{\alpha} \ot
\bF_{\alpha}$ with an implicit summation understood, for the twist and its inverse. 
The identities \eqref{twist}, \eqref{btwist},  are then rewritten respectively as
\beq\label{twist-F}
\F^\beta \one{\F^{\alpha}} \ot \F_\beta \two{\F^\alpha} \ot \F_\alpha
= 
\F^\alpha \ot  \F^\beta \one{\F_{\alpha}} \ot \F_\beta \two{\F_\alpha} 
\eeq
\beq\label{twist-bF}
\one{\bF^{\alpha}}  \bF^\beta \ot  \two{\bF^\alpha}\bF_{\beta} \ot \bF_\alpha
= 
\bF^\alpha \ot   \one{\bF_{\alpha}}\bF^\beta \ot \two{\bF_\alpha} \bF_\beta  \; .
\eeq

\begin{rem}\label{rem:topological}
  When $K$ is not finite
  dimensional over $\bbK$, the twist $\F$ may not necessarily be an
  element of $K\otimes K$, but rather it belongs to a topological
  completion of the tensor product algebra. In the
  examples of this paper we avoid this problem either by diagonalizing
  $\F$ or by considering representations of $\F$ involving only a
  finite number of addends in $\F=\F^\alpha\otimes F_\alpha$. 
\end{rem}
\begin{ex}
The $\r$-matrix  $\r$ of a quasitriangular bialgebra $\K$ is a twist for $\K$. Condition \eqref{twist} follows from the quasitriangular condition \eqref{iiiR} and the quantum Yang-Baxter equation:
$
(\r \ot 1) [(\Delta \ot \id) \r]=\r_{12 }\r_{13} \r_{23}= \r_{23}\r_{13} \r_{12} =(1 \ot \r)[ (\id \ot \Delta)
$.
\end{ex}

When $\K$ has a twist, $\K$ can be endowed with a second bialgebra structure which is obtained by deforming its coproduct and leaving its counit and multiplication unchanged:
\begin{prop}\label{prop:twist}
Let $\F=\F^{ \alpha} \ot \F_\alpha$  be a twist on a bialgebra  $(\K, m, \eta, \Delta, \cun)$. Then
the algebra $(\K, m, \eta)$ with coproduct 
\beq\label{cop-twist}
\Delta_\F(k):= \F \Delta(k) \bF= \F^{\alpha} \one{k} \bF^{\beta} \ot \F_\alpha \two{k} \bF_{\beta}~, \qquad k \in \K
\eeq
and counit $\cun$  is a bialgebra. If in addition $\K$ is a Hopf algebra, then the
twisted bialgebra $\KF:=(\K, m, \eta, \Delta_\F, \cun)$
is a Hopf algebra with antipode $$S_\F(k):=\uf S(k) \buf ,
$$ where $\uf$ is the invertible element
$\uf:=\F^{\alpha} S(\F_\alpha)$ and $\buf= S(\bF^{\alpha})\bF_{\alpha}$ its inverse.
\end{prop}
We use the notation $\Delta_\F (k)=: \fone{k} \ot \ftwo{k}$ for the coproduct in $\KF$ to distinguish it from the original coproduct $\Delta(k)= \one{k} \ot \two{k}$ in $\K$.

Let us recall that the inverse $\bF$ of the twist $\F$ is a twist for the twisted bialgebra $\KF$, with the twist condition for $\bF \in \KF \ot \KF$ coinciding with \eqref{twist-bF}.  Clearly, using $\bF$ to deform the coproduct in $\KF$ leads back to the initial bialgebra $(\KF)_{\bF}=\K$.

For later use, we next recall some results which are dual to those related to deformations of bialgebras via 2-cocycles \cite{doi}, and which were addressed in \cite[\S 4]{pgc}. 

Let $\K$ be a  bialgebra (or Hopf algebra) with a twist $\F \in \K \ot \K$ and $\KF$ the resulting twisted 
bialgebra, as above. 
Any  $\K$-module $V$ with left action $\trl_V: \K \ot V \to V$, is also a $\KF$-module
with the same linear map  $\trl_V$, now thought as a map 
$\trl_V: \KF \ot V \to V $. 
Indeed, the module conditions in \eqref{eqn:Hmodule}
only involve the algebra structure of $\K$, and  the twisted bialgebra
$\KF$ coincides with $\K$ as algebra. 
When thinking of $V$ as a $\KF$-module we denote it by $V_\F$, with action $\trl_{V_F}$. 
Moreover, any $\K$-module morphism $\psi : V\to W$ 
can be thought as a morphism $\psi_F  : V_\F \to W_\F$  
since, as stated, the action of $\K$ on any $V,W$ coincides with the action of $\KF$ on $V_\F, W_\F$.
This amounts to say that there is an equivalence of category 
$\Gamma : \KM \to\KFM$ with the functor $\Gamma$ just the identity both on objects and on morphisms:  
$\Gamma(V):=V_\F$ and  $\Gamma(\psi):=\psi_F$. The use of $\bF$ as a twist for $\KF$,  which `twists back' $\KF$ to $\K$, inverts the construction going from $\KFM$ to $\KM$.

With the monoidal structure one needs the deformed coproduct \eqref{cop-twist}. 
Explicitly, for $\KF$ modules $(V_\F, \trl_{V_\F})$ and $(W_\F,\trl_{W_\F})$,  
the tensor product action of $\KF$ on $V_\F \otF W_\F \simeq V \ot W$ (as linear spaces) 
is given by 
\begin{align}\label{trlVW-F}
\trl_{V_\F \otF W_\F} : \KF \ot (V_\F\otF W_\F) &
                                                        \longrightarrow  V_\F \otF W_\F ~,\\
k \ot  v\otF w & \longmapsto  (\fone{k} \trl_{V_\F}  v) \otF (\ftwo{k} \trl_{W_\F} w) ~. \nn
\end{align}

We denote by $(\KFM,\otF)$ the monoidal category of left $\KF$-modules. We have then: 
\begin{prop}\label{pro:funct}
The functor $\Gamma : \KM \to\KFM$ together with the natural isomorphism   $\varphi : \otF \circ
(\Gamma\times\Gamma)\Rightarrow \Gamma\circ \otimes$ given  by the isomorphism of $\KF$-modules
\begin{align}\label{nt-mod}
\varphi_{V,W}: V_\F \otF W_\F &\longrightarrow  (V \ot W)_\F  
\\
v \otF w &\longmapsto  (\bF^{\alpha} \trl_{V_\F} v)\ot (\bF_{\alpha} \trl_{W_\F} w) \, ,\nn
\end{align}
is an equivalence of monoidal categories $(\KM, \ot)$ and 
$(\KFM, \otF)$.  
\end{prop} 
As a consequence of the twist condition \eqref{twist-bF} the isomorphisms $\varphi_{-,-}$ satisfy 
\beq\label{prop-ass-phi}
\varphi_{V,W \ot_\F Z} \circ (\id_{V_F} \ot_\F \varphi_{W,Z})
=\varphi_{V \ot_\F W , Z} \circ (\varphi_{V,W} \ot_\F \id_{Z_F})
\eeq
for every $\KF$-modules $V_\F,W_\F,Z_\F$.
 
 There is also an equivalence between the category $\KA$ of left $\K$-module algebras 
and the corresponding category $\KFA$ of $\KF$-module algebras
\beq\label{functGammaA}
\Gamma : \KA\to \KFA \, , \quad {(A,m_A,\eta_A, \trl_A)
  \mapsto (A_\F ,m_{A_\F},\eta_{A_\F}, \trl_{A_\F})} \; .
\eeq
It is no longer the identity on objects. 
Given $A\in \KA$ with multiplication $m_A$ and unit $\eta_A$, in order for the action $\trl_{A_\F}$ to be an algebra map one has to define a new product on the $\K$-module
$A_\F=\Gamma(A)$. The new algebra
structure $m_{A_\F} , \eta_{A_\F}$ on
$A_\F\in \KFA$ is defined by using the natural
isomorphisms $\varphi$ in \eqref{nt-mod}:
$$
m_{A_\F}:=\Gamma(m_A) \circ \varphi_{A,A} \, , \quad 
\eta_{A_\F}:=\Gamma(\eta_A)
$$
Explicitly, the unit is unchanged,  while the product is deformed to
\begin{align}\label{rmod-twist} 
 m_{A_\F} : A_\F \otF A_\F & \,\longrightarrow A_\F \nn \\
a\otF a^\prime  & \,\longmapsto a \dotF a':= (\bF^{\alpha} \trl_A a)\, (\bF_{\alpha} \trl_A a') 
 \end{align}
The functor $\Gamma$ is the identity on morphism: 
for any algebra map $\psi : A\to A^\prime$ one shows  that 
$\Gamma(\psi)=\psi_F : A_\F \to A^\prime_\F$ is an algebra map for
 the deformed products. 
 
Similarly to above, there is also an equivalence between the category 
$\KC$ of left $\K$-module coalgebras and that of $\KF$-module coalgebras
\beq\label{functGammaC}
\Gamma : \KC \to\KFC \, , \quad {(C,\Delta_C,\cun_C, \trl_C) 
\mapsto (C_\F ,\Delta_{C_\F},\cun_{C_\F},\trl_{C_\F})} \; .
\eeq
Again, $\Gamma$ acts as the identity on morphisms, but not on objects. 
Each $\K$-module coalgebra $C$ with costructures $\Delta_C,\cun_C$
is mapped to the $\KF$-module coalgebra $C_\F=\Gamma(C)$ with costructures 
$\Delta_{C_\F},\cun_{C_\F}$ 
defined by 
$$
\Delta_{C_\F} := \varphi_{C,C}^{-1} \circ \Gamma(\Delta_C) \, , \quad
\cun_{C_\F}:=\Gamma(\cun_C) .
$$
Explicitly, while the counit does not change, the
deformed coproduct is 
\beq\label{cc-twist} 
 \Delta_{C_\F} : C_\F\longrightarrow C_\F\otF C_\F ~,~~c \longmapsto  \fone{c} \ot \ftwo{c}:=
(\F^{\alpha} \trl_C \one{c})  \ot (\F_\alpha \trl_C \two{c}) \,.
\eeq

\subsection{$\K$-modules of linear maps} 
We  study twist deformations of linear maps of $\K$-modules 
for later applications   in \S  \ref{sec:Fbd} to braided derivations. 

Consider the $\K$-module algebra
   $(\Hom(V,V),\circ)$ of linear maps of a $\K$-module  $V$, with action $\trl_{\Hom(V,V)}$ as in \eqref{action-hom}.
On the one hand, we obtain a $\K_\F$-module algebra $(\Hom_\F(V,V),\circ_\F)$ out of $(\Hom(V,V),\circ)$ by changing the multiplication 
$\circ$ as in \eqref{rmod-twist}:
\beq\label{circF}
\psi \circ_\F \phi = (\bF^{\alpha} \trl_{\Hom(V,V)} \psi)\circ (\bF_{\alpha} \trl_{\Hom(V,V)} \phi) \, .
\eeq  
The $\K_\F$-action $\trl_{\Hom_\F(V,V)}$ 
 coincides with  $\trl_{\Hom(V,V)}$  as linear map.
On the other hand, there is  the $\K_\F$-module algebra $(\Hom(V_\F,V_\F),\circ)$ of linear maps  of 
the $\K_\F$-module  $V_\F$ with action
\begin{align}\label{action-homF}
\trl_{\Hom(V_\F,V_\F)}:  K_\F \ot \Hom(V_\F,V_\F) &\to \Hom(V_\F,V_\F) 
\nn\\
k \ot \psi & \mapsto  k \trl_{\Hom(V_\F,V_\F)} \psi : \; a \mapsto \fone{k} \trl_{V_\F} \psi(S_\F(\ftwo{k})\trl_{V_\F} a) .
\end{align}
These two module algebras are isomorphic  (cf. \cite[Th.~4.7]{AS14}):
\begin{prop}\label{prop:Diso}
The 
$\K_\F$-module algebras $(\Hom_\F(V,V),\circ_\F)$  and $(\Hom(V_\F,V_\F),\circ)$ are isomorphic via the map 
\begin{align}\label{mappaD}
\dd :  \Hom_\F(V,V)  & \to \Hom(V_\F,V_\F) 
\nn
\\
\psi & \mapsto \dd(\psi): v \mapsto (\bF^\alpha \trl_{\Hom_\F(V,V)} \psi)(\bF_\alpha \trl_V v) 
\end{align}
with 
 inverse
\begin{align}\label{mappaDinv}
\dd^{-1}:\Hom(V_\F,V_\F) & \to \Hom_\F(V,V) \nn \\
\psi  & \mapsto 
\dd^{-1}(\psi): v \mapsto (\F^\alpha  \trl_{\Hom(V_\F,V_\F)} \psi)(\F_\alpha \trl_{V_\F}  v) \; .
\end{align}
\end{prop}
\begin{proof}
We first observe that from the twist condition and the definition of  the action on endomorphims, we can also write
\begin{align*} 
\dd(\psi)(v) &= \F^\alpha \trl_V \psi (S(\F_\alpha) \buf \trl_V v)\; ; \; \nn \\ 
\dd^{-1}(\psi)(v) &= \bF^\alpha \trl_{V_\F} \psi (S_\F(\bF_\alpha) \uf \trl_{V_\F} v).
\end{align*}
Then it is easy to see that the two maps are one the inverse of the other (also recalling that the actions $\trl_{V}$ and $\trl_{V_\F}$ coincides as linear maps).
We show the $\K_\F$-equivariance $k \trl_{\Hom(V_\F,V_\F)} \dd (\psi)= \dd( k \trl_{\Hom_\F(V,V)} \psi)$:
\begin{align*}
(k \trl \dd (\psi))(v) &=  \fone{k} \trl \dd(\psi) (S_\F (\ftwo{k}) \trl v)
\\
&= \F^\alpha \one{k} \bF^\beta \trl \dd(\psi) (\uf  S(\F_\alpha \two{k} \bF_\beta)  \buf \trl v)
\\
&= \F^\alpha \one{k} \bF^\beta \F^\gamma \trl \psi (S(\F_\gamma)\buf  \uf  S(\F_\alpha \two{k} \bF_\beta)  \buf \trl v)
\\
&= \F^\alpha \one{k}  \trl \psi ( S(\F_\alpha \two{k} )  \buf \trl v)
\\
&= \F^\alpha \trl \big( (k \trl \psi) ( S(\F_\alpha  )  \buf \trl v)\big)
\\
&= \dd( k \trl \psi)(v) \; .
\end{align*} 
Next we show that  $\dd(\psi \circ_\F \phi)= \dd(\psi) \circ \dd(\phi)$. 
For that we compute
\begin{align*}
(\psi \circ_\F \phi) (v) &=
\big((\bF^\alpha \trl \psi)\circ (\bF_\alpha \trl \phi)\big) (v)
\\
&= (\bF^\alpha \trl \psi)(\one{\bF_\alpha} \trl \phi (S( \two{\bF_\alpha}) \trl v)
\\
&= \one{\bF^\alpha} \trl \psi( S(\two{\bF^\alpha})\one{\bF_\alpha} \trl \phi (S( \two{\bF_\alpha}) \trl v)
\\
&= \bF^\alpha \F^\gamma \trl \psi( S(\one{\bF_\alpha} \bF^\beta \F_\gamma)\two{\bF_\alpha} \one{\bF_\beta} \trl \phi (S( \three{\bF_\alpha} \two{\bF_\beta}) \trl v)
\\
&= \bF^\alpha \F^\gamma \trl \psi( S(\F_\gamma) S( \bF^\beta )\one{\bF_\beta} \trl \phi (S( \bF_\alpha \two{\bF_\beta}) \trl v)
\\
&= \bF^\alpha \F^\gamma \trl \psi( S(\F_\gamma) \buf  \F^\beta   \trl \phi (S( \bF_\alpha \F_\beta) \trl v)
\\
&= \bF^\alpha  \trl \dd(\psi)(   \F^\beta   \trl \phi (S( \F_\beta) S( \bF_\alpha) \trl v)
\\
&= \bF^\alpha  \trl \dd(\psi)\big(  \dd( \phi) (\uf S( \bF_\alpha) \trl v) \big)
\\
&= \dd^{-1} (\dd(\psi) \circ \dd(\phi)) (v) 
\end{align*}
thus showing that   $\dd$ is an algebra morphism.
\end{proof}

\subsection{Quasitriangular  bialgebras} 
We  consider the case when the twist comes from a 
quasitriangular bialgebra $(\K,\r)$. The twisted bialgebra $\K_F$ is still quasitriangular. Moreover, the isomorphisms $\varphi$ 
in \eqref{nt-mod} for module algebras are algebra maps (Proposition \ref{prop:funct-alg}); for this we need the following result. 
\begin{prop}\label{main-lem}
Let $(K,\r)$ be a quasitriangular bialgebra and $\F \in  \K \ot \K $ a twist for $\K$. 
Then the following identity holds in $\K^{\ot 4}$:
\begin{multline} \label{lhs=rhs-m}
[(\Delta \ot \Delta)\F](1 \ot \r_{21}  \ot 1)[(\id \ot  \tau \ot \id)(\Delta \ot \Delta)\bF] \\
=
(\bF \ot \bF)(1 \ot \rF_{21}  \ot 1)[(\id \ot  \tau \ot \id)(\F \ot \F)]\; ,
\end{multline} 
where $\rF_{21} = \F \, \r_{21} \, \bF_{21}$. That is,
\begin{multline}\label{lhs=rhs}
\one{\F^{\gamma}} \one{\bF^{\alpha}} \ot \two{\F^{\gamma}} \r_\beta \one{\bF_{\alpha}}  \bot_\F \one{\F_{\gamma}}\r^\beta\two{\bF^{\alpha}} \ot
 \two{\F_{\gamma}}\two{\bF_{\alpha}} \\
=
 \bF^{\beta} \F^{\gamma} \ot \bF_{\beta} \rF_\alpha \F^{\delta}  \bot_\F  \bF^{\mu}  
 \rF^\alpha \F_{\gamma}  \ot  \bF_{\mu}\F_{\delta} \, .
\end{multline}
\end{prop}
\begin{proof}
Firstly we observe that the twist condition \eqref{twist} for $\F$  gives
$$
[(\Delta \ot \Delta)\F]= (\bF \ot 1 \ot 1)[(1 \ot (\id \ot \Delta) \F][(\id \ot \Delta^2)\F]
$$
and, similarly, the twist condition \eqref{btwist} written for the inverse $\bF$ of $\F$ gives
$$
[(\Delta \ot \Delta)\bF]= [(\id \ot \Delta^2)\bF][(1 \ot (\id \ot \Delta) \bF](\F \ot 1 \ot 1) \; .
$$
Then, identity \eqref{lhs=rhs} is equivalent to 
\begin{align*}
 [(1 \ot (\id \ot \Delta) \F][(\id \ot \Delta^2)\F](1 \ot \r_{21}  \ot 1)(\id \ot  \tau \ot \id)\big[ [(\id \ot \Delta^2)\bF][1 \ot (\id \ot \Delta) \bF ]\big]
\\
=(1 \ot 1 \ot \bF)(1 \ot \F \, \r_{21} \, \bF_{21}   \ot 1)[(\id \ot  \tau \ot \id)(1 \ot 1 \ot \F)] \; .
\end{align*}
We multiply both sides by $(1 \ot 1 \ot \F)$ and use again the twist condition \eqref{twist}
in the left hand side to simplify
$(1 \ot \F \ot 1)$ and obtain
\begin{align*}
 [(1 \ot (\Delta \ot \id) \F][(\id \ot \Delta^2)\F](1 \ot \r_{21}  \ot 1)(\id \ot  \tau \ot \id)\big[ [(\id \ot \Delta^2)\bF][1 \ot (\id \ot \Delta) \bF ]\big]
\\
=(1 \ot \r_{21} \, \bF_{21}   \ot 1)[(\id \ot  \tau \ot \id)(1 \ot 1 \ot \F)]\; .
\end{align*}
The quasitriangular condition $\Delta^{cop} = \r  \Delta \br$ gives
$$
(\r_{21}  \ot 1)(\tau \ot \id)\Delta^2= (\r_{21} \Delta^{cop} \ot \id) \Delta=  \Delta^2 (\r_{21} \ot 1) \; ,
$$
which we can use to simplify the left hand side to
$$
 [(1 \ot (\Delta \ot \id) \F](1 \ot \r_{21}  \ot 1)(\id \ot  \tau \ot \id) [1 \ot (\id \ot \Delta) \bF ]
$$
and further to
$$
(1 \ot \r_{21}  \ot 1) [(1 \ot (\Delta^{cop} \ot \id) \F](\id \ot  \tau \ot \id) [1 \ot (\id \ot \Delta) \bF ]
$$ having used once again the quasitriangularity.
Summarizing, we are left to show that
\begin{multline*}
 [(1 \ot (\Delta^{cop} \ot \id) \F](\id \ot  \tau \ot \id) [1 \ot (\id \ot \Delta) \bF ] \\
=
(1 \ot \bF_{21}   \ot 1)[(\id \ot  \tau \ot \id)(1 \ot 1 \ot \F)]
\end{multline*}
or, by (appropriately) simplifying the operator $(\id \ot  \tau \ot \id)$, that
$$
 [(1 \ot (\Delta\ot \id) \F] [1 \ot (\id \ot \Delta) \bF ]
=
(1 \ot \bF  \ot 1)(1 \ot 1 \ot \F) \; .
$$
The identity $[(\Delta\ot \id) \F] [(\id \ot \Delta) \bF ] = (\bF \ot 1)(1 \ot \F)$ follows from the twist condition \eqref{twist}, hence the identity \eqref{lhs=rhs} in the Proposition is proven.
\end{proof}

As mentioned, the twisted bialgebra $\KF$ is quasitriangular as well. 
\begin{prop}\label{prop:KFRF}
If $(\K,\r)$ is a quasitriangular bialgebra and $\F \in  \K \ot \K $ is a twist for $\K$, then the twisted bialgebra $\KF$  with 
twisted coproduct $\Delta_\F$ (see Proposition \ref{prop:twist}) is quasitriangular with $\r$-matrix 
\beq\label{urm-def}
\rF := \F_{21} \, \r \, \bF = \F_\alpha \r^\beta \bF^{\gamma} \ot \F^\alpha \r_\beta \bF_{\gamma} 
\eeq
and inverse   $\brF :=\F \, \br \, \bF_{21} = \F^\alpha \br^\beta \bF_{\gamma} \ot \F_\alpha \br_\beta \bF^{\gamma}$. If  $(\K,\r)$ is  triangular, so is  $(\KF,\rF)$.
Also, if $\K$ is a quasitriangular Hopf algebra, then so is $\KF$ with twisted antipode $S_\F$.
\end{prop}

The category $\KFM$ of left $\KF$-modules  is a braided monoidal category with braiding
 \beq\label{braidingF}
\Psi^{\rF}_{V,W}: V \ot W  \longrightarrow W \ot V ~, \quad v\ot w \longmapsto  
(\rF_{\alpha} \trl_W w) \ot (\rF^{\alpha} \trl_V v)~,
\eeq  
for any pair of left $\KF$-modules $(V, \trl_V), (W, \trl_W)$.
\\

\begin{rem}
The matrix $\rF$ has been used in the Proposition \ref{main-lem}. 
By inspection one sees that the identity \eqref{lhs=rhs-m} can be written as:
\beq
1 \ot \Psi^{\rF}_{\KF, \KF}  \ot 1 = (\F \ot \F) [(\Delta \ot \Delta)\F][(1 \ot \Psi^{\r}_{\K,\K}  \ot 1)(\Delta \ot \Delta)\bF] (\bF \ot \bF) \, .
\eeq
Here $\K$ (respectively $\KF$) is seen as a left $\K$-module (respectively left $\KF$-module).  
\end{rem}

For quasitriangular bialgebras we compare the monoidal categories of module algebras $(\KA, \bot)$ and $(\KFA, \bot_\F)$. Proposition \ref{main-lem} and Proposition \ref{pro:funct} imply they are equivalent.

\begin{prop}\label{prop:funct-alg} 
Let $(\K,\r)$ be a quasitriangular bialgebra and $\F\in \K \ot \K$ a twist. 
There is an equivalence of monoidal categories between $(\KA, \bot)$ and $(\KFA, \bot_\F)$ 
given by the functor $\Gamma : \KA \to \KFA$ {in \eqref{functGammaA}} and the isomorphisms in $\KFA$
\begin{align} \label{nt-alg-br}
\varphi_{A,C}: A_\F \bot_\F C_\F \longrightarrow  (A \bot C)_\F  ~,
\quad
a \bot_\F c \longmapsto   (\bF^{\alpha} \trl_{A_\F} a)\bot (\bF_{\alpha} \trl_{C_\F} c)  ~,\nn
\end{align}
with 
$A_\F \bot_\F C_\F$  the braided tensor product of the algebras $A_\F$ and $ C_\F$, and 
$(A \bot C)_\F$ the image via $\Gamma$ of the braided tensor product of the algebras $A$ and $C$. 
\end{prop}
\begin{proof} 
This is really a corollary of Proposition \eqref{main-lem}.
We need to show that the
isomorphisms $\varphi_{A,C}$ in $\KFM$, or equivalently their inverses
$$
\varphi_{A,C}^{-1}:   (A \bot C)_\F  \longrightarrow A_\F \bot_\F C_\F 
\quad
a \bot c \longmapsto   (\F^{\alpha} \trl_{A_\F} a)\bot_\F (\F_{\alpha} \trl_{C_\F} c)  ~,\nn
$$
  are algebra maps:
\beq\label{phi-am} m_{A_\F \bot_\F C_\F} \circ (\varphi_{A,C}^{-1} \ot \varphi_{A,C}^{-1})= \varphi_{A,C}^{-1} \circ  \,m_{(A \bot C)_\F} \; .
\eeq
To lighten notations, we omit the subscripts in the actions.
The product $m_{A_\F \bot_\F C_\F}$ in the l.h.s. is the product in the braided
 tensor product algebra $A_\F \bot_\F C_\F$, defined as in \eqref{tpr}, but now with respect to the $\r$-matrix  $\rF$: 
\begin{align*}
m_{A_\F \bot_\F C_\F}\left( (a \bot_\F c) \ot (a' \bot_\F c')\right)
&= a\dotF (\rF_\alpha \trl a') \bot_\F (\rF^\alpha \trl c) \dotF c'\, ,
\\
&= (\bF^{\beta} \trl a)(\bF_{\beta} \rF_\alpha \trl a') \bot_\F  (\bF^{\gamma}  \rF^\alpha \trl c)  (\bF_{\gamma} \trl c')\, .
\end{align*}
The product $\,m_{(A \bot C)_\F}$ in the r.h.s. is the deformation  (as in \eqref{rmod-twist}) of the product $\tpr$ in the tensor
product algebra $A \bot C$ in  \eqref{tpr} :
\begin{align*}
\,m_{(A \bot C)_\F} \left( (a \bot c) \ot (a' \bot c')\right)
&=
\big( \bF^{\alpha} \trl(a \bot c) \big)\tpr \big( \bF_{\alpha} \trl(a' \bot c') \big)
\\
&=\big( (\one{\bF^{\alpha}} \trl a) \bot (\two{\bF^{\alpha}} \trl c) \big)\tpr \big( 
(\one{\bF_{\alpha}} \trl a') \bot (\two{\bF_{\alpha}} \trl c') \big)
\\
&= (\one{\bF^{\alpha}} \trl a)(\r_\beta \one{\bF_{\alpha}} \trl a') \bot (\r^\beta\two{\bF^{\alpha}} \trl c) 
 (\two{\bF_{\alpha}} \trl c') \, .
\end{align*}
Then, the l.h.s. of \eqref{phi-am} on the generic element $(a \bot a') \ot (c \bot c')$   gives
\begin{multline*}
m_{A_\F \bot_\F C_\F}  (\varphi_{A,C}^{-1} (a \bot c) \ot \varphi_{A,C}^{-1}(a' \bot c')) \\
=
 (\bF^{\beta} \F^{\gamma}\trl a)(\bF_{\beta} \rF_\alpha \F^{\delta} \trl a') \bot_\F  (\bF^{\mu}  \rF^\alpha \F_{\gamma} \trl c)  (\bF_{\mu}\F_{\delta} \trl c') \,.
\end{multline*}
Similarly, the r.h.s. of equation \eqref{phi-am}  gives
\begin{multline*}
\varphi_{A,C}^{-1} \circ  \,m_{(A \bot C)_\F} ((a \bot a') \ot (c \bot c')) \\
=(\one{\F^{\gamma}} \one{\bF^{\alpha}} \trl a)(\two{\F^{\gamma}} \r_\beta \one{\bF_{\alpha}} \trl a') \bot_\F (\one{\F_{\gamma}}\r^\beta\two{\bF^{\alpha}} \trl c) 
 (\two{\F_{\gamma}}\two{\bF_{\alpha}} \trl c') \; .
\end{multline*}
 Equation \eqref{phi-am}  then follows from \eqref{lhs=rhs} in Proposition \ref{main-lem}.
\end{proof}

\subsection{Braided bialgebras}\label{sec:TBB}  
Next consider a braided bialgebra $(L, m_L, \eta_L, \Delta_L,\cun_L, \trl_L)$, 
associated with the quasitriangular  bialgebra $(\K,\r)$.

On the one hand, $(L, m_L, \eta_L, \trl_L)$ is a $\K$-module 
algebra that we can deform using the functor in \eqref{functGammaA} into the 
$\KF$-module algebra $({L_\F}, m_{L_\F}, \eta_{L_\F}, \trl_{L_\F})$; this differs from $L$ only for the product given by
$\xi \dotF \xi'= (\bF^{\alpha} \trl \xi)\, (\bF_{\alpha} \trl \xi')$, for all $\xi,\xi' \in L_\F$, as in \eqref{rmod-twist}.

On the other hand, $(L, \Delta_L,\cun_L, \trl_L)$ is a 
$\K$-module 
coalgebra 
 that we  deform using the functor in \eqref{functGammaC} so to obtain the 
$\KF$-module coalgebra $({L_\F}, \Delta_{L_\F}, \cun_{L_\F}, \trl_{L_\F})$  with twisted coproduct
$
 \Delta_{L_\F} (\xi)=  \fone{\xi} \ot \ftwo{\xi}:=
(\F^{\alpha} \trl \one{\xi})  \ot (\F_\alpha \trl \two{\xi}) $
for each $\xi \in L_\F$,
as in \eqref{cc-twist}. 

We have the following result (cf. \cite[Prop. 4.7]{pgc} for a dual result):
\begin{prop}\label{prop:td-bb}
The twist deformation  of  $(L, m_L, \eta_L, \trl_L)$ 
as a $\K$-module algebra and of  
$(L, \Delta_L,\cun_L, \trl_L)$  as a $\K$-module coalgebra
is a braided bialgebra 
associated with the twisted Hopf algebra $\KF$. That is, $({L_\F}, m_{L_\F}, \eta_{L_\F}, \Delta_{L_\F},\cun_{L_\F}, \trl_{L_\F})$ is a bialgebra in the  braided monoidal category $(\KFM, \otF,
\Psi^{\rF})$ of $\KF$-modules.
If $L$ is a braided Hopf algebra, then $L_\F$ is a braided Hopf algebra, with
antipode $S_{L_\F}=\Gamma(S_L)$ (equal to $S_L$ as a linear map).
\end{prop}
\begin{proof}
We have to show that $\Delta_{L_\F}:L_\F \to L_\F\bot_\F L_\F$ and  $\cun_{L_\F}: L_\F \to \bbK$ 
are  algebra maps:
\beq
\cun_{L_\F}\circ m_{L_\F}  = m_\bbK \circ (\cun_{L_\F} \ot \cun_{L_\F}) \, , \quad
\Delta_{L_\F}\circ m_{L_\F}  = m_{{L_\F} \bot_\F {L_\F}} \circ (\Delta_{L_\F}\otF \Delta_{L_\F}) \, ,
\eeq
for the product in $L_\F$, 
$m_{L_\F} ( \xi \ot \zeta)=  (\bF^{\alpha} \trl \xi)\, (\bF_{\alpha} \trl \zeta) $, with $\xi,\zeta \in L_\F$
 and 
the product $m_{{L_\F} \bot_\F {L_\F}} = (m_{L_\F} \otF m_{L_\F})  \circ (\id_{L_\F}\otF \Psi^{\rF}_{{L_\F},{L_\F}} \otF \id_{L_\F})$ in the braided tensor product algebra ${L_\F}\bot_\F {L_\F}$. This latter is given from \eqref{tpr} by
$$
m_{L_\F \bot_\F L_\F}\left( (\xi \bot_\F \zeta) \ot (\xi' \bot_\F \zeta')\right)
= (\bF^{\beta} \trl \xi)(\bF_{\beta} \rF_\alpha \trl \xi') \bot_\F  (\bF^{\gamma}  \rF^\alpha \trl \zeta)  (\bF_{\gamma} \trl \zeta')\, .
$$
The identity involving the counit follows from the algebra map property of $\cun_{L}$, by using  the module coalgebra axiom \eqref{mod-co} and the fact that $\F$ is unital.

For the second identity, with coproduct $\Delta_{L_\F}: \xi \mapsto \fone{\xi} \ot \ftwo{\xi}:=
(\F^{\alpha} \trl \one{\xi})  \ot (\F_\alpha \trl \two{\xi}) $,   
we compute  the l.h.s. by using  the module coalgebra axiom \eqref{mod-co} and the property of the coproduct $\Delta_L$ to be a braided algebra map: on the generic element $\xi \ot \zeta$ one has
\begin{align*}
l.h.s. &=
\Delta_{L_\F}\circ m_{L_\F}  (\xi \ot \zeta) 
\\
&=
 \Delta_{L_\F} \left( (\bF^{\alpha} \trl \xi)\, (\bF_{\alpha} \trl \zeta)  \right)
\\
&=
\F^\gamma \trl \big( \one{(\bF^{\alpha} \trl \xi)} \r_\beta \trl \one{( \bF_{\alpha} \trl \zeta) } \big)
\ot
\F_\gamma \trl \big( \r^\beta \trl  \two{(\bF^{\alpha} \trl \xi)}\two{(\bF_{\alpha} \trl \zeta) } \big)
\\
&=
\F^\gamma \trl \big( (\one{\bF^{\alpha}} \trl \one{\xi}) (\r_\beta \one{\bF_{\alpha}} \trl \one{\zeta }) \big)
\ot
\F_\gamma \trl \big( (\r^\beta \two{\bF^{\alpha}} \trl\two{ \xi}(\two{\bF_{\alpha}} \trl \two{\zeta }) \big)
\\
&=
(\one{\F^\gamma} \one{\bF^{\alpha}} \trl \one{\xi}) (\two{\F^\gamma}\r_\beta \one{\bF_{\alpha}} \trl \one{\zeta }) 
\ot
(\one{\F_\gamma}\r^\beta\two{\bF^{\alpha}}) \trl\two{ \xi}(\two{\F_\gamma}\two{\bF_{\alpha}}) \trl \two{\zeta } \; .
\end{align*}
Similarly, we compute the r.h.s.:
\begin{align*}
r.h.s.&= m_{{L_\F} \bot_\F {L_\F}} \circ (\Delta_{L_\F}\otF \Delta_{L_\F})(\xi \ot \zeta)
\\
&=
(\bF^{\beta} \trl \fone{\xi})(\bF_{\beta} \rF_\alpha \trl \fone{\zeta}) \bot_\F  (\bF^{\gamma}  \rF^\alpha \trl \ftwo{\xi})  (\bF_{\gamma} \trl \ftwo{\zeta})
\\
&=
(\bF^{\beta} \F^\gamma \trl \one{\xi})(\bF_{\beta} \rF_\alpha  \F^\delta \trl \one{\zeta}) \bot_\F  (\bF^{\mu}  \rF^\alpha  \F_\gamma \trl \two{\xi})  (\bF_{\mu} \F_\delta\trl \two{\zeta})\; .
\end{align*}
The two expressions coincide from \eqref{lhs=rhs} in Proposition \ref{main-lem}.

If $L$ is a Hopf algebra with antipode $S_L$, which is a $\K$-module map,   
its image under $\Gamma$, $\xi\mapsto
S_{L_\F}(\xi)=S_L(\xi)$, is a $\K_\F$-module map. We have only to show that  $S_{L_\F}$ satisfies the antipode conditions \eqref{bantipode} for 
the twisted bialgebra $L_\F$.  For any element $\xi \in L_\F$ one has
\begin{align*}
m_{L_\F}  \circ (\id_{L_\F} \ot_\F S_{L_\F})\circ \Delta_{L_\F} (\xi)
& = \bF^\alpha \trl (\F^\beta \trl \one{\xi}) \ot \bF_\alpha \trl (S_L ( \F_\beta \trl \two{\xi}   ))
\\
& = (\bF^\alpha \F^\beta) \trl \one{\xi} \ot (\bF_\alpha  \F_\beta) \trl S_L(\two{\xi}   )
=   \cun_{L_\F} (\xi) 1_{L_\F} \; .
\end{align*}
Analogously one shows that 
$m_{L_\F}  \circ (S_{L_\F} \ot_\F \id_{L_\F} )\circ \Delta_{L_\F} =\eta_{L_\F}  \circ \cun_{L_\F}$.
\end{proof}

\begin{ex}\label{ex:braided-tw}
Let $
\tK :=(\K, m_\K, \eta_\K, \tDelta, \cun_\K, \tS, \tad)
$ be the braided Hopf algebra associated  with the quasitriangular Hopf algebra $(\K,\r)$, as  in Example \ref{ex:braided}.

If $\F$ is a twist for $\K$, the data 
\beq
\tK_\F :=(\tK, m_\KF, \eta_\KF, \tDelta_\F, \cun_\KF, \tS_\F, \tad)
\eeq  
is a braided Hopf algebra associated with $\KF$. From the construction above, $\tK_\F$ has the 
same unit as $\tK$ while its product is the deformed one as in \eqref{rmod-twist}:
$$
m_\KF (h \ot k ):= h \dotF k= (\bF^{\alpha} \trl  h)\, (\bF_{\alpha} \trl k)
$$

\noindent
Its coproduct is the deformed one as in \eqref{cc-twist}:
$$
\tDelta_\F(k):= \F^{\alpha} \tad (\one{k} S(\r_\beta))  \ot (\F_\alpha \r^\beta) \tad \two{k} \, .
$$
The left adjoint action of $\KF$ on $\tK_\F$ is the initial (unchanged) left adjoint action,
$$
 \tad : \KF \ot  \tK_\F \to \tK_\F\; , \quad h \ot k \mapsto h\tad k := \one{h} k S(\two{h}).
$$

\noindent
Finally, the antipode $\tS_\F$ of $\tK_\F$ coincides with that of $\tK$,
$
\tS_\F:  k \mapsto \r_\alpha S( \r^\alpha \tad k)
$, as a linear map.
\end{ex}

\begin{ex}\label{ex:braidedF}
Rather than first considering the braided Hopf algebra $\tK$ and then
twisting it to $\tK_\F$, one could start with the quasitriangular 
Hopf algebra $\KF$  and associate to it the braided Hopf algebra 
$$
\underline{\KF} :=(\KF, m_\KF, \eta_\KF, \underline{\Delta_\F}, \cun_\KF, \underline{S_\F}, \tadF)~,
$$
again constructed as in Example \ref{ex:braided}. 
This consists of the algebra $(\KF, m_\KF, \eta_\KF)$ endowed  with the same counit as $\KF$ and with the braided coproduct
\beq
\underline{\Delta_\F} (k):= \fone{k} S_\F (\rF_\alpha) \ot (\rF^\alpha \tadF \ftwo{k}) \, ,
\eeq
where $\Delta_\F(k)=\fone{k}  \ot \ftwo{k} =\F^{\alpha} \one{k} \bF^{\beta} \ot \F_\alpha \two{k} \bF_{\beta}$ is the coproduct \eqref{cop-twist} of $\KF$. Here the left adjoint action is
\beq
\tadF: \KF \ot \underline{\KF} \to \underline{\KF}\; , \quad 
h \ot k \mapsto h\tadF k := \fone{h} k S_F(\ftwo{h}),
\eeq
with 
$S_\F$ the antipode of $\KF$ : $S_\F(k)=\uf S(k) \buf = \F^{\alpha} S(\F_\alpha) S(k)  S(\bF^{\beta})\bF_{\beta}$.
On elements, 
$$
h \ot k \mapsto h\tadF k 
:= \F^{\alpha} \one{h} \bF^{\beta} \, k \,\F^{\gamma} S(\F_\gamma)
S(\F_\alpha \two{h} \bF_{\beta}) S(\bF^{\delta})\bF_{\delta} \, . 
$$
With these structures $\underline{\KF}$ is a $\KF$ module algebra and module coalgebra, moreover 
$\underline{\Delta_\F}$ 
is an algebra map (with respect to the braided tensor product). The antipode is given by
\beq
\underline{S_\F}: \underline{\KF} \to \underline{\KF} \; , \quad  k \mapsto \underline{S_\F}(k):= \rF_\alpha S_\F( \rF^\alpha \tadF k)
\eeq
and is a left $\KF$-module morphism.
\end{ex}

The following Proposition is dual to Theorem 4.9 in \cite{pgc}; its
proof is  in Appendix \ref{app:DKF}.
\begin{prop} \label{DKF}
The two braided Hopf algebras $\tK_\F$ and $\underline{\KF}$ associated with $\KF$ are isomorphic via the map (cf. \eqref{mappaD}) 
\beq\label{mapD}
\dd :\tK_\F \to \underline{\KF} \; , \quad k \mapsto 
\dd(k):=(\bF^\alpha \tad k)\bF_\alpha 
= \F^\alpha k S(\F_\alpha) \buf
\eeq
with inverse
\beq\label{mapDinv}
\dd^{-1}: \underline{\KF} \to \tK_\F   \; , \quad k \mapsto 
\dd^{-1}(k)=(\F^\alpha \tadF k)\F_\alpha 
= \bF^\alpha k S_\F(\bF_\alpha) \uf~.
\eeq
\end{prop}

\subsection{Dual pairings of braided bialgebras and their twisting} \label{sec:dualpair}
\begin{defi}
A dual pairing of two braided
bialgebras
$(L, m_L, \eta_L, \Delta_L,\cun_L, \trl_L)$ and 
$(N, m_N, \eta_N, \Delta_N,\cun_N, \trl_N)$ associated with the quasitriangular  bialgebra $(\K,\r)$ is a bilinear map $\hsr{\cdot}{\cdot}  : L \otimes N \to \bbK$ such that
\begin{align}\label{hrs-cop}
\hsr{\xi}{xy} 
=  \hsr{\Delta_L{\xi}}{x \ot y} := \hsr{\one{\xi}}{\r_\alpha \trl_N x} \hsr{\r^\alpha \trl_L \two{\xi}}{y}
\nn
\\
\hsr{\xi \eta}{x} 
=  \hsr{\xi \ot \eta}{\Delta_N{x}} 
:= \hsr{\xi}{\r_\alpha \trl_N \one{x}} \hsr{\r^\alpha \trl_L \eta }{\two{x}}
\end{align}
and 
\beq\label{hrs-counit}
\hsr{\xi}{1_N}= \cun_L(\xi) \; , \quad \hsr{1_L}{x}= \cun_N(x)
\eeq
for all $\xi,~\eta \in L$ and $x,~y\in N$. Moreover $\hsr{\cdot}{\cdot}  : L \otimes N \to \bbK$ is a morphism of $\K$-modules: 
\beq\label{pair-mor}
\hsr{\one{k}\trl_L \xi}{\two{k}\trl_N x} = \cun(k)\hsr{\xi}{x}
\eeq
for each $k \in \K$, $\xi \in L$ and $x\in N$. 
We say that $L$ and $N$ are dually paired if they admit a dual pairing.
\end{defi}
\begin{lem}
When $K$ is a Hopf algebra, the condition for the pairing
$\hsr{\cdot}{\cdot}$ to be a morphism of $\K$-modules can 
be given equivalently as \beq\label{pair-mor-S}
\hsr{{k}\trl_L \xi}{ x} =\hsr{\xi}{S(k)\trl_N x} \, , \quad \forall ~\xi \in L\;, x \in N \; ,
\eeq
or, with $S^{-1}$ the inverse of the antipode of $K$, 
\beq\label{pair-mor-S2}
\hsr{\xi}{k\trl_N x} =\hsr{S^{-1}(k)\trl_L \xi}{ x} \, , \quad \forall ~ \xi \in L\;, x \in N \; .
\eeq
\end{lem}
\begin{proof}
If \eqref{pair-mor} holds, then for all $\xi \in L\;, x \in N$ we have
$$
\hsr{{k}\trl_L \xi}{ x} 
= \hsr{\one{k}\trl_L \xi}{\cun(\two{k})\trl_N x} 
= \hsr{\one{k}\trl_L \xi}{\two{k} \trl_N (S(\three{k})\trl_N x)} 
= \hsr{\xi}{S(k)\trl_N x} .
$$
Conversely, if \eqref{pair-mor-S} holds, then
$$
\hsr{\one{k}\trl_L \xi}{\two{k}\trl_N x} 
=\hsr{\xi}{S(\one{k})\trl_N (\two{k}\trl_N x)} 
=\hsr{\xi}{(S(\one{k})\two{k})\trl_N x} 
= \cun(k)\hsr{\xi}{x}.
$$
Being the antipode of a quasitriangular Hopf algebra invertible (see \eqref{Sinv}), identity \eqref{pair-mor-S}  gives \eqref{pair-mor-S2}, when replacing $k$  by $S^{-1}(k)$.
\end{proof}

\begin{lem}
If $\hsr{\cdot}{\cdot}  : L \otimes N \to \bbK$ is a dual pairing between two braided Hopf algebras $L$ and $N$ associated with the quasitriangular Hopf algebra $(K,\r)$, then 
\beq
\hsr{S_L(\xi)}{x}= \hsr{\xi}{S_N(x)} \, , \quad \forall ~ \xi \in L\;, x \in N \; .
\eeq
\end{lem}
\begin{proof}
We omit the subscripts $_L$ and $_N$. Using the pairing axioms \eqref{hrs-cop}, we compute
\begin{align*}
\hsr{S(\xi)}{x}&= \hsr{S(\one{\xi})}{\r_\alpha \trl \one{x}} \cun(\r^{\alpha} \trl \two{\xi})
\cun(\two{x})
\\
&= \hsr{S(\one{\xi})}{\r_\alpha \trl \one{x}} \hsr{\r^{\alpha} \trl \two{\xi}}{\two{x} S (\three{x})}
\\
&= \hsr{S(\one{\xi})}{\r_\alpha \trl \one{x}} 
\hsr{\one{(\r^{\alpha} \trl \two{\xi})}}{\r_\beta \trl \two{x}} 
\hsr{\r^\beta \trl \two{(\r^{\alpha} \trl \two{\xi})}}{ S (\three{x})} 
\\
&= \hsr{S(\one{\xi})}{\r_\alpha \trl \one{x}} 
\hsr{\one{\r^{\alpha} }\trl \two{\xi}}{\r_\beta \trl \two{x}} 
\hsr{(\r^\beta \two{\r^{\alpha}}) \trl \three{\xi}}{ S (\three{x})} 
\end{align*}
where we used the module coalgebra axioms \eqref{mod-co} in the last equality. Next, we use the quasitriangularity condition  $(\Delta \ot \id) \r= \r_{13} \r_{23}$ and get
\begin{align*}
\hsr{S(\xi)}{x}=\hsr{S(\one{\xi})}{\r_\alpha \r_\gamma \trl \one{x}} 
\hsr{{\r^{\alpha} }\trl \two{\xi}}{\r_\beta \trl \two{x}} 
\hsr{(\r^\beta {\r^{\gamma}}) \trl \three{\xi}}{ S (\three{x})} .
\end{align*} 
On the other hand, with analogous computations we have
 \begin{align*}
\hsr{\xi}{S(x)}&= \cun( \one{\xi}) \cun(\r_{\gamma} \trl\one{x}) \hsr{\r^\gamma \trl \two{\xi}}{ S(\two{x})} 
\\
&= \hsr{ S(\one{\xi})\two{\xi}}{\r_{\gamma} \trl\one{x}} \hsr{\r^\gamma \trl \three{\xi}}{ S(\two{x})} 
\\
&= \hsr{ S(\one{\xi})}{\r_\alpha \trl \one{(\r_{\gamma} \trl\one{x})}}
\hsr{\r^\alpha \trl  \two{\xi}}{\two{(\r_{\gamma} \trl\one{x})}}
 \hsr{\r^\gamma \trl \three{\xi}}{S( \two{x})} 
\\
&= \hsr{ S(\one{\xi})}{\r_\alpha \one{\r_{\gamma}} \trl\one{x}}
\hsr{\r^\alpha \trl  \two{\xi}}{\two{\r_{\gamma} }\trl\two{x}}
 \hsr{\r^\gamma \trl \three{\xi}}{ S(\three{x})} .
\end{align*}
Finally, the quasitriangular condition $ (\id \ot \Delta) \r=\r_{13} \r_{12} $  gives
\begin{align*}
\hsr{\xi}{S(x)}=
 \hsr{ S(\one{\xi})}{\r_\alpha {\r_{\gamma}} \trl\one{x}}
\hsr{\r^\alpha \trl  \two{\xi}}{{\r_{\beta} }\trl\two{x}}
 \hsr{\r^\beta \r^\gamma  \trl \three{\xi}}{ S(\three{x})} = \hsr{\xi}{S(x)}, 
\end{align*}
the stated identity. \end{proof}

\begin{prop}\label{TwistDPBH}
Let $\hsr{\cdot}{\cdot}  : L \otimes N \to \bbK$ be a dual pairing between two braided
bialgebras
$(L, m_L, \eta_L, \Delta_L,\cun_L, \trl_L)$ and 
$(N, m_N, \eta_N, \Delta_N,\cun_N, \trl_N)$ associated with the quasitriangular  bialgebra $(\K,\r)$. 
Then 
the braided bialgebras
$({L_\F}, m_{L_\F}, \eta_{L_\F}, \Delta_{L_\F},\cun_{L_\F}, \trl_{L_\F})$ and 
$({N_\F}, m_{N_\F}, \eta_{N_\F}, \Delta_{N_\F},\cun_{N_\F}, \trl_{N_\F})$ associated with 
the quasitriangular  bialgebra $(\KF,\rF)$ (as given in Proposition \ref{prop:td-bb})
are dually paired via
\beq\label{hrsF}
\hsrF{\cdot}{\cdot}  : L_\F \otimes N_\F \to \bbK \; , \quad
\xi \otF x \mapsto \hsrF{\xi}{x}:=\hsr{\bF^\alpha \trl_L \xi}{\bF_\alpha \trl_N x}
\eeq
\end{prop}
\begin{proof}
Firstly we show that this map
$\hsrF{\cdot}{\cdot}$ 
is a $\KF$-module map. With the coproduct for $\KF$ given in \eqref{cop-twist},
we have
\begin{align*}
\hsrF{\fone{k}\trl_L \xi}{\ftwo{k}\trl_N x} &=
\hsrF{\F^\alpha \one{k}\bF^\beta \trl_L \xi}{\F_\alpha\two{k}\bF_\beta\trl_N x} 
=\hsr{ \one{k}\bF^\beta \trl_L \xi}{\two{k}\bF_\beta\trl_N x} 
\\
&=  \cun(k)\hsr{\bF^\beta \trl_L \xi}{\bF_\beta\trl_N x} =
 \cun(k)\hsrF{\xi}{x}
\end{align*}
where we have used that the pairing $\hsr{\cdot}{\cdot}: L \ot N \to \bbK$ is a $\K$-module map.
Next, being $F$ unital, the  compatibility properties \eqref{hrs-counit} 
with the counits are easily verified.   
We show the first of properties \eqref{hrs-cop}, the other one can be established similarly. 
To lighten the notation we omit the subscripts. We have
\begin{align*}
\hsrF{\xi}{x \dotF y} &=
\hsr{\bF^\alpha \trl \xi}{\bF_\alpha \trl \big( (\bF^\beta \trl x )(\bF_\beta \trl y)\big)}
\\
&=
\hsr{\bF^\alpha \trl \xi}{(\one{\bF_\alpha} \bF^\beta \trl x )(\two{\bF_\alpha} \bF_\beta \trl y)}
\\
&=
\hsr{\one{\bF^\alpha}\bF^\beta \trl \xi}{(\two{\bF^\alpha}\bF_\beta \trl x )({\bF_\alpha}  \trl y)}
\\
&=
\hsr{\one{\bF^\alpha}\one{\bF^\beta} \trl \one{\xi}}{\r_\gamma\three{\bF^\alpha}\bF_\beta \trl x }
\hsr{\r^\gamma\two{\bF^\alpha}\two{\bF^\beta} \trl \two{\xi}}{{\bF_\alpha}  \trl y}
\end{align*}
where for the last two identities we have used the twist condition \eqref{twist-bF}, followed by  the module coalgebra condition  \eqref{mod-co}.
This can be further simplified by using the quasitriangular condition (\ref{iiR}) of the coproduct of $\K$ and then the 
 property \eqref{pair-mor} for the pairing to be a $\K$-module map:
\begin{multline*}
\hsr{\one{\bF^\alpha}\one{\bF^\beta} \trl \one{\xi}}{\two{\bF^\alpha}\r_\gamma\bF_\beta \trl x }
\hsr{\three{\bF^\alpha}\r^\gamma \two{\bF^\beta} \trl \two{\xi}}{{\bF_\alpha}  \trl y}
\\
=\hsr{\one{\bF^\beta} \trl \one{\xi}}{\r_\gamma\bF_\beta \trl x }
\hsr{{\bF^\alpha}\r^\gamma \two{\bF^\beta} \trl \two{\xi}}{{\bF_\alpha}  \trl y} \, .
\end{multline*}
By using again the twist condition and the coproduct in $L_\F$, we have
\begin{align*}
\hsr{{\bF^\beta} \F^\tau \trl \one{\xi}}{\r_\gamma \two{\bF_\beta} \bF_\mu\trl x }
&
\hsr{{\bF^\alpha}\r^\gamma \one{\bF_\beta}\bF^\mu \F_\tau \trl \two{\xi}}{{\bF_\alpha}  \trl y} 
\\
&=\hsr{{\bF^\beta}  \trl \fone{\xi}}{\r_\gamma \two{\bF_\beta} \bF_\mu\trl x }
\hsr{{\bF^\alpha}\r^\gamma \one{\bF_\beta}\bF^\mu  \trl \ftwo{\xi}}{{\bF_\alpha}  \trl y} 
\\
&=\hsr{{\bF^\beta}  \trl \fone{\xi}}{ \one{\bF_\beta} \r_\gamma\bF_\mu\trl x }
\hsr{{\bF^\alpha} \two{\bF_\beta}\r^\gamma \bF^\mu  \trl \ftwo{\xi}}{{\bF_\alpha}  \trl y} 
\\
&=\hsr{{\bF^\beta}  \trl \fone{\xi}}{ \one{\bF_\beta} \bF^\mu \rF_\gamma \trl x }
\hsr{{\bF^\alpha} \two{\bF_\beta} \bF_\mu \rF^\gamma \trl \ftwo{\xi}}{{\bF_\alpha}  \trl y} 
\end{align*}
where in the last but one equality we have used once again the quasitriangular condition (\ref{iiR}) in the Definition \ref{def:qt}, 
of the coproduct and in the last one the definition of $\rF$. Then, the twist condition and the
$\K$-module map property \eqref{pair-mor} of the pairing yields
\begin{align*}
\hsr{\one{\bF^\mu}{\bF^\beta}  \trl \fone{\xi}}{ \two{\bF^\mu} {\bF_\beta}  \rF_\gamma \trl x }
&
\hsr{{\bF^\alpha}  \bF_\mu \rF^\gamma \trl \ftwo{\xi}}{{\bF_\alpha}  \trl y} 
\\
&=
\hsr{{\bF^\beta}  \trl \fone{\xi}}{{\bF_\beta}  \rF_\gamma \trl x }
\hsr{{\bF^\alpha}   \rF^\gamma \trl \ftwo{\xi}}{{\bF_\alpha}  \trl y} 
\\
&=
 \hsrF{\fone{\xi}}{\rF_\alpha \trl x} \hsrF{\rF^\alpha \trl \ftwo{\xi}}{y} \; ,
\end{align*}
thus concluding the proof.
\end{proof}

\section{Braided Lie algebras and their twisting}

Further insights in the use of braided Hopf algebras as algebras of symmetries  come from the
study of associated braided Lie algebras: Lie algebras in the braided
monoidal category of $K$-modules, with $(K,\r)$
triangular. 

\subsection{Braided Lie algebras}
\begin{defi}
A \textbf{braided Lie algebra} associated with a triangular
  Hopf algebra $(K, \r)$, or simply a $K$-braided Lie algebra, 
is a $K$-module $\g$ with a
bilinear map (a nonassociative non-unital multiplication)
$$
[~,~]: \g\otimes \g\to \g
$$
that satisfies the conditions:
\begin{enumerate}[(i)]
\item
$K$-equivariance: $$k\trl [u,v]=[\one{k}\trl u,\two{k}\trl v] ,$$ 
\item braided antisymmetry: $$[u,v]=-[\r_\alpha\trl v, \r^\alpha \trl u ] , $$
\item
braided Jacobi identity: $$[u,[v,w]]=[[u,v],w]+[\r_\alpha\trl v, [\r^\alpha \trl u,w] ] , $$
\end{enumerate}
for all $u,v,w \in \g$, $k\in K$.
A $K$-braided Lie
  algebra morphism $\varphi: \g\to \g'$ is a morphism of the $K$-modules $\g$
  and $\g'$ such that
  $\varphi\circ [~,~]=[~,~]'\circ (\varphi\otimes\varphi)$, for $[~,~]$ and $[~,~]' $ the brackets of $\g$ and $\g'$ respectively.
  \end{defi}
We add the subscript $[~,~]_\r$ to the bracket when we need to
  emphasize the $\r$-matrix structure we are using. 

\begin{lem}\label{lem:Abla} 
Let $(K, \r)$ be a triangular Hopf algebra.
A $K$-module algebra $A$  is a $K$-braided Lie algebra with bracket
\beq\label{Abla}
[~,~]: A\ot A\to A, \qquad a\ot b\mapsto [a,b] = ab - (\r_\alpha \trl b) \, (\r^\alpha\trl a) \, .
\eeq
\end{lem}
\begin{proof} 
The $K$-equivariance follows from the
$K$-module algebra property \eqref{mod-a} of $A$ and quasi-cocommutativity of $K$:
\begin{align*}
k \trl [a,b] & = 
(\one{k} \trl a)(\two{k} \trl b) - (\one{k} \r_\alpha \trl b) (\two{k}  \r^\alpha\trl a)
\\
& = 
(\one{k} \trl a)(\two{k} \trl b) - (\r_\alpha \two{k} \trl b) (  \r^\alpha \one{k} \trl a)
\\
&= 
[\one{k} \trl a ,\two{k} \trl b] \, .
\end{align*} 
The braided antisymmetry follows from the triangularity of $K$: 
\begin{align*}
[\r_\alpha\trl b, \r^\alpha \trl a ] &= (\r_\alpha\trl b)(\r^\alpha \trl a)- (\r_\beta \r^\alpha \trl a) \, (\r^\beta \r_\alpha\trl b)
\\
&= (\r_\alpha\trl b)(\r^\alpha \trl a)- a b = - [a,b] .
\end{align*} 
We are left to prove Jacobi identity.  On the one hand, using  $K$-equivariance of $[ \, , \, ]$ just shown, and
condition $(\id \ot \Delta)\r=\r_{13} \r_{12}$ together with quasi co-commutativity, we have
\begin{align*}
 [a,[b,c]] & = a[b,c] - (\r_\alpha \trl  [b,c] ) (\r^\alpha \trl a)
\\
&= a[b,c] -  [\one{\r_\alpha} \trl b, \two{\r_\alpha} \trl c] (\r^\alpha \trl a)
\\ 
&=  
a b c - a (\r_\beta \trl c) \, (\r^\beta\trl b)
- (\one{\r_\alpha} \trl b)( \two{\r_\alpha} \trl c) (\r^\alpha \trl a) 
\\
& \quad +
  (\r_\beta \two{\r_\alpha} \trl c) (\r^\beta \one{\r_\alpha} \trl b)(\r^\alpha \trl a) 
\\ 
&=  
a b c - a (\r_\beta \trl c) \, (\r^\beta\trl b)
- (\one{\r_\alpha} \trl b)( \two{\r_\alpha} \trl c) (\r^\alpha \trl a) 
\\
& \quad +
  (\one{\r_\alpha} \r_\beta \trl c) (\two{\r_\alpha}  \r^\beta  \trl b)(\r^\alpha \trl a)  
\\ 
&=  
a b c - a (\r_\beta \trl c) \, (\r^\beta\trl b)
- (\r_\gamma \trl b) (\r_\alpha \trl c)(\r^\alpha \r^\gamma \trl a)  
\\
& \quad +  (\r_\alpha \r_\beta  \trl c) \, ( \r_\gamma \r^\beta \trl b) (\r^\gamma  \r^\alpha 
  \trl a)
\, .
\end{align*}

 On the other hand, using also $(\Delta \ot \id) \r=\r_{13} \r_{23}$, analogous computations give
\begin{align*}
& [[a,b],c] + [\r_\alpha \trl b ,[\r^\alpha \trl a , c]]
\\ 
&= [a,b]c 
- (\r_\beta \trl c) \, (\r^\beta \trl [ a,  b] )
+ (\r_\alpha \trl b) [\r^\alpha \trl a , c] 
- (\r_\beta \trl [\r^\alpha \trl a , c])(\r^\beta \r_\alpha \trl b)
\\
 &= [a,b]c 
- (\r_\beta \trl c) \,  [\one{\r^\beta} \trl a, \two{\r^\beta} \trl b] 
+ (\r_\alpha \trl b) [\r^\alpha \trl a , c] 
\\
&\quad
- [ \one{\r_\beta} \r^\alpha \trl a , \two{\r_\beta} \trl  c](\r^\beta \r_\alpha \trl b)
\\
 &= [a,b]c 
- (\r_\gamma \r_\beta \trl c) \,  [\r^\gamma \trl a, \r^\beta \trl b] 
+ (\r_\alpha \trl b) [\r^\alpha \trl a , c] 
\\
&\quad
- [ \r_\gamma \r^\alpha \trl a , \r_\beta \trl  c]( \r^\beta  \r^\gamma\r_\alpha \trl b)
\\
 &= [a,b]c 
- (\r_\gamma \r_\beta \trl c) \,  [\r^\gamma \trl a, \r^\beta \trl b] 
+ (\r_\alpha \trl b) [\r^\alpha \trl a , c] 
- [ a , \r_\beta \trl  c]( \r^\beta   \trl b)
\\
&=ab c  - (\r_\alpha \trl b) \, (\r^\alpha\trl a) c
- (\r_\gamma\r_\beta \trl c) \, (\r^\gamma \trl a ) (\r^\beta \trl b )
\\
& \quad +  (\r_\gamma\r_\beta \trl c)  ( \r_\alpha \r^\beta \trl b) \, (\r^\alpha \r^\gamma \trl a) 
+ (\r_\alpha \trl b) (\r^\alpha \trl  a)c 
\\
& \quad -  (\r_\alpha \trl b) (\r_\beta \trl c) \, (\r^\beta \r^\alpha \trl a) 
-  a (\r_\beta \trl  c)( \r^\beta   \trl b)
+  (\r_\alpha \r_\beta \trl  c)  (\r^\alpha \trl a ) ( \r^\beta   \trl b)
\\
&=ab c  
+  (\r_\gamma\r_\beta \trl c)  ( \r_\alpha \r^\beta \trl b) \, (\r^\alpha \r^\gamma \trl a) 
\\
& \quad -  (\r_\alpha \trl b) (\r_\beta \trl c) \, (\r^\beta \r^\alpha \trl a) 
-  a (\r_\beta \trl  c)( \r^\beta   \trl b)
\end{align*} 
where for the fourth equality we have used the triangularity of $\r$.
This expression equals that for $ [a,[b,c]]$ found above.
\end{proof}
\begin{rem}\label{Poisson}
Let $(K,\r)$ be a triangular Hopf algebra. A {\bf $K$-braided Poisson
  algebra}  $(P,\cdot, [~,~])$ is a $K$-module algebra  $(P,\cdot)$,
and a 
$K$-braided Lie algebra $(P,[~,~])$ such that the bracket  is a braided derivation of the  multiplication: for all $p,r,s \in P$,
\begin{equation}\label{eqPois}
[p,r\cdot s]=[p,r]\cdot
s+\r_\alpha\trl r\cdot [\r^\alpha \trl p,s]~.
\end{equation}
In particular, any $K$-braided algebra $(A, \cdot)$ is canonically a $K$-braided Poisson
algebra $(A, \cdot, [~,~])$ with the bracket
defined in \eqref{Abla}. 
\end{rem}
An enveloping algebra of $\g$ is the datum $(A,\alpha)$ of a 
$K$-module  algebra $A$, with braided Lie algebra structure as in Lemma \ref{lem:Abla},
and  
a braided Lie algebra morphism $\alpha: \g\to A$, that is a $K$-equivariant map such that
$\alpha([u,v])=[\alpha(u),\alpha(v)]$, for all $u,v\in \g$.  

The universal enveloping algebra $(\U(\g), \iota)$ of $\g$ is characterised by the following universal
property:  given an enveloping algebra $(A,\alpha)$ there is a
unique lift of $\alpha$ to a $K$-equivariant algebra map $\hat\alpha: \U(\g)\to A$, such that $\hat\alpha\circ \iota=\alpha:\g\to A$.

Existence of $(\U(\g), \iota)$ follows from considering the usual 
tensor algebra $T(\g)$, canonically a $K$-module algebra,  and the quotient  $\U(\g):=T(\g)/I$
where $I$ is the algebra ideal generated by $u\ot v-\r_\alpha
\trl v\ot\r^\alpha\trl u-[u,v]$. The quotient is a $K$-module algebra since the ideal $I$ is preserved by the $K$-action due to the quasi-cocommutativity of $K$. The $K$-equivariant map $\hat\iota:\g\to T(\g)$, $u\mapsto\hat\iota(u)=u$
induces the braided Lie algebra morphism $\iota:\g\to \U(\g)$ on the quotient. As in the classical case, the universality 
of $(\U(\g), \iota)$ follows from that of $T(\g)$ and uniqueness (up to isomorphism) from the universal property.

\begin{prop}\label{prop:uea}
The universal enveloping algebra $\U(\g)$ of a braided Lie algebra $\g$ is a braided Hopf
algebra associated with the triangular Hopf algebra $(K, \r)$.
\end{prop}
\begin{proof}
We first show that the tensor algebra $T(\g)$ is a braided Hopf algebra. 
We define the coproduct $\Delta (u)=u\bot 1+1\bot u$ for degree one elements $u\in \g$ and extend it to all
$T(\g)$ by requiring it to be a unital algebra map between $T(\g)$ and the braided tensor
algebra $T(\g)\bot T(\g)$ (cf. Proposition \ref{BTPalgK}):
$$
\Delta(\zeta \ot \xi)=(\one{\zeta}\ot\r_\alpha\trl\one{\xi})\bot
(\r^\alpha\trl\two{\zeta}\ot\two{\xi})$$
 for all $\zeta,\xi\in T(\g)$.
Coassociativity of $\Delta$ on degree one elements is trivial, it is then easy to see that if it holds for 
elements $\zeta$ and $\xi$ in $T(\g)$ it holds also for their product.
The counit $\varepsilon:T(\g)\to \bbK$ is the unital algebra map defined by
$\varepsilon(u)=0$ for all $u\in \g$.  
Then $(T(\g), \Delta, \varepsilon)$
is a $K$-module coalgebra as in \eqref{mod-co}.
Indeed,
the $K$-module coalgebra property $\Delta(k\trl
\zeta)=\one{k}\trl \one{\zeta}\bot\two{k}\trl\two{\zeta}$ 
holds on degree one elements and, using quasi-cocommutativity of $K$, it is  seen to hold on the
product $\zeta\otimes\xi$ if it holds on the single elements $\zeta,
\xi\in T(\g)$. The condition on the counit is trivially satisfied. Since $\Delta$ has been defined as an algebra map
$T(\g)\to \T(\g)\bot T(\g)$ we have that $T(\g)$ is a braided bialgebra 
associated with $(K,\r)$. It is a braided Hopf algebra with antipode 
the $K$-equivariant map  $S(u)=-u$ for $u\in \g$, extended to $T(\g)$ as braided anti-algebra map.

Finally, the quotient $\U(\g) = T(\g) / I$ is a braided Hopf algebra since the ideal 
$I = \langle u\ot v-\r_\alpha \trl v\ot \r^\alpha\trl u-[u,v] \rangle$ is a braided Hopf
ideal. Indeed, as already mentioned $I$ is a $K$-module; it is also a coideal since
$$\Delta(u\ot v-\r_\alpha \trl v\ot \r^\alpha\trl v-[u,v])\in I\bot
T(\g)+T(\g)\bot I$$ for all $u,v\in \g$, as can be shown by using triangularity of the universal $\r$ matrix.
\end{proof}

\begin{rem}
Over a field of  characteristic zero a braided Lie algebra $\g$ is
the $K$-submodule in $\U(\g)$ of primitive elements:
${\rm{Prim}}(\U(\g))=\g$. More in general, a connected braided cocommutative
Hopf algebra $L$ is the universal enveloping algebra of the braided Lie
algebra $\g={\rm{Prim}}(L)$ of its primitive elements, see \cite{Khor}.
\end{rem}

\subsection{Braided Lie algebras of braided derivations}\label{BLA}
We now study derivations of $K$-module algebras.

Let $A$ be a $K$-module algebra and $(\Hom(A,A), \trl_{\Hom(A,A)})$ the $K$-module
of linear maps from $A$ to $A$ as in  \eqref{action-hom}. 
We denote
\beq\label{Der}
\Der{A}:= \{ \psi \in  \Hom(A,A) \, | \, \psi(aa')= \psi(a)
a' + (\r_\alpha\trl a)\,(\r^\alpha 
\trl_{\Hom(A,A)}
\psi)(a') \}
\eeq
the  $\bbK$-module
of \textbf{braided derivations} of $A$. We add the superscript $\r$
and write $\rm{Der}^\r{(A)}$
when we wish to emphasize the
role of the braiding. 
{Braided derivations} are
 a {$K$-submodule} of $\Hom(A,A)$: 
 \begin{lem} For $(K,\r)$ quasi-triangular,
$\Der{A}$ is a $K$-module with action given by the restriction of 
 $\trl_{\Hom(A,A)}$:
 \begin{align}\label{action-der}
\trl_{\Der{A}}:  K \ot \Der{A} &\to \Der{A}
\nn
\\
k \ot \psi & \mapsto  k \trl_{\Der{A}} \psi : \; a \mapsto \one{k} \trl \psi(S(\two{k})\trl a) .
\end{align}
\end{lem}
\begin{proof}
We only need to show that the restriction of $\trl_{\Hom(A,A)}$ to $\Der{A}$ has image in $\Der{A}$, that is, $k \trl_{\Hom(A,A)} \psi$ is a braided derivation if $\psi$ is such, for each $k \in K$. 
\begin{align*}
(k \trl & \psi )(ab) = \one{k} \trl \psi \big(S(\two{k}) \trl (ab)\big)
\\
& = 
\one{k} \trl \psi \big((S(\three{k}) \trl a) (S(\two{k}) \trl b)\big)
\\
& = 
\one{k} \trl \Big( \psi\big (S(\three{k}) \trl a \big) (S(\two{k}) \trl b)
+ (\r_\alpha S(\three{k}) \trl a) (\r^\alpha \trl  \psi) \big(S(\two{k}) \trl b\big) \Big)
\\
& = 
\one{k} \trl \psi\big( S(\four{k}) \trl a \big) \;  (\two{k}S(\three{k}) \trl b)
\\ & \qquad
+ (\one{k}  \r_\alpha S(\four{k}) \trl a) \; \two{k} \trl \big((\r^\alpha \trl \psi)(S(\three{k}) \trl b)\big)
\end{align*}
using that $\psi$ is a braided derivation and $A$ is a module algebra.
Then, the first term in the sum simplifies to $(k\trl \psi) (a) \, b$, while for the second term we have
\begin{align*}
\big(\one{k}  \r_\alpha S(\four{k}) \trl a\big) \; &\two{k} \trl \big((\r^\alpha \trl \psi)(S(\three{k}) \trl b)\big)
\\
&= 
(\one{k}  \r_\alpha S(\four{k}) \trl a) \;  (\two{k}  \one{\r^\alpha} )\trl \psi (S(\three{k} \two{\r^\alpha}) \trl b)
\\
&= 
(\one{k}  \r_\alpha \r_\beta S(\four{k}) \trl a) \;  ( \two{k} \r^\alpha )\trl \psi (S(\three{k} \r^\beta ) )\trl b)
\\
&= 
(\one{k}  \r_\alpha  S(\four{k} \br_\beta ) \trl a) \;  ( \two{k} \r^\alpha )\trl \psi ( S( \three{k} \br^\beta)\trl b)
\\
&= ( \r_\alpha  \two{k}  S( \br_\beta \three{k} ) \trl a) \;  ( \r^\alpha \one{k} )\trl \psi ( S(  \br^\beta \four{k})\trl b)
\\
&= ( \r_\alpha    \r_\beta \trl a) \;  ( \r^\alpha \one{k} )\trl \psi ( S(  \r^\beta \two{k})\trl b)
\end{align*}
using quasitriangularity for the second equality,  $(\id \ot S) \br = \r$ for the third one,  and
quasi-cocommutativity twice. Finally, again by quasitriangularity this is rewritten as 
$$
( \r_\alpha    \trl a) \;  ( \one{(\r^\alpha k)} )\trl \psi ( S( \two{(\r^\alpha k}))\trl b)
=  ( \r_\alpha    \trl a) \;  ((\r^\alpha  k) \trl \psi) (  b)
$$
thus showing that $k \trl \psi$ is a braided derivation:
$$
(k \trl \psi )(ab) = (k\trl \psi) (a) b +  ( \r_\alpha    \trl a) \;  (\r^\alpha  \trl (k \trl \psi)) (  b) .
$$
This ends the proof.
\end{proof}

As in Lemma \ref{lem:Abla}, when the Hopf algebra $K$ is triangular, associated with the $K$-module
    algebra $(\Hom(A,A),\circ)$ there is the braided Lie algebra
    $(\Hom(A,A), [~,~])$ where $[~,~]$ is the braided commutator as in \eqref{Abla}.  

For $a\in A$ we define
\begin{equation}\label{defofell}\ell_a\in \Hom(A,A)~,~~
  \ell_a:A\to A\,,~ a'\mapsto \ell_a(a')=aa'~.
\end{equation}
Using the $K$-module structure \eqref{action-hom} of  $\Hom(A,A)$, one shows  that $\ell :A\to \Hom(A,A)$, $a\mapsto \ell_a$, is $K$-equivariant, $k\trl\ell _a=\ell_{k\trl
    a}$ for all $k\in K$, $a\in A$. It follows that a linear map $\psi\in
  \Hom(A,A)$ is a $(K,\r)$-braided derivation, $\psi(ab)= \psi(a) b + (\r_\alpha\trl a)\,(\r^\alpha  \trl \psi)(b) $, if and only if
$\psi (\ell_a (b))= \ell_{\psi(a)} (b) + (\ell_{\r_\alpha\trl a}\circ (\r^\alpha  \trl \psi))(b) $, that is 
if and only if 
  \begin{equation}
    \label{psiella0}
    [\psi,\ell_a]=\ell_{\psi(a)}\end{equation}
  for all $a\in A$. 

\begin{prop}\label{prop:der}
Let $(K,\r)$ be triangular.
 For each $K$-module algebra $A$,  the  $K$-module of braided
 derivations with
\begin{align}\label{bracket-der}
[~ , ~]  :& \, \Der{A}\ot \Der{A} \to \Der{A}
\nn \\
& \psi \ot \lambda \mapsto [\psi,\lambda] :=\psi\circ \lambda-(\r_\alpha\trl_{\Der{A}} \lambda) \circ (\r^\alpha\trl_{\Der{A}} \psi) 
\end{align}
is a  braided Lie subalgebra of  $\Hom(A,A)$.
\end{prop}
\begin{proof}
We only need to show that the image of $[~,~]$ is in $\Der{A}$.    
Let $\psi,\lambda \in \Der{A}$ and $a\in A$. Using the
characterization \eqref{psiella0} of briaded derivations  and the Jacobi identity in $\Hom(A,A)$, we have
\begin{align*}
[[\psi,\lambda], \ell_a] &= [\psi,[\lambda, \ell_a]] - [ \r_\alpha\trl \lambda, [\r^\alpha\trl \psi , \ell_a]]
\\
&= [\psi, \ell_{\lambda(a)}] - [ \r_\alpha\trl \lambda,  \ell_{(\r^\alpha\trl \psi) (a)}]
\\
&= \ell_{\psi (\lambda(a))} - \ell_{( \r_\alpha\trl \lambda)(\r^\alpha\trl \psi) (a)}
\\
&= \ell_{[\psi ,\lambda](a)} 
\end{align*}
and thus $[\psi,\lambda]$ is a braided derivation.
\end{proof}

Recall that the  $K$-module algebra $A$ is {\bf quasi-commutative} (or
braided commutative) when
\beq\label{qcAB} 
a\, a'  =   (\r_\alpha \trl{a'}) \, (\r^\alpha\trl{a})~, 
\eeq
 for all $ a,a' \in A$. In this case the braided Lie algebra $\Der{A}$ is also
a left $A$-submodule of $\Hom(A,A)$  
by defining
\begin{equation}\label{AmodderA}
  (a \psi)(c):=a\, \psi(c)
  \end{equation} 
for  $\psi \in
  \Hom(A,A)$, $a,c\in A$, that is,
  $a\psi:=\ell_a\circ \psi~.$
 Indeed, for  $a\in A$, 
$\psi\in
\Der{A}$ one has:
\begin{align*}
a\psi(cc')
&=a\psi(c)c'+a (\r_\alpha \trl
c)(\r^\alpha\trl\psi)(c')
=a\psi(c)c'+(\r_\beta\r_\alpha \trl
c)(\r^\beta\trl a)(\r^\alpha\trl\psi)(c')
\\[.2em]
&
=  a\psi(c)c'+( \r_\alpha \trl
c)(\one{\r^\alpha} \trl a)(\two{\r^\alpha}\trl\psi)(c')
=
a\psi(c)c'+ (\r_\alpha \trl
c)(\r^\alpha\trl a\psi)(c')\,.
\end{align*}
The compatibility with the $K$-action, as in \eqref{eqn:KArelmod},
shows that $\Der{A}$ is a 
$(K,A)$-relative Hopf module. The compatibility of the $A$-module and the braided
Lie algebra structures follows from 
$(\Hom(A,A),\circ, [~,~])$ being a braided Poisson algebra (cf. Remark \ref{Poisson}).

\begin{prop}\label{modLie}
Let $(K,\r)$ be triangular and let $A$ be a quasi-commutative $K$-module
algebra. The Lie bracket   of $\Der{A}$ is a derivation of the
$A$-module   $\Der{A}$, that is, it satisfies, for all
$a,a'\in A, \psi,\psi'\in \Der{A}$,
$$[\psi, a'\psi']=\psi(a')\psi'+\r_\alpha\trl a'[\r^\alpha\trl
\psi,\psi']$$
and more generally
\begin{equation}\label{moduloLie}\begin{split}
    [a \psi, a'\psi']\,=&~a \psi(a') \psi'
    +a(\r_\alpha\trl a')  [\r^\alpha\trl
    \psi,\psi']\\[.2em]
&    +\r_\beta\r_\alpha\trl a' \big(\r_\delta\r_\gamma\trl \psi'\big)\!\!\:( \r^\delta\r^\beta \trl a)\:\!
    \r^\gamma\r^\alpha\trl \psi~.\end{split}
\end{equation}
\end{prop}
\begin{rem}
The knowledge of the bracket on a $K$-submodule $\mathfrak{X}$ of
$\Der{A}$ (not necessarily a braided Lie subalgebra)
that generates $\Der{A}$ as a left $A$-module is therefore sufficient
to determine via \eqref{moduloLie} the
$K$-braided Lie algebra $\Der{A}$.
\end{rem}
\begin{proof}
Let $(P, \cdot, [~,~])$ be a $K$-braided Poisson algebra as in Remark
\ref{Poisson}. From \eqref{eqPois} and the braided antisymmetry of the bracket we have,  $[p\cdot q,
r]=p\cdot [q,r]+[p,\r_\alpha\trl r]\cdot \r^\alpha \trl q$, for
all $p,q,r\in P$. Using
this and \eqref{eqPois} again we obtain
more generally 
\begin{equation}\label{pqp'q'}\begin{split}
    [p\cdot q, p'\cdot q']\,=&~p\cdot [q,p']\cdot q'
    +[p, \r_\alpha\trl  p']\cdot \r^\alpha \trl q\cdot q'
    +\r_\beta\r_\alpha\trl p'\cdot \r^\beta \trl p\cdot [\r^\alpha\trl
    q,q']\\ &
    +\r_\beta\r_\alpha\trl p'\cdot [ \r^\beta \trl p ,\r_\gamma\trl q']
    \cdot\r^\gamma\r^\alpha\trl q~.
         \\[-1.2em]
  \end{split}
\end{equation}
We now consider the case $(P,\cdot, [~,~])=(\Hom(A,A),\circ,[~,~])$.
When $A$ is quasi-commutative the left action of $A$ on $\Der{A}$
is $a\psi=\ell_a\circ\psi$ for all $a\in A$, $\psi\in \Der{A}$, so
that $[\psi,\ell_{a'}]\circ\psi'=\ell_{\psi(a')}\circ \psi'=\psi(a')\psi'$. Quasi-commutativity of $A$ in the form
$[\ell_a,\ell_{a'}]=0$ for all $a,a'\in A$ implies
$[\ell_a,\ell_{\r_\alpha\trl a'}]\otimes \r^\alpha=0$. Then, one obtains
equation \eqref{moduloLie} setting $p=\ell_a$, $p'=\ell_{a'}$, $q=\psi$,
$q'=\psi'$  in \eqref{pqp'q'}.
\end{proof}

\begin{ex}\label{exAtr}
{\it Braided Lie algebra of a bicovariant
   differential calculus on a cotriangular Hopf algebra.}
 Let $(A,R)$ be a cotriangular Hopf algebra with dual triangular Hopf
algebra $(U,\mathscr{R})$, with 
pairing $\langle\, , \, \rangle : U\ot A \to \bbK$ such that $R(a\ot a')= \langle  \mathscr{R}, a \ot a'\rangle:=\langle
\mathscr{R}^\alpha, a\rangle\langle \mathscr{R}_\alpha, a'\rangle$,
for all $a,a' \in A$. The opposite  Hopf algebra $U^{op}$  is triangular with universal matrix $\overline{\mathscr{R}}$.
The algebra $A$ is a $U^{op}\otimes U$-module algebra via
$\trl : U^{op}\otimes U\otimes A\to A$, $(\zeta\otimes\xi)\trl
a=\langle\zeta,\one{a}\rangle\two{a}\langle\xi,\three{a}\rangle$.
 Furthermore,
$U^{op}\otimes U$ is triangular with
$$\mathfrak{R}=(\id\otimes \rm{flip}\otimes\id)(\overline{\mathscr{R}}\otimes
\mathscr{R})~.$$
It is not difficult to see  that the
property $m_{op}=R*m*\overline{R}$ of the
cotriangular Hopf algebra $(A, R)$ (see (ii) in Definition \ref{defcoqt}) is just
the quasi-commutativity property
\eqref{qcAB} of $A$ with respect to the triangular structure $\mathfrak{R}$ of $U^{op}\otimes U$. In this quasi-commutative case
the canonical notion of braided derivations of $A$
is given by equation \eqref{Der} where the braiding is the one
associated with ${\mathfrak{R}}$; we denote this braided Lie algebra by
${\rm{Der}}^{\mathfrak{R}\!\!\:}(A)$.
Hence 
${\rm{Der}}^{\mathfrak{R}\!\!\:}(A)$ is a  $U^{op}\otimes U$-braided
Lie algebra and a $(U^{op}\otimes U, A)$-relative Hopf
module, the compatibility between these structures being as in
Proposition \ref{modLie}.

As shown in \cite[\S 3.3]{LC-Aschieri}, if $U$ separates the elements
of $A$, the braided derivations
${\rm{Der}}^{\mathfrak{R}\!\!\:}(A)$ define a bicovariant differential calculus on $A$,
the unique bicovariant differential calculus compatible with the
cotriangular structure $R$ studied in \cite[\S 4.3]{Gomez-Majid}.
The braided derivations 
${\rm{Der}}^{\mathfrak{R}\!\!\:}(A)$ give the $A$-bicovariant bimodule dual
to that of one-forms.

By the fundamental theorem of bicovariant bimodules 
${\rm{Der}}^{\mathfrak{R}\!\!\:}(A)=
A\otimes {\rm{Der}}^{\mathfrak{R}\!\!\:}(A){}_{\rm{inv}}$, 
that is, braided derivations
 are freely generated as  a left
$A$-module by the right-invariant ones
\begin{equation}\label{rdifcalc}
  {\rm{Der}}^{\mathfrak{R}\!\!\:}(A){}_{\rm{inv}}:=\{u\in
{\rm{Der}}^{\mathfrak{R}\!\!\:}(A)\,|\,\Delta(u(a))= u(\one{a})\otimes
\two{a} \mbox{ for all } a\in A\}\,.
\end{equation}
We write $ {\rm{Der}}^{\mathfrak{R}\!\!\:}(A){}_{\rm{inv}}$
rather than ${\rm{Der}}^{\mathfrak{R}\!\!\:}_{\M^A}(A){}$ 
to conform with Woronowicz's notation \cite{wor}. Right-invariance implies
that the $U$-action is trivial,  
$
\xi\trl u=\varepsilon(\xi) u
\:\!
$
for all $\xi\in U$, $u\in
{\rm{Der}}^{\mathfrak{R}\!\!\:}(A){}_{\rm{inv}}$. Indeed
$\xi\trl(u(a))
=u(a)_{(1)}\hsr\xi{u(a)_{(2)}}=u(a_{(1)})\hsr\xi{a_{(2)}}
=u(\xi\trl
a)$. Also the $U^{op}$-action closes in
${\rm{Der}}^{\mathfrak{R}\!\!\:}(A){}_{\rm{inv}}$:
\begin{equation}\label{Uopaction}
  \begin{split}(\zeta\trl u)(a)&=
\one{\zeta}\trl(u(S^{-1}(\two{\zeta})\trl a))\\ &=\langle
S^{-1}(\two{\zeta}) , \one{a}\rangle\,\one{\zeta}\trl(u(\two{a}))\\
&=
\langle S^{-1}(\two{\zeta}) , \one{a}\rangle\langle
\one{\zeta},\one{u(\two{a})}\rangle\, \two{u(\two{a})}\\
&=
\langle S^{-1}(\two{\zeta}) , \one{a}\rangle\langle
\one{\zeta},u(\two{a})\rangle\, \three{a}
\end{split}
\end{equation}
for all $\zeta\in U^{op}$, $u\in
{\rm{Der}}^{\mathfrak{R}\!\!\:}(A){}_{\rm{inv}}$, $a\in A$,
and 
\begin{equation}\label{closureUop}
  \Delta((\zeta\trl u)(a))
=\langle S^{-1}(\two{\zeta}) , \one{a}\rangle\langle
\one{\zeta},{u(\two{a})}\rangle\, \three{a}\otimes\four{a}
=(\zeta\trl u)(\one{a})\otimes\two{a}~.
\end{equation}
This implies that ${\rm{Der}}^{\mathfrak{R}\!\!\:}(A){}_{\rm{inv}}$ is
a $U^{op}\otimes U$-submodule and a braided subalgebra of
${\rm{Der}}^{\mathfrak{R}\!\!\:}(A){}$. Since
\begin{equation*}
  (\mathfrak{R}_\al\trl a) \!\:\mathfrak{R}^\al\trl
u=(\overline{\mathscr{R}}_\alpha\!\!\:\trl\!\!\:
(\mathscr{R}_\beta\trl a))\;\!\overline{\mathscr{R}}^\al\!\!\:\trl\!\!\:(\mathscr{R}^\beta\trl
u)=
(\overline{\mathscr{R}}_\alpha\trl a)\:\!\overline{\mathscr{R}}^\alpha\trl u~,
\end{equation*}
we further have
\begin{equation}\label{deri=deri}
  {\rm{Der}}^{\mathfrak{R}\!\!\:}(A){}_{\rm{inv}}={\rm{Der}}^{\overline{\mathscr{R}}\!\!\:}(A){}_{\rm{inv}}~.
  \end{equation}
This  is a $U^{op} $-braided Lie algebra
isomorphism.
\end{ex}
\subsection{Actions of braided Lie algebras and universal enveloping algebras}\label{BLAD}

An action of a $K$-braided Lie algebra $\g$ on a $K$-module $W$ is a
$K$-equivariant map
$${\lie} : \g\otimes W\to W ~,~~u\otimes w\mapsto {\lie}_u(w)
$$
such that
${\lie}_{[u,v]}={\lie}_u\circ
{\lie}_v-{\lie}_{\r_\alpha\trl_V v}\circ {\lie}_{\r^\alpha\trl_Vu}$
for all $u,v\in \g$.
Equivalently, an action such that
$(\Hom(W,W ), \lie)$ is an enveloping algebra of $\g$. 

When $W=A$ is a $K$-module algebra, the action of $\g$   
is in addition required to satisfy the braided Leibniz rule:
\begin{equation}\label{leibuab}
\lie_u(ab)= \lie_u(a) b + (\r_\alpha\trl a)\,\lie_{\r_{}^\alpha \trl_\g u}(b) ~,
\end{equation}
for all
$u\in \g$, $a,b\in A$.
Equivalently, $\lie : 
\g\to \Hom(A,A)$  is required to have image in $\Der{A}$. Hence it is a morphism
$\lie: \g\to \Der{A}$ of braided Lie algebras.\\

\begin{prop}
  Let $\U(\g)$ be the universal enveloping algebra of a braided Lie
  algebra $\g$ associated with $(K,\r)$.
  Any action $\lie: \g\otimes A\to A$  of a braided Lie algebra $\g$ on a $K$-module
  algebra $A$ lifts to an action $\lie : \U(\g)\otimes A\to A$ of the braided Hopf algebra $\U(\g)$
  on $A$.
\end{prop}
  \begin{proof}  By definition of universal enveloping algebra, 
the $K$-braided Lie algebra morphism $\lie: (\g, [~,~]) \to (\Hom(A,A), [~,~])$  lifts
to a unique $K$-equivariant algebra morphism $\widehat{\lie}: (\U(\g), \cdot) \to (\Hom(A,A),\circ)$ that with
slight abuse of notation we still denote ${\lie}: \U(\g) \to
\Hom(A,A)$.

We show that condition \eqref{LAbraidedaction} holds; for all 
$\xi\in \U(\g)$,
\begin{equation}\label{Lxiableib}
\lie_\xi(ab)= \lie_{\one{\xi}}(\r_\gamma\trl a)\,\lie_{\r_{}^\gamma 
\trl \two{\xi}}(b)
\end{equation} (where we dropped the subscripts of the $K$-actions).
This is indeed the case if $\xi\in\g\subset\U(\g)$. The equality is
then proven inductively, by showing that if $\zeta, \xi\in\U(\g)$ satisfy \eqref{Lxiableib}
then so does the product $\zeta\xi$. On one hand we have
\begin{align*}
  \lie_{\zeta\xi}(ab)=\lie_\zeta(\lie_\xi(ab))&=\lie_\zeta(\lie_{\one{\xi}}(\r_\gamma\trl
                                                a)\!\:\lie_{\r^\gamma\trl\two{\xi}}(b))\\
&= \lie_{\one{\zeta}}(\r_\alpha\trl \lie_{\one{\xi}}(\r_\gamma\trl
                                                                                      a))\!\:\lie_{\r^\alpha\trl\two{\zeta}}(\lie_{\r^\gamma\trl\two{\xi}}(b))\\
       &=\lie_{\one{\zeta}}(\lie_{\r_{\one{\alpha}}\trl
         \one{\xi}}(\r_{\two{\alpha}}\r_\gamma\trl
         a))\:\!\lie_{\r^\alpha\trl\two{\zeta}}(\lie_{\r^\gamma\trl\two{\xi}}(b))\\
  &=\lie_{\one{\zeta}}\lie_{{\r_{\alpha}}\trl \one{\xi}}(\r_\beta\r_\gamma\trl  a)\;\!\lie_{\r^\beta\r^\alpha\trl\two{\zeta}}(\lie_{\r^\gamma\trl\two{\xi}}(b))\\
&=\lie^{}_{\one{\zeta}(\r_{\alpha}\trl \one{\xi})}(\r_\beta\r_\gamma\trl
                                                                                                                                                                  a)\:\!\lie_{(\r^\beta\r^\alpha\trl\two{\zeta})(\r^\gamma\trl\two{\xi})}(b)~.
\end{align*}
On the other hand, since  
$\Delta(\zeta\xi)=\Delta(\zeta)\tpr\Delta(\xi)=\one{\zeta}(\r_\alpha\trl\one{\xi})\otimes(\r^\alpha\trl\two{\zeta})\two{\xi}$, we have
\begin{align*}
  \lie_{\one{(\zeta\xi)}}(\r_\gamma\trl a)\lie_{\r^\gamma\trl
  \two{(\zeta\xi)}} (b) \,& =\lie^{}_{\one{\zeta}(\r_{\alpha}\trl
                          \one{\xi})}(\r_\gamma\trl
                          a)\:\!\lie^{}_{\r^\gamma\trl
                          ((\r^\alpha\trl\two{\zeta})\two{\xi})}(b)\\
  & =\lie^{}_{\one{\zeta}(\r_{\alpha}\trl
                          \one{\xi})}(\r_\gamma\trl
                          a)\:\!\lie^{}_{\one{(\r^\gamma}\r^\alpha\trl\two{\zeta})(\two{\r^\gamma}\trl\two{\xi})}(b)\\
  & =\lie^{}_{\one{\zeta}(\r_{\alpha}\trl \one{\xi})}(\r_\beta\r_\gamma\trl a)\:\!\lie^{}_{(\r^\beta\r^\alpha\trl\two{\zeta})(\r^\gamma\trl\two{\xi})}(b)~.
\end{align*}
Comparing these two
expressions we obtain $
\lie_{\zeta\xi}(ab)=\lie_{\one{(\zeta\xi)}}(\r_\gamma\trl
a)\:\!\lie_{\r^\gamma\trl  \two{(\zeta\xi)}} (b) \,$.
\end{proof}
When $\g=\Der{A}$ with action $\Der{A}\otimes A\to A$ given by the evaluation, we obtain: 
  \begin{cor}\label{UDAonA} Let $A$ be a $K$-module algebra and $\Der{A}$ the
    associated braided Lie algebra of derivations. 
  The universal enveloping algebra $\U(\Der{A})$ acts on $A$ as a
  $K$-braided Hopf algebra.
\end{cor}
Thus, a $K$-module algebra $A$ is canonically a braided
$\U(\Der{A})$-module algebra with respect to the braided Hopf algebra $\U(\Der{A})$.
We notice that in this  context primitive
elements $\psi \in \Der{A} \subset \U(\Der{A})$ act on $A$ as braided derivations:   
\begin{align*}
\lie_\psi(ab) = \lie_{\one{\psi}}(\r_\alpha \trl_A a) \, (\r^\alpha \trl_{\Der{A}} \lie_{\two{\psi}})(b) 
= \lie_\psi(a) b + (\r_\alpha\trl a)\,\lie_{\r_{}^\alpha \trl_{\Der{A}} \psi}(b) . 
\end{align*}

There is a canonical action of a braided Lie algebra $\g$ on the
$K$-module algebra $\U(\g)$,  the adjoint action.
\begin{prop} 
Given a braided Lie algebra $\g$ associated with $(K, \r)$, the braided adjoint action
\eqref{bLad} is an action of $\g$ on the braided Hopf algebra $\U(\g)$; equivalently, a morphism $\g\to \Der{\U(\g)}$ of braided Lie algebras.
\end{prop}
\begin{proof}
Let $\g$ be a braided Lie algebra and  $\U(\g)$ its universal enveloping algebra, a braided Hopf algebra for 
Proposition \ref{prop:uea}.
From Lemma \ref{lem:psiell},  there is a
morphism of $K$-modules:
\begin{align*}
\lie  : (\U(\g) , \trl_{\U(\g)} ) &\to (\Hom(\U(\g) ,\U(\g) ), \trl_{\Hom(\U(\g) ,\U(\g) )})  \\
\xi &\mapsto \lie_\xi , \quad \lie_\xi  (\zeta):= \xi  \tadr \zeta =\one{\xi }(\r_\alpha \trl \zeta) \, \r^\alpha \trl S_{\U(\g)} (\two{\xi }), \quad \zeta \in \U(\g) .
\end{align*}
This restricts to the $K$-module morphism 
\begin{align*}  
\lie : (\g , \trl_{\g} ) &\to (\Der{\U(\g)}, \trl_{\Der{\U(\g)}})  \nn
\\
u  &\mapsto \lie_u , \quad \lie_u  (\zeta)= u  \tadr \zeta =
u \zeta
- (\r_\alpha \trl \zeta) \, (\r^\alpha \trl  u), \quad \zeta \in \U(\g) .
\end{align*}
Indeed, the map $\lie_u$ is a braided derivation: from \eqref{bralg} the map $\tadr$ is an action, then
\begin{align*}
\lie_u (\xi \zeta) &= u \tadr (\xi \zeta)
\\
& = (\one{u}\tadr (\r_\alpha\trl
  \xi)) \, ((\r^{\alpha}\trl \two{u})\tadr \zeta) 
\\
& = (u\tadr 
  \xi) \, \zeta 
+
( \r_\alpha\trl
  \xi) \, ((\r^{\alpha}\trl u)\tadr \zeta)  
\\
& = \lie_u  (\xi) \, \zeta 
+
( \r_\alpha\trl
  \xi) \, (\lie_{\r^{\alpha}\trl u}\tadr \zeta)  
\\
& = \lie_u  (\xi) \, \zeta 
+
( \r_\alpha\trl
  \xi) \, (\r^{\alpha}\trl \lie_{u}) (\zeta)  
\end{align*}
where the last equality uses $K$-equivariance of $\lie$.
\end{proof}

\subsection{Twisting braided Lie algebras}\label{sec:tbla}
We next consider twist deformations of braided Lie algebras. Let $\g$ be a
braided Lie algebra associated with a triangular Hopf algebra
$(K,\r)$ and $\F$ be a twist for $K$. The $K_\F$-module $\g_F$
(this is just $\g$ with $K$-action; see the construction after
Proposition \ref{prop:twist}) inherits from $\g$ a twisted bracket:
\begin{prop}\label{prop:gf}
The $K_\F$-module $\g_\F$ with bilinear map 
\beq 
[~ , ~]_\F  =  \g_\F\otimes \g_\F\to \g_\F~,~~u\otimes v\to [u,v]_\F :=[\bF^\alpha \trl u,\bF_\alpha\trl  v]~.
\eeq
is a 
braided Lie algebra associated with $(K_\F, \r_\F)$.
\end{prop}
\begin{proof}
The $K_\F$-equivariance is easily proven using the analogue property for $[~,~]$ and recalling from \eqref{cop-twist} that $\Delta_{\K_\F}= \F \Delta_\K \bF$: 
for all $u,v\in \g$, $k\in K_\F$,  
$$
k\trl [u,v]_\F
= [\one{k} \bF^\alpha \trl u, \two{k} \bF_\alpha \trl  v] 
= [\bF^\alpha \fone{k} \trl  u, \bF_\alpha   \ftwo{k}\trl  v] 
=[\fone{k}\trl u,\ftwo{k}\trl v]_\F   \; .
$$
Similarly, the braided antisymmetry $[u,v]_\F=-[{\r_\F}_\beta\trl
v,{\r_\F}^\beta\trl u ]_\F$ follows  from  the braided antisymmetry of $[~,~]$ and recalling that
$\r_\F=\F_{21}\r\bF$.

We are left to prove the Jacobi identity: for all $u,v,w\in \g_\F$, we have 
\begin{align*}
[u,[v,&w]_\F]_\F
=[\bF^\beta\trl u\,,\,[\one{\bF_\beta}\bF^\alpha\trl v\,,\,\two{\bF_\beta}\bF_\alpha\trl w]]
\\
&=[[\bF^\beta\trl u\,,\,\one{\bF_\beta}\bF^\alpha\trl v]\,,\,\two{\bF_\beta}\bF_\alpha\trl w]
+ 
[\r_\gamma\one{\bF_\beta}\bF^\alpha\trl v\,,\,[\r^\gamma\bF^\beta\trl u\,,\,\two{\bF_\beta}\bF_\alpha\trl w]]
\\
&=[\one{\bF^\beta}\bF^\alpha\trl u\,,\,\two{\bF^\beta}\bF_\alpha\trl v]\,,\,\bF_\beta\trl w]]
+
[\r_\gamma\two{\bF^\beta}\bF_\alpha\trl v\,,\,[\r^\gamma\one{\bF^\beta}\bF^\alpha\trl u\,,\,\bF_\beta\trl w]] \\
&=[\bF^\beta\trl[\bF^\alpha\trl u\,,\,\bF_\alpha\trl
  v]\,,\,\bF_\beta\trl w]+
[\one{\bF^\beta}\r_\gamma\bF_\alpha\trl 
  v\,,\,[\two{\bF^\beta}\r^\gamma\bF^\alpha\trl u\,,\,\bF_\beta\trl w]]\nn\\
&=[[u,v]_\F,w]_\F+
[\one{\bF^\beta}\bF^\alpha{\r_\F}_\gamma\trl 
  v\,,\,[\two{\bF^\beta}\bF_\alpha{\r_\F}^\gamma\trl u\,,\,\bF_\beta\trl w]]\nn\\
&=[[u,v]_\F,w]_\F+
[{\bF^\beta}{\r_\F}_\gamma\trl 
  v\,,\,[\one{\bF_\beta}\bF^\alpha{\r_\F}^\gamma\trl
  u\,,\,\two{\bF_\beta}\bF_\alpha\trl w]]\nn\\
&=[[u,v]_\F,w]_\F+
[{\bF^\beta}{\r_\F}_\gamma\trl 
  v\,,\,{\bF_\beta}\trl[\bF^\alpha{\r_\F}^\gamma\trl u\,,\,\bF_\alpha\trl w]]\nn\\
&=[[u,v]_\F,w]_\F+[{\r_\F}_\gamma\trl v,
[{\r_\F}^\gamma \trl u,w]_\F ]_\F .
\end{align*}
For the first two equalities we used the $K$-equivariance and Jacobi identity for $[~,~]$, for the third the twist property \eqref{twist-bF},  for the forth the $K$-equivariance and the quasi-cocommutativity of $K$. Then  
$\r\bF=\bF_{21}\r_\F$ and the twist property end the proof.  
\end{proof}

\begin{ex}\label{ex:FbaA}
When $\g$ is a $K$-module algebra $A$ as in Lemma \ref{lem:Abla},
  its braided commutator is twisted to
  \begin{align*}
[a,b]_\F &= 
(\bF^\alpha \trl_A a ) (\bF_\alpha\trl_A  b)  - (\r_\beta \bF_\alpha \trl_A b) \, (\r^\beta \bF^\alpha  \trl_A a)
\\
&= (\bF^\alpha \trl_A a ) (\bF_\alpha\trl_A  b)  - ( \bF^\alpha \rF_\beta  \trl_A b) \, ( \bF_\alpha \rF^\beta \trl_A a)
\\
&= a \dotF b - ( \rF_\beta  \trl_A b)  \dotF (  \rF^\beta \trl_A a) . 
\end{align*}
Since $\trl_A$ and $ \trl_{A_\F}$ coincide,  this shows that the braided Lie algebra $\g_\F$ is just the twisted $\K_\F$-module algebra $A_\F$  with braided commutator defined as in \eqref{Abla} but using the universal  $R$-matrix $\rF$ of $\K_\F$ and the product in $A_\F$.
 \end{ex}

\begin{ex}\label{ex:hom}
As a particular case of the previous example, consider the $\K$-module
algebra $(\Hom(V,V), \circ)$ of linear maps  of a $K$-module $V$ with
action \eqref{action-hom}.  The twist deformation of the braided Lie algebra 
$((\Hom (V,V), \circ), [\, ,\, ])$ is the braided Lie algebra 
$((\Hom_\F (V,V), \circ_\F), [\, ,\, ]_\F)$,
that is the $\K_\F$-module algebra $(\Hom_\F (V,V), \circ_\F)$,
see \eqref{circF}, with  braided commutator
$$
[\psi,\lambda]_\F = \psi \circ_\F \lambda -  ( \rF_\alpha  \trl_{\Hom(V,V)} \lambda)  \circ_\F (  \rF^\alpha \trl_{\Hom(V,V)} \psi)\; .
$$
We prove that it is isomorphic to the braided Lie algebra $((\Hom(V_\F,V_\F), \circ)
,[\, , \,]_{\rF} )$ associated with the triangular  Hopf algebra $(K_\F, \r_\F)$.  
Indeed, the $\K_\F$-module algebra isomorphism $\dd :  (\Hom_\F(V,V) , \circ_\F)  \to (\Hom(V_\F,V_\F), \circ) $  in Proposition \ref{prop:Diso} is also a braided Lie algebra  isomorphism: 
\begin{align*}
\dd \big([\psi,\lambda]_\F \big) &=
 \dd (\psi) \circ \dd( \lambda)  -  \dd( \rF_\alpha  \trl_{\Hom(V,V)} \lambda)  \circ \dd (  \rF^\alpha \trl_{\Hom(V,V)} \psi)
\\
&= 
\dd (\psi) \circ \dd( \lambda)  - \rF_\alpha  \trl_{\Hom(V_\F,V_\F)}\dd (  \lambda)  \circ  \rF^\alpha \trl_{\Hom(V_\F,V_\F)} \dd (  \psi)
\\
&= [\dd (\psi) , \dd( \lambda) ]_{\rF} .
\end{align*}
Here we used that $\dd$ is an algebra map and a morphism of
$\K_\F$-modules.

Furthermore, from    the $K$-braided Poisson algebra structure $(\Hom (V,V),
\circ, [\, ,\, ])$, we   have that  
$(\Hom_\F (V,V), \circ_\F, [\, ,\, ]_\F)$, is a $K_\F$-braided Poisson
algebra. Indeed, the derivation property of the bracket $ [\, ,\, ]_\F$
with respect to the   multiplication $\circ_\F$ is proven
exactly as in the proof of the Jacobi identity of Proposition
\ref{prop:gf}  (just replace the
second bracket $[~,~]_\F$ with the multiplication $\circ_\F$).
Then it follows that $\dd :  (\Hom_\F(V,V) , \circ_\F, [~,~]_\F)  \to
(\Hom(V_\F,V_\F), \circ, [~,~]_{\r_\F})$ is a $K_\F$-braided Poisson
algebra isomorphism
 \end{ex}

The braided Hopf algebra $\U(\g)_\F$ associated with
$(K_\F , \r_\F)$ is given, as in Proposition \ref{prop:td-bb}, by the twist deformation 
of the universal enveloping algebra $\U(\g)$ of the braided Lie algebra $\g$.
We now show that it coincides (up to isomorphism) with the universal enveloping algebra
$\U(\g_\F)$ of the braided Lie algebra $\g_\F$, given as in Proposition \ref{prop:gf}. In other words, 
the universal enveloping algebra construction commutes with twist
deformation: $\g \to \g_\F \to \U(\g_\F)$ equals $\g \to \U(\g)\to \U(\g)_\F$.

The isomorphism 
$\U(\g_\F)\simeq \U(g)_\F$ 
is induced from the $K_\F$-module isomorphism
$$
\varphi^{(2)}:=\varphi_{\g,\g}:
\g_\F\ot_\F\g_\F\to (\g\ot\g)_\F~, \quad u\ot_\F v\mapsto \bF^\alpha\trl_\g u\ot
\bF_\alpha\trl_\g v~,
$$ 
given as in \eqref{nt-mod}, and its lifts $\varphi^{(n)}$ to $n$-th tensor products 
$\varphi^{(n)}:\g_\F^{\ot_\F n}\to (\g^{\ot n})_\F$, for $n \in \mathbb{N}$, defined recursively, 
for  $n\geq 2$, by
$$
\varphi^{(n+1)}=\varphi^{}_{\g,(\g^{\ot n})}\circ (\id_{\g_\F}\ot_\F\varphi^{(n)}) ,
$$
with the maps $\varphi^{}_{\g,(\g^{\ot n})}$ again as in \eqref{nt-mod}, and $\varphi^{(0)}:=\id_\bbK$, 
$\varphi^{(1)}:=\id_{\g_\F}$. 
Explicitly, 
\begin{multline*}
\varphi^{(n)}(u_1\ot_\F \ldots
u_{n-1}\ot_\F u_{n}) = \\ =\bF^{\alpha_1}\trl u_1\ot\bF_{\alpha_1} \trl (\ldots    
\bF^{\alpha_{n-2}}\trl u_{n-2}\ot\bF_{\alpha_{n-2}}\trl (\bF^{\alpha_{n-1}}\trl
  u_{n-1}\ot
  \bF_{\alpha_{n-1}}\trl u_n))\,, 
\end{multline*}
for all $u_i\in \g_\F$, $i=1,\ldots n$.  
Property \eqref{prop-ass-phi} of the isomorphisms $\varphi_{-,-}$, generalized to $n$-th tensor products, gives 
\begin{lem}
For integers $p,q \in \mathbb{N}$ with $p+q=n$, 
\beq\label{np+q}
\varphi^{(n)}=\varphi^{}_{\g^{\ot p},\g^{\ot q}}\circ 
(\varphi^{(p)}\ot_\F\varphi^{(q)}) \, .
\eeq
\end{lem}
\begin{proof}
By induction: the identity holds for $n\leq 2$; 
assume it holds for $n$, then 
\begin{align*}
\varphi^{(n+1)}
&=\varphi_{\g, \g^{\ot n}}
\circ
 (\id_{\g_\F}\ot_\F\varphi^{(n)})
\\
&=\varphi_{\g,\g^{\ot n}}
\circ 
  (\id_{\g_\F}\ot_\F   \varphi_{\g^{\ot p},\g^{\ot q}}
\circ 
(\varphi^{(p)}\ot_\F\varphi^{(q)}) )
\\
&=\varphi_{\g,\g^{\ot n}}
\circ 
(\id_{\g_\F}\ot_\F   \varphi_{\g^{\ot p},\g^{\ot q}})
 \circ
  (\id_{\g_\F}\ot_\F   \varphi^{(p)}\ot_\F\varphi^{(q)}) \\
&=\varphi^{}_{\g^{\ot p+1},\g^{\ot q}}
\circ 
(\varphi_{\g,\g^{\ot p}}\ot_\F   \id_{(\g^{\ot q})_\F})
\circ
  (\id_{\g_\F}\ot_\F   \varphi^{(p)}\ot_\F\varphi^{(q)}) \\
&=\varphi^{}_{\g^{\ot p+1},\g^{\ot q}}
\circ 
(\varphi^{(p+1)}\ot_\F\varphi^{(q)})~
\end{align*}
where  for the  fourth equality we have used property \eqref{prop-ass-phi} of the isomorphisms $\varphi_{-,-}$.
\end{proof}

\begin{prop}\label{prop:ugf}
 The braided Hopf algebras $\U(\g)_\F$ and $\U(\g_\F)$ associated with the triangular Hopf algebra
$(K_\F , \r_\F)$ are isomorphic.
\end{prop}

\begin{proof}
The isomorphisms 
$\varphi^{(n)}:\g_\F^{\ot_\F n}{\longrightarrow} (\g^{\ot n})_\F$ induce the
  isomorphism 
$$
{\varphi^\bullet}:=\oplus_{n\geq 0}\varphi^{(n)}:T(\g_\F)\to T(\g)_\F
$$
of (graded) $K_\F$-modules. We show it is a braided Hopf algebra isomorphism.

Let us denote by $m_{\ot_\F}$ the multiplication in the tensor algebra $T(\g_\F)$, 
and by $(m_\ot)^{}_\F$ the one in $T(\g)_\F$ given by 
$
(m_\ot)^{}_\F(\tilde\zeta,\tilde\xi):=\bF^\alpha\trl \tilde\zeta\ot
  \bF_\alpha\trl \tilde\xi
$
 on
 homogeneous elements 
  $\tilde\zeta\in (\g^{\ot n})_\F$, $\tilde\xi\in (\g^{\ot m})_\F $, as in \eqref{rmod-twist}.
Then for all
  $\zeta\in \g_\F^{\ot_\F n},\xi\in \g_\F^{\ot_\F m} $, 
\begin{align*}
(m_\ot)^{}_\F\circ (\varphi^{(n)}\times \varphi^{(m)})(\zeta,\xi)
&=
(m_\ot)^{}_\F(\varphi^{(n)}(\zeta), \varphi^{(m)}(\xi))
\\
&=\bF^\alpha\trl \varphi^{(n)}(\zeta)\ot \bF_\alpha\trl  \varphi^{(m)}(\xi)\\
&= \varphi^{}_{ \g^{\ot n}, \g^{\ot m}} \circ 
(\varphi^{(n)}\ot_\F \varphi^{(m)})(\zeta\ot_\F\xi)\\
&=
\varphi^{(n+m)}\circ m^{}_{\ot_\F}(\zeta,\xi)
\end{align*}
using \eqref{np+q} for the last equality. 
This implies $(m_\ot)^{}_\F\circ
(\varphi^{\bullet}\times \varphi^{\bullet})=\varphi^{\bullet}\circ
m^{}_{\ot_\F}$ so that $\varphi^\bullet: T(\g_\F)\to T(\g)_\F$ is
  an algebra isomorphism, in fact a $K_\F$-module algebra isomorphism.

We are left to show that $\varphi^\bullet$ is a coalgebra morphism. Since the counits are zero (but on $\bbK$), the condition 
$\varepsilon_{T(\g_\F)} = \varepsilon_{T(\g)_\F} \circ  \varphi^{\bullet}$ is trivially true. 
To show the compatibility with the coproducts, 
$(\varphi^{\bullet}\bot_\F \varphi^{\bullet}) \circ \Delta_{T(\g_\F)} = \Delta_{T(\g)_\F} \circ  \varphi^{\bullet}$,
it suffices to check it on degree zero and one elements since both sides are braided algebra maps.
Being $\varphi^{(0)}$ and $\varphi^{(1)}$ both identity maps, this reduces to show that $\Delta_{T(\g_\F)} = \Delta_{T(\g)_\F}$ on $u \in \g_\F$. But this is immediate. Indeed, 
$T(\g_\F)$ has coproduct $\Delta_{T(\g_\F)}$ given by 
$\Delta_{T(\g_\F)}(u)=u\bot_\F 1+1\bot_\F u$ for $u \in \g_\F$, (and extended as a  braided
algebra map, see Proposition \ref{prop:uea}). 
On the other hand, the coproduct on $T(\g)_\F$ is
the twist deformation as \eqref{cc-twist} of the coproduct of $T(\g)$, 
$$
\Delta_{T(\g)_\F} := \varphi_{T(\g),T(\g)}^{-1} \circ \Gamma(\Delta_{T(\g)}) : \quad
\tilde\xi \longmapsto  \fone{\tilde\xi} \bot_\F \ftwo{\tilde\xi}:=
(\F^{\alpha} \trl \one{\tilde\xi})  \bot_\F (\F_\alpha \trl \two{\tilde\xi}),
$$
which reduces to $\Delta_{T(\g)_\F}(u)=u\bot_\F 1+1\bot_\F u$ in degree one. 
 
The braided Hopf algebra isomorphism $\varphi^{\bullet}$ induces such an isomorphism 
$\U(\g_\F) \simeq \U(\g)_\F$ on the quotients.
Firstly, the ideal $T(\g_\F) \supseteq I_{\g_\F} = \langle  u\ot_\F v-{\r_\F}^{}_{\!\beta}\trl v\ot_\F{\r_\F}^\beta \trl u-[u,v]_\F \rangle$
is mapped by $\varphi^{\bullet}$ to the ideal
$T(\g)_\F \supseteq  (I_\g)_\F = \langle u\ot v-\r_\beta\trl v\ot\r^\beta\trl u-[u,v] \rangle $.
For all $u,v\in \g_\F$, since
$\varphi^{(1)}=\id_{\g_\F}$ and using $\bF_{21}\r_\F=\r \bF$,  
we compute
\begin{align*}
\varphi^\bullet (u\ot_\F v-{\r_\F}^{}_{\beta} &\trl v\ot_\F{\r_\F}^\beta \trl u-[u,v]_\F)
\\
&=\bF^\alpha\trl u\ot \bF_\alpha\trl  v - \bF^\alpha{\r_\F}^{}_{\beta} \trl v \ot \bF_\alpha{\r_\F}^\beta \trl u -[u,v]_\F
\\
&=\bF^\alpha\trl u\ot \bF_\alpha\trl v - \r_\beta\bF_\alpha\trl v\ot \r^\beta\bF^\alpha\trl u - [\bF^\alpha\trl u,\bF_\alpha\trl v] .
\end{align*} 
Since $\g_\F$ is a $K_\F$-module, the last line is a sum over the index $\alpha$ of generators
of $(I_\g)_\F$.

The algebra map property of
$\varphi^\bullet$ then implies that the ideal $I_{\g_\F}\subseteq T(\g_\F)$ 
is isomorphic to the ideal $(I_\g)_\F\subseteq T(\g)_\F$. Both
$I_{\g_\F}$ and $(I_\g)_\F$ are $K_\F$-modules, hence $\varphi^\bullet$
induces a $K_\F$-equivariant algebra isomorphism on the quotients
$$
\U(\g_\F)=T(\g_\F)/I_{\g_\F}\simeq T(\g)_\F/(I_\g)_\F~.
$$
This is actually a braided Hopf algebra isomorphism 
since both $I_{\g_\F}$ and $(I_\g)_\F$  are also coideals.
The proposition is proven by observing that as braided Hopf algebras,
$$
T(\g)_\F/(I_\g)_\F\simeq
\Big(T(\g)/ I_\g \Big)^{}_\F=\U(\g)_\F
$$
 where $T(\g)_\F/(I_\g)_\F\simeq
\Big(T(\g)/ I_\g \Big)^{}_\F$ via the identity map of $\bbK$-modules.
\end{proof}

\subsection{ Twisting  braided derivations}\label{sec:Fbd}
Let $A$ be a $\K$-module algebra  and $A_\F$ the corresponding
$\K_\F$-module algebra obtained by twisting the product as in
\eqref{rmod-twist}. 
As in Example \ref{ex:hom} there is an 
isomorphism of braided Lie algebras:
\begin{thm}\label{thm:Dalg}
The braided  Lie algebras $(\Der{A}_\F , [\, ,\, ]_\F )$ and
$(\Der{A_\F} , [\, ,\, ]_{\rF})$ are isomorphic via the map
$\dd: \Der{A}_\F \to  \Der{A_\F} $, the restriction of
  the isomorphism $\dd:  ((\Hom(A,A)_\F,
\circ_\F), [\, ,\, ]_\F)\to ((\Hom(A_\F,A_\F), \circ) ,[\, , \,]_{\rF}
)$.
\end{thm}
\begin{proof}
Since $\Der{A}_\F\subset\Hom(A,A)_\F$ and  $\Der{A_\F}\subset \Hom(A_\F,A_\F)$ are braided Lie subalgebras,
we  just need to prove that
$\dd: \Der{A}_\F \to  \Der{A_\F} $ is a bijection between the sets
$\Der{A}=\Der{A}_\F$ of $(K,\r)$-braided derivations of $A$ and
$\Der{A_\F}$ of $(K_\F,\r_\F)$-braided derivations of $A_\F$.
Recall  that a linear map $\psi\in
  \Hom(A,A)$ is a $(K,\r)$-braided derivation if and only if  \eqref{psiella0} holds:
 \begin{equation}\label{psiella}
   [\psi,\ell_a]_\r=\ell_{\psi(a)}
 \end{equation}
 for all $a\in A$. This condition is
  linear in $\psi$ and $a$ and, since $A$ and $\Der{A}$ are
$K$-modules, it implies 
\begin{equation}\label{Fpsiell}
  [\bF^\alpha\trl \psi,\ell_{\bF_\alpha\trl
    a}]_\r=\ell_{(\bF^\alpha\trl \psi)(\bF_\alpha\trl a)}
\end{equation}
for all $a\in A$. Similarly, applying $\F$ to $\F^{-1}\trl(\psi\otimes a)$ we see
  that  \eqref{psiella} is equivalent to \eqref{Fpsiell}. 
In turn, from the $K$-equivariance of $\ell$ the latter is equivalent to
  $$[\psi,\ell_{a}]_\F:=
  [\bF^\alpha\trl \psi,\bF_\alpha\trl\ell_{
    a}]_\r=\ell_{\dd(\psi)(a)}~. 
    $$
Applying the isomorphism $\dd:  ((\Hom(A,A)_\F,
\circ_\F), [\, ,\, ]_\F)\to ((\Hom(A_\F,A_\F), \circ) ,[\, , \,]_{\rF}
)$ of $K_\F$-braided Lie algebras  
shows
that \eqref{psiella} is equivalent to
\begin{equation}\label{ddbderla}
[\dd(\psi), \dd(\ell_a)]_{\rF}=\dd(\ell_{\dd(\psi)(a)})~.
\end{equation}
Thus $\psi$ is a $(K,\r)$-braided derivation if and only if \eqref{ddbderla} holds. The theorem is proven by showing that the latter is equivalent to $D(\psi)$ being a $(K_\F,\r_\F)$-braided derivation.
To this end we
evaluate  \eqref{ddbderla} on any $a'\in A$ and, using
$K$-equivariance of $\ell$, we obtain
$$\dd(\psi)\big( \dd(\ell_a)(a')\big)- (\rF_\alpha\trl
\dd(\ell_{a}))\big((\rF^\alpha\trl
\dd(\psi))(a')\big)\,=\,\dd(\psi)(a)\dotF a'$$
that using also $K_\F$-equivariance of $\dd$ is equivalent to
\begin{equation}\dd(\psi)(a\dotF a')=\dd(\psi)(a)\dotF a'+  (\rF_\alpha\trl a)\dotF(\rF^\alpha\trl
  \dd(\psi))(a')~.\label{RFbrder}
  \end{equation}
The implication  \eqref{psiella} $\Rightarrow$ \eqref{RFbrder} shows that
the bijection $\dd:  \Hom(A,A)_\F\to \Hom(A_\F,A_\F)$  restricts
to the injection $\dd: \Der{A}_\F \to  \Der{A_\F} $ (recall that
$\Hom(A,A)=\Hom_\F(A,A)$ as sets). Since any
element of $\Hom(A_\F,A_\F)$ can be written as $\dd(\psi)$ for some
$\psi\in  \Hom(A,A)_\F$, the
implication \eqref{RFbrder} $\Rightarrow$ \eqref{psiella} shows that
$\dd: \Der{A}_\F \to  \Der{A_\F} $ is a surjection.
\end{proof}

When $A$ is quasi-commutative, the $K$-braided Lie
algebra $\Der{A}$ has an $A$-module structure defined in \eqref{AmodderA}
that is compatible with the Lie bracket of $\Der{A}$ (cf. Proposition \ref{modLie}). 
This implies that also the $K_\F$-braided Lie
algebras $\Der{A}_\F$ and $\Der{A_\F}$ have compatible $A_\F$-module
structures and the isomorphism $\dd:\Der{A}_\F\to \Der{A_\F}$ maps one
into the other. 
The $(K,A)$-relative Hopf module $\Der{A}$  with $A$-module structure \eqref{AmodderA} is twisted to  the 
$(K_\F,A_\F)$-relative Hopf module $\Der{A}_\F$ with  
\begin{equation}\label{acdotFD}
  a\cdot_\F\psi:=(\bF^{\alpha}
  \trl_A a)\, (\bF_{\alpha}\trl_{\Der{A}} \psi)
\end{equation}
for all $a\in A_\F$,
$\psi\in \Der{A}_\F$, that is $a\cdot_\F\psi=\ell_a\circ_\F\psi$. The braided derivation property of the bracket
with respect to this module structure reads, for all $\psi,\psi'\in \Der{A}_\F$, $a\in A_\F$,
\begin{equation}\label{LieAmodF}
  [\psi,
a\cdot_\F\psi']_\F=[\psi,a]_\F\cdot_\F\psi'+({\r_\F}_\alpha\trl
a)\cdot_\F[{\r_\F}^\alpha\trl \psi,\psi']_\F
\end{equation}
where $[\psi,a]_\F:=(\bF^\alpha\trl
\psi)(\bF_\alpha\trl a)$ in agreement with $[\psi,\ell_a]_\F= [\bF^\alpha\trl \psi,\bF_\alpha\trl \ell_a]=\ell_{(\bF^\alpha\trl
\psi)(\bF_\alpha\trl a)}$. We recall that 
$(\Hom (A,A)_\F, \circ_\F, [\, ,\, ]_\F)$, is a $K_\F$-braided Poisson
algebra (cf. Example \ref{ex:hom}). Then the proof of \eqref{LieAmodF}
is along the proof of
Proposition \ref{modLie}  where now we consider a
  $K_\F$-braided Poisson algebra $(P_\F, \cdot_\F, [~,~]_{\F})$  and
the $K$-equivariant map $\ell: A\to \Hom(A,A)$, with $\ell_a$ defined in
\eqref{defofell}, is seen as a
$K_\F$-equivariant map
$\ell: A_\F\to \Hom(A,A)_\F$. Setting
$p=\id_{A_\F}$, $q=\psi$, $p'=\ell_a$, $q'=\psi'$ in the
$K_\F$-braided Poisson algebra version of \eqref{pqp'q'} we
obtain
\begin{equation*}\begin{split}
[\psi, a\cdot_\F\psi']_\F =[\psi, \ell_a\circ_\F \psi']_\F&=[\psi,\ell_a]_\F\circ_\F\psi'+({\r_\F}_\alpha\trl\ell_{a})\circ_\F[{\r_\F}^\alpha\trl \psi,\psi']_\F \\ &=
[\bF^\alpha\trl \psi,\ell_{\bF_\alpha\trl a}]\circ_\F\psi'+\ell_{\bF^\beta{\r_\F}_\alpha\trl
  a}\circ\bF_\beta\trl [{\r_\F}^\alpha\trl \psi,\psi']_\F \\
& =
\ell_{(\bF^\alpha\trl \psi)(\bF_\alpha\trl a)}\circ_\F\psi'+(\bF^\beta{\r_\F}_\alpha\trl
  a)\: \F_\beta\trl [{\r_\F}^\alpha\trl \psi,\psi']_\F\\ &
=
[\psi,a]_\F\cdot_\F\psi'+({\r_\F}_\alpha\trl
  a)\cdot_\F [{\r_\F}^\alpha\trl \psi,\psi']_\F~.
\end{split}\end{equation*}

This establishes the braided derivation property \eqref{LieAmodF}. More generally, setting
$p=\ell_a$, $q=\psi$, $p'=\ell_{a'}$, $q'=\psi'$ in \eqref{pqp'q'} we obtain
\begin{equation}\label{LieAmodF+}\begin{split}
  [a\cdot_\F\psi,
  a'\cdot_\F\psi']_\F=& ~a\cdot_\F[\psi,a']_\F\cdot_\F\psi'+
  a\cdot_\F({\r_\F}_\alpha\trl
  a')\cdot_\F[{\r_\F}^\alpha\trl \psi,\psi']_\F\\ &
  +{\r_\F}_\beta\:\!{\r_\F}_\alpha\trl a' \cdot_\F[{\r_\F}_\delta{\r_\F}_\gamma\trl \psi'\,,\, \r_\F{}^\delta\r_\F{}^\beta \trl a]_\F\cdot_\F
    \r_\F{}^\gamma\r_\F{}^\alpha\trl \psi~.
\end{split}\end{equation}

The above shows that we have a $(K_\F,\r_\F)$-braided Lie algebra   $\Der{A}_\F$ with compatible 
left $A_\F$-module structure \eqref{acdotFD}. In addition 
we have also that $\Der{A_\F}$ is a
$(K_\F,\r_\F)$-braided Lie algebra with compatible $A_F$-module
structure  defined by 
\begin{equation}\label{buldpsi}
(a\dotF\tilde\psi)(a'):=a\dotF\tilde\psi(a')
\end{equation}
for all $a,a'\in A_\F,\tilde\psi\in \Der{A_\F}$,  as in \eqref{AmodderA}.
Indeed
if $A$ is $K$-quasi-commutative then $A_\F$ is
$K_\F$-quasi-commutative  so that we can apply Proposition \ref{modLie} to this
case.
\begin{cor}\label{DAmodiso}
If the $K$-module algebra $A$ is quasi-commutative the braided Lie
algebra isomorphism $\dd: (\Der{A}_\F , [\, ,\, ]_\F )\rightarrow (\Der{A_\F} , [\, ,\, ]_{\rF})$ of Theorem \ref{thm:Dalg} is also an isomorphism of
the $A_\F$-modules $\Der{A}_\F$ and $\Der{A_\F}$.  
\end{cor}
\begin{proof}
  For any $a\in A_\F$ we have $\ell_a:A\to A$ and $\dd(\ell_a):A_\F\to
  A_\F$. Let $\dd(\psi)$ be an  element of $\Der{A_\F}$, we
  show that $a\dotF \dd(\psi)=\dd(\ell_a)\circ \dd(\psi)$. For any
  $a'\in A_\F$,
  \begin{equation*}\begin{split}
\dd(\ell_a)\circ
 \dd(\psi)\:\!(a')&=\dd(\ell_a)\big(\dd(\psi)(a')\big)=
 \ell_{\bF^\alpha\trl a}\big(\bF_\alpha\trl \big(\dd(\psi)(a')\big)\big)=
 a\dotF\big(\dd(\psi)(a')\big)\\ &=(a\dotF \dd(\psi))(a')~.
\end{split} \end{equation*}
We then have for any $\psi\in \Der{A}$,
$\dd(a\cdot_\F\psi)=\dd(\ell_a\circ_\F\psi)=\dd(\ell_a)\circ \dd(\psi)=a\dotF\dd(\psi)$.
\end{proof}
We denote $\tilde\psi=\dd(\psi)$ for the image of $\psi \in \Der{A}_\F$.
 When applying $\dd$ to \eqref{LieAmodF+}, we obtain the
equality, 
\begin{align}\label{LieAmodF++}
  [a\dotF\tilde\psi,
  a'\dotF\tilde\psi']_{\r_\F}=& ~a\dotF\tilde\psi(a')\dotF\tilde\psi'+
  a\dotF({\r_\F}_\alpha\trl
  a')\dotF[{\r_\F}^\alpha\trl \tilde\psi,\tilde\psi']_{\r_\F}\\ &
  +{\r_\F}_\beta{\r_\F}_\alpha\trl a' \dotF[{\r_\F}_\delta{\r_\F}_\gamma\trl \tilde\psi'\,,\, {\r_\F}^\delta{\r_\F}^\beta \trl a]_{\r_\F}\dotF
    {\r_\F}^\gamma{\r_\F}^\alpha\trl \tilde\psi \nn
\end{align}
for all $\tilde\psi$, $\tilde\psi'\in \Der{A_\F}$.
We used  that 
$[\psi,a]_\F={(\bF^\alpha\trl \psi)(\bF_\alpha\trl a)}=
\dd(\psi)(a)$.

\section{Gauge group of Hopf--Galois extensions}
The notion of gauge group of a quantum principal bundle 
$A^{co H}\subset A$ has been studied in \cite{brz-tr, brz-maj}  
as the group of invertible $H$-equivariant unital maps $A \to A$. These are not
required to be algebra maps and in the case of a commutative principal
bundle the resulting gauge group is much bigger than the usual
commutative one (see \cite[Ex. 3.1]{pgc}).

Here, along the lines of \cite{pgc}
we follow a different route to the gauge group. 
We  first study the group of unital $H$-equivariant 
{\sl algebra} maps (right $H$-comodule algebra maps) from $A$ to $A$.
Examples are provided that show when this notion is too
restrictive and when it is not. 
In the next section we introduce $K$-equivariant Hopf--Galois extensions $A^{co H}\subset A$, 
with $K$ an external triangular Hopf algebra of symmetries of the  
extension. In that context we provide a more general definition of  
gauge symmetries as braided Hopf algebras.

\subsection{Hopf--Galois extensions}
We consider noncommutative principal bundles as Hopf--Galois
extensions. These are $H$-comodule algebras $A$ with a canonically
defined map $\chi: A\ot_B A\to A\ot H$ which is required to be
invertible.  
\begin{defi}
Let $H$ be a Hopf algebra and let $A\in\A^H$ be an $H$-comodule
algebra with coaction $\delta^A$.
Consider the subalgebra $B:= A^{coH}=\big\{b\in A ~|~ \delta^A (b) = b
\otimes 1_H \big\} \subseteq A$ of coinvariant elements 
and let
$A \ot_B A$
be the corresponding balanced tensor product with multiplication $m: A \ot_B
A \to A$  induced from $m_A: A \ot A \to A$. 
 The extension $B\subseteq 
A$ is called an $H$-\textbf{Hopf--Galois extension} provided the
map
\begin{align}\label{canonical}  \can := (m \ot \id) \circ (\id \otimes_B \delta^A ) : A \otimes_B A  \longrightarrow A \ot H~  ,  \quad a' \ot_B a \longmapsto a' a_{\;(0)} \ot a_{\;(1)} ,
\end{align} 
the {\bf canonical  map}, is bijective. 
\end{defi}
\noindent
Being $\chi$ left $A$-linear, its inverse, when it exists, is determined by the translation map:
 $$\tau={\chi^{-1}}{|_{_{1 \ot {H}}}}:\,1\ot H\simeq H\to A\ot_B A \,, \quad h\mapsto
\tuno{h} \ot \tdue{{h}} . $$

A Hopf--Galois extension is {\bf cleft} if there is a convolution 
invertible morphism of $H$-comodules $j : H \to A$ (the {\bf cleaving
map}), where $H$ has coaction $\Delta$. This is equivalent to an isomorphism $A \simeq B \otimes H$ of left
$B$-modules and right $H$-comodules, where $B \otimes H$ is a left
$B$-module via multiplication on the left and a right H-comodule via $\id \otimes \Delta$.
When the cleaving map is an algebra map,
the algebra $A$ is isomorphic to a  smash product algebra $A\simeq B\sharp H$.

The extension $B=A^{co H}\subset A$ is called 
$H$-{\bf{principal comodule algebra}} if it is Hopf--Galois and 
$A$ is $H$-equivariantly projective as a left $B$-module, i.e., there
exists a left $B$-module and right $H$-comodule morphism $s : A \to B
\otimes A$ that is a section of the (restricted) product $m : B
\otimes A \to A$. 
This is equivalent to requiring that the total space algebra $A$ is
faithfully flat as a left $B$-module.
 Cleft Hopf--Galois extensions are always principal comodule algebras.

\subsection{Gauge group}
Let $\Hom_{\M^H}(A, A)$ be the $\bbK$-module of $H$-comodule maps from $A$ to $A$ and let  $\Hom_{\A^H}(A, A)$ be the subset of
algebra maps from $A$ to $A$.
The following theorem improves previous results in \cite[Prop. 3.6]{pgc} (where $A$ was quasi-commutative) and
\cite[Prop. 3.3]{HanLandi}  (where the faithfully flat
condition was used), see also \cite[Rem. 3.11]{Schneider}.
\begin{thm}\label{thm:Finv} Let $B=A^{coH}\subseteq A$ be an
  $H$-Hopf--Galois extension.
The set 
\beq\label{def1GG}
\Aut{A}:=\{ \F \in \Hom_{\A^H}(A, A) \, |  \,\, \F_{|_B}=\id\}
\eeq
of right $H$-comodule algebra morphisms that restrict to the identity
on the subalgebra $B$  is a group
with respect to the composition of maps 
$$
\F \cdot \mathsf{G}:=  \mathsf{G} \circ \F 
$$
for all  $\F , \mathsf{G} \in \Aut{A}$. For $\F\in \Aut{A}$ its 
inverse $\F^{-1} \in
\Aut{A}$  is given by
\begin{eqnarray}\label{inv-F}
\F^{-1}:=m\circ (\id\ot m )\circ (\id\ot \F\ot_B\id)\circ (\id\ot  
\tau)\circ \delta^A\: :  A&\!\longrightarrow \!& A\\ a &\!\longmapsto\!& \zero{a} \F (\tuno{\one{a}}) \tdue{\one{a}} 
~. ~~~\nn  
\end{eqnarray} 
\begin{proof}
The map $\F^{-1}$ in \eqref{inv-F} is well-defined because $\F\ot_B \id$ is well-defined
due to the right $B$-linearity of $\F$. 
The theorem is proven by showing that $\F^{-1}$ is the inverse
of $\F$. We first prove that $\F^{-1} \circ \F= \id_A$. 
Indeed, being $\F$ an 
$H$-equivariant of algebra map,  
\begin{equation}\label{FFmenouno}
\F^{-1}(\F(a))=\F(\zero{a}) \F (\tuno{\one{a}}) \tdue{\one{a}}=
\F(\zero{a}  \tuno{\one{a}}) \tdue{\one{a}}=\F(1) a=a~.
\end{equation}
Here in the last but one equality we used that $\chi$ and its inverse are left
$A$-linear  
so that  $
\zero{a}\tau(\one{a}) =\zero{a}\chi^{-1}(1\otimes
\one{a})=\chi^{-1}(\zero{a}\otimes \one{a})=
\chi^{-1}(\chi(1\otimes {a}))=1\otimes a$.

To prove that $\F \circ \F^{-1}= \id_A$ we notice that the balanced tensor
product map $\F\otimes_B\F:
A\otimes_B A\to A\otimes_B A$ is well defined because $\F$ is
 left and right $B$-linear.  We show that
\begin{equation*}
(\F\otimes_B \F)\circ \tau=\tau~.
\end{equation*}
This is equivalent to $\chi\circ
(\F\otimes_B \F)\circ \tau=1_A\otimes\id_H$; for all $h\in H$,
\begin{equation*}
\begin{split}
\chi\big(\F(\tuno{h})\otimes_B \F(\tdue{h})\big)&= \F(\tuno{h})\,\zero{\F(\tdue{h})_{}}\otimes
\one{\F(\tdue{h})_{}}\\[.2em]
&=\F(\tuno{h})\,\F(\zero{\tdue{h}})_{} \otimes
\one{\tdue{h}}\\[.2em]
&=\F(\tuno{h}\zero{\tdue{h}})\otimes
\one{\tdue{h}}\\[.2em]
&=(\F\otimes\id_H)(\chi(\tuno{h}\otimes_B\tdue{h}))\\[.2em]
&=1_A\otimes h
\end{split}
\end{equation*}
 where in the second equality we used $H$-equivariance of $\F$,  in the
 last that $\tau={\chi^{-1}}{|_{_{1 \ot {H}}}}$.

 Using that $\F$ is an algebra map and the above property       it is
 now easy to show that $\F \circ \F^{-1}=\F^{-1} \circ \F= \id_A$; for
 all $a\in A$,
 \begin{equation*}
   \begin{split}
  \F (\F^{-1}(a))&=\F\big(\zero{a} \F (\tuno{\one{a}}) \tdue{\one{a}}
  \big)=\F(\zero{a}) \F(\F (\tuno{\one{a}}))
  \F(\tdue{\one{a}})\\[.2em] &=
  \F(\zero{a}) \F (\tuno{\one{a}})\tdue{\one{a}}=\F^{-1}(\F(a))=a~.
\end{split}
\end{equation*}
where in the last but one equality we used equation \eqref{FFmenouno}.
  \end{proof}
\end{thm}

In general this notion of gauge group is too restrictive. 
For instance, a Hopf algebra $H$ coacting on itself with its coproduct, can be seen as the Hopf--Galois extension
of the ground field $\bbK$. 
In this case the gauge group
coincides with the group of  
characters of $H$. Indeed for  $\chi:H\to \bbK$ a
character, the map  $\alpha_\chi:=(\chi\otimes\id)\circ\Delta:H\to H$ is
a right $H$-comodule algebra morphism. Vice versa, given a  right
$H$-comodule algebra morphism $\alpha$, the composition 
$\chi=\varepsilon\circ\alpha$ is an algebra map. These two
constructions are one the inverse of the other. 
This equivalence holds also for locally compact quantum groups \cite{Chirvasitu}.
For example, the gauge group of the exention $
\mathbb{C}\subseteq\O_q(SL(2))$, $q^2\not=1$, is
just the multiplicative group $\mathbb{C}^\times$.

More in general, for  a faithfully flat Galois object $A$ over $\bbK$ (a Hopf--Galois
extension $\bbK=A^{co H} \subset A$) 
we have $\Aut{A}={\rm{Char}}(L)$. This is the group of 
characters of the Hopf algebra 
$L:=(A\otimes A)^{co H}$  of coinvariant elements of the $H$-comodule algebra
$A\otimes A$ under the diagonal coaction \cite[\S 3.1]{SchauenburgFIC}, \cite[\S 6.1]{HanLandi}.

The gauge group $\Aut{A}$ of an $H$-Hopf--Galois extension is in general
very small because, being a group, it is a classical object
associated to a noncommutative principal bundle.
Nonetheless in the relevant case of deformation quantization this notion is adeguate.
We show that for  twist deformations of Hopf--Galois extensions, it  
leads to a gauge group which is isomorphic to the
classical gauge group of the undeformed principal bundle.

\begin{ex}\label{btp}
For simplicity we begin
with the  Galois object
$\bbK\subseteq H$ with $H$-coaction 
given by the coproduct.
Let $H$ be  commutative with trivial cotriangular structure $R=\varepsilon\ot \varepsilon$.
Let $\cot:H\otimes H\to \bbK$ be a 2-cocycle on $H$ 
with convolution inverse $\bar\cot$
(every twist on a Hopf algebra
dual to $H$ gives a 2-cocycle on $H$ via $\cot(h\otimes h')=
\langle \F^\alpha ,h \rangle \, \langle \F_\alpha , h'\rangle$, see
e.g. \cite[App. A]{ppca}). 
As in \cite[\S 4]{pgc}
and dually to the theory in Section \ref{TbHa}, the  Hopf
algebra $(H, R=\varepsilon\ot \varepsilon)$ is deformed to the
cotriangular Hopf algebra $(H_\cot , R_\cot=\cot_{21} *\bar\cot)$. 
The total space algebra
$(A=H,\cdot,\Delta)$ is deformed as an $H$-comodule algebra to
$(A_\cot=H_{\dotcot}, \dotcot, \Delta)$, where
$h\dotcot h'=\one{h}\one{h'}\bar\cot(\two{h}\otimes
\two{h'})$ for all $h,h'\in H$.
We thus obtain the Hopf--Galois extension
$\bbK\subseteq H_{\dotcot}$. This is a cleft
extension with cleaving map $j=\id_H: H_\cot\to H_{\dotcot}$, but
in general needs not be a trivial extension since $ H_\cot$ and
$H_{\dotcot}$ are in general not isomorphic as $H_\cot$-comodule
algebras. The algebra $H_{\dotcot}$ is  quasi-commutative:
$h{\dotcot}h'=  \one{h'} \!\!\:\dotcot\one{h}~
\Ru{\two{h}}{\two{h'}}$ for all $h,h'\in H$.

 The gauge group of this noncommutative cleft extension $\bbK\subseteq H_{\dotcot} $ is
isomorphic to the gauge group of the trivial extension $\bbK\subseteq
H$. Indeed an algebra map $\F: H\to H$ that is also a right
$H$-comodule map satisfies $\Delta (\F(h))=\F(\one{h})\otimes \two{h}$
and therefore
\begin{align}
\F(h\dotcot
h')&=\F(\one{h}\one{h'})\bar\cot(\two{h}\otimes \two{h'})=
\F(\one{h})\F(\one{h'})\bar\cot(\two{h}\otimes \two{h'})\nonumber\\[.2em]
&=
\one{\F(h)}\one{\F(h')}\bar\cot(\two{\F(h)}\otimes \two{\F(h')})=
\F(h)\dotcot\F(h') \label{multpropFcot}
\end{align}
for all $h,h'\in H_{\dotcot}$. This shows that $\F$ is also an algebra map $\F:
H_{\dotcot}\to H_{\dotcot}$ and an $H_\cot$-comodule algebra
map (since the coproduct of $H_\cot$ is the same as that of $H$).

\noindent
For example, if $H=\O(SL(n))$, 
the gauge group of the Hopf--Galois extension 
$\bbK\subseteq \O(SL(n))^{}_{\dotcot}$ is the classical gauge group $SL(n)$.
 \end{ex}

This construction generalizes, thus extending the results of \cite[Prop. 4.12]{pgc} to the case where $A$ is not necessarily quasi-commutative. A right $H$-comodule algebra $A$ is deformed, via  a 2-cocycle $\cot$ on $H$, 
to the right
$H_\cot$-comodule algebra $A_\cot$ with product
\begin{equation}\label{acota'=}
a\dotcot a'=\one{a}\one{a'}\bar\cot(\two{a}\otimes
\two{a'})
\end{equation}
for all $a,a'\in A$. The subalgebra $B\subseteq A$ of $H$-coinvariant
elements is also the subalgebra of $A_\cot$ of 
$H_\cot$-coinvariant elements. Then  $B\subseteq A_\cot$ is an
$H_\cot$-Hopf--Galois extension if and only if $B\subseteq A$ is an
$H$-Hopf--Galois extension (\cite[Corol. 3.7]{ppca}).
\begin{prop}\label{isoclass} 
Let $B=A^{co H}\subseteq A$ be a Hopf--Galois extension and $\gamma$
be a 2-cocycle on $H$. The gauge group $\Aut{A_\gamma}$ of the twisted
Hopf--Galois extension $B=A_\gamma^{co H_\gamma}\subseteq A_\gamma$ is
isomorphic to the gauge group $\Aut{A}$ of the original Hopf--Galois extension. 
\end{prop}
\begin{proof}
Let $\F:A\to A$ be a unital algebra map, then the same $\bbK$-linear
map $\F$ thought as a map $\F: A_\cot\to A_\cot$ is a unital algebra
map. Indeed the unit of $A_\cot$ and of $A$ are the same. The
algebra property
$
\F(a\dotcot a' )=
\F(a)\dotcot\F(a')$ for  $a,a'\in A_\cot$
is shown as in \eqref{multpropFcot}.
\end{proof}

In particular we have the following example (cf. \cite[Ex. 4.19]{pgc}).
\begin{ex}
Consider  the Hopf--Galois extension $B=A^{co H}\subseteq A$, with $A=\O(P)$, $B=P/G$, $H=\O(G)$, of 
a  principal $G$-bundle $P\to P/G$
 of affine varieties. Its twisted deformation
 $B=A_\gamma^{co H_\gamma}\subseteq A_\gamma$
   has  gauge group isomorphic to the one of $P\to P/G$.
 \end{ex}

\section{Braided Hopf algebra of gauge transformations}\label{secondapproach}

We study the gauge symmetry
of a $K$-equivariant Hopf--Galois extension, 
for $K$ a triangular Hopf algebra. 
This is an $H$-Hopf Galois extension $B=A^{coH}\subseteq A$ with
$A$ a $K$-equivariant $H$-comodule algebra, as defined in Section
\ref{sub:Kequiv}, and  $B$ a $K$-submodule (this holds for example
when $K$ is flat as
$\bbK$-module, cf. \cite[Prop. 3.12]{ppca}). 
In this context  there is a $K$-braided Hopf algebra of (infinitesimal) gauge transformations  that
generalizes the notion of gauge group 
of the previous section.  

According to the definition of gauge group $\Aut{A}$  in \eqref{def1GG}, infinitesimal gauge
transformations of a Hopf--Galois extension $B=A^{co H}\subseteq A$
are $H$-comodule maps that are vertical derivations
\begin{align}
\auto{A}:=\{u\in\Hom(A,A)\;|\;& \delta (u(a))=u(\zero{a})\otimes \one{a},\nonumber
  \\ &u(a a')=u(a) a'+au(a')\,,\label{LieGG}\\& u(b)=0,
  \mbox{ for all } a, a'\in A, b\in B\}~.
  \nonumber
\end{align}
They form a Lie algebra with Lie bracket given by the commutator.

In the $K$-equivariant case
we can require a braided derivation property.
\begin{defi}\label{defBLieGG}
Let $(K,\r)$ be a triangular Hopf algebra. Infinitesimal gauge
transformations of a $K$-equivariant Hopf--Galois extension $B=A^{co H} \subseteq A$
are $H$-comodule maps that are braided vertical derivations
\begin{align}
\aut{A}:=\{u\in \Hom(A,A)\;|\;& \delta (u(a))=u(\zero{a})\otimes \one{a},\nonumber
  \\ &u(a a')=u(a) a'+(\r_\alpha\trl a)(\r^\alpha\trl u)(a')\,,\label{BLieGG}\\& u(b)=0,
  \mbox{ for all } a, a'\in A, b\in B\}~.
  \nonumber
\end{align}
\end{defi}

In this context  infinitesimal
gauge transformations form a 
Lie algebra object in the category of $K$-modules.
For a $(K,\r)$-module algebra $A$, 
recall  that  
${\rm{Der}}^\r(A)$ as in  \eqref{Der} denotes the braided Lie algebra of braided derivations of
$A$.
\begin{prop}\label{prop:autRB}  
  The linear space $\aut{A}$ with bracket 
      \begin{align}\label{restrictedbracket}
     [~,~]_\r:\aut{A}\otimes
        \aut{A} & \to \aut{A}
       \nn  \\
    u\otimes u' &  \mapsto
        [u,u']_\r:=u\circ u'-\r_\alpha\trl u'\circ
        \r^\alpha \trl u~,~~
      \end{align}
  for all $u,u'\in  \aut{A}$,   is a $K$-braided Lie subalgebra of
${\rm{Der}}^\r(A)$.    
\begin{proof}
We show that $\aut{A}$ is the intersection of three $K$-modules. 
Firstly, 
the  $\bbK$-submodule $\Hom_{\M^H}(A,A)\subseteq \Hom(A,A)$  of
$H$-equivariant $\bbK$-linear maps from $A$ to $A$  is a
$K$-submodule of $\Hom(A,A)$. Indeed commutativity of
the $K$-action with the $H$-coaction on $A$ is equivalent to $H$-equivariance,
$\delta(\lie_k a) =(\lie_k\otimes\,\id_A)\delta(a)$, 
of the map $\lie_k: A\to A$, $\lie_k(a)=k\trl a$
for all $k\in K$ and $a\in A$.
Then
for $u\in  \Hom_{\M^H}(A,A)$ and from $(k\trl_{\Hom(A,A)} u)(a)=
\one{k}\trl(u(S(\two{k})\trl a))$ for all $a\in A$ (cf. \eqref{action-hom}), we
have
$k\trl_{\Hom(A,A)} u =
\lie_{\one{k}}\circ\, u\circ \lie_{S(\two{k})}\in
\Hom_{\M^H}(A,A)$ since composition of $H$-equivariant
maps.

  Since $B$ is a $K$-module,  the
  $\bbK$-submodule $\Hom_B(A,A)\subseteq \Hom(A,A)$ of linear maps
  that annihilate $B$ (that is of vertical maps) is a 
  $K$-submodule of $\Hom(A,A)$. Hence
  $\aut{A}=\Hom_{\M^H}(A,A)\cap{\rm{Der}}^\r(A)\cap\Hom_B(A,A)$ is a $K$-submodule
  since intersection of $K$-submodules.

The bracket in \eqref{restrictedbracket} is just  the restriction to $\aut{A}$ of the one of ${\rm{Der}}^\r(A)$ defined in \eqref{bracket-der}.
It closes in $\aut{A}$: being the braided
Lie bracket of ${\rm{Der}}^\r(A)$ a braided commutator,  $H$-equivariance and verticality  follow. 
\end{proof}
\end{prop}

In this  braided context infinitesimal gauge transformations of a $K$-equivariant  Hopf--Galois extension $B=A^{co H}\subseteq
A$ are thus encoded in the braided Lie algebra $\aut{A}$, or
equivalently, in the braided Hopf algebra
 $\U(\aut{A})$ associated with  $\aut{A}$
( cf. Corollary \ref{UDAonA}).
We recover the Lie algebra \eqref{LieGG}
of (usual) vertical derivations when $K=\bbK$ with $\r=1\otimes 1$.

 \subsection{Quantum principal bundle over quantum homogeneous
   space}\label{QPBQHS}
 A quantum subgroup of a Hopf algebra $A$ is a  Hopf algebra $H$ together with  
a surjective bialgebra (and thus Hopf algebra) homomorphism $\pi: A \to H$. Then
$A$ is a right $H$-comodule algebra via the projection of the coproduct
\begin{equation}\label{coactAH}
\delta : = (\id \ot \pi)\circ \Delta: A \longrightarrow A \ot H~ .
\end{equation}
When
$B=A^{coH}\subseteq A$ is a Hopf--Galois extension, we call
$A$ a {\it quantum principal bundle
over the quantum homogeneous space} $B$  (see e.g. \cite[\S 5.1]{brz-maj}).

As in Example \ref{exAtr}, let $A$ be cotriangular with dual
Hopf algebra $U$ that is triangular with $\mathscr{R}\in
 U\otimes U$ 
 (here a topological tensor  product is understood if $U$ is not finite dimensional over $\bbK$).
From the $U^{op}\otimes U$-action
$\trl: U^{op}\otimes U\otimes A\to A$, $(\zeta\otimes\xi)\trl
a=\langle\zeta,\one{a}\rangle\two{a}\langle\xi,\three{a}\rangle$ the right coaction of $A$ on itself $\Delta: A\to
A\otimes A$ and the
$U^{op}$-action commute,  hence so do the right  $H$-coaction $\delta: A\to
A\otimes H$ and the  $U^{op}$-action (this is in general not the case for the $U$-action).
It follows that the
quantum principal bundle $A$
over the quantum homogeneous space $B$ is  a
$K$-equivariant Hopf--Galois extension $B=A^{co H}\subseteq A$ with $(K,\r)=(U^{op},\overline{\mathscr{R}})$.
The associated $K$-braided Lie algebra of infinitesimal gauge
transformations is $\aut{A}\subseteq \rm{Der}^\r{(A)}$.

We study the relation between
$\aut{A}$ and the braided Lie algebra
of  braided derivations
${\rm{Der}}^{\mathfrak{R}\!\!\:}(A)=A\otimes {\rm{Der}}^{\mathfrak{R}\!\!\:}(A){}_{\rm{inv}}$ of the cotriangular Hopf algebra
$A$ in Example \ref{exAtr}.
Let $(H,R_H)$ be cotriangular and 
the cotriangular structure on $A$ be given by that on $H$, that is,
$R(a\otimes a')=R_H(\pi(a)\otimes\pi(a'))$ for all $a,a'\in A$.
In terms of the triangular structure $(U_H, \mathscr{R}_H)$ dual to
$(H,R_H)$ we have $U_H\subset U$ and  $\mathscr{R}_H= \mathscr{R}\in
U_H\otimes U_H\subset U\otimes U$.
This implies $\r_H:=\overline{\mathscr{R}}_H=\overline{\mathscr{R}}$ 
so that
\begin{equation}\label{bd=}
  {\rm{Der}}^{\r_H\!\!\:}(A)={\rm{Der}}^{\r}(A)
  \end{equation}
as linear spaces of braided derivations; the first is  a  $K_H=U_H^{op}$-braided
Lie algebra, the second is a $K=U^{op}$-braided Lie algebra.
Similarly,  $\mathfrak{R}_H=\mathfrak{R}$, so that
${\rm{Der}}^{\mathfrak{R}_H\!\!\:}(A)={\rm{Der}}^{\mathfrak{R}\!\!\:}(A)$
as linear spaces; the first is a  $U_H^{op}\otimes U_H$-braided
Lie algebra, the second is a $U^{op}\otimes U$-braided one.
The linear subspace of
$H$-equivariant derivations of $A$, $${\rm{Der}}^{\mathfrak{R}_H\!\!\:}_{{\!\mathcal{M}}^H\:\!\!}(A)=\{u\in
{\rm{Der}}^{\mathfrak{R}_H\!\!\:}(A)\,|\, \delta (u(a))=u(\zero{a})\otimes \one{a}, {\mbox{ for all }} a\in A\}$$
is a  $U_H^{op}\otimes U_H$-braided Lie subalgebra of
${\rm{Der}}^{\mathfrak{R}_H\!\!\:}(A)$. Indeed
the action of $U_H$ is trivial
while the action of  $U_H^{op}$ closes in
${\rm{Der}}^{\mathfrak{R}\!\!\:}_{{\!\mathcal{M}}^H\:\!\!}(A)$, with 
proofs similar to those for 
${\rm{Der}}^{\mathfrak{R}\!\!\:}(A){}_{\rm{inv}}$ in Example \ref{exAtr}.

\begin{thm}\label{propgc}
Let $(A, R)$ be a cotriangular Hopf algebra with dual triangular Hopf
algebra $(U,\mathscr{R})$. Let 
$B=A^{coH}\subseteq A$ be 
a {quantum principal bundle
over the quantum homogeneous space} $B$,
with Hopf algebra projection $\pi : A\to H$.
Assume $(H,R_H)$ is cotriangular with 
$R = R_H \circ (\pi \ot \pi$).

Then the  braided gauge transformations $\aut{A}$ are the 
vertical braided vector fields in
\mbox{$B\otimes {\rm{Der}}^{\mathfrak{R}\!\!\:}(A){}_{\rm{inv}}$},
where
$ {\rm{Der}}^{\mathfrak{R}\!\!\:}(A){}_{\rm{inv}}$ are the
right-invariant vector fields defining the
bicovariant differential calculus on $(A,R)$.  
This linear space isomorphism is a $U^{op}_H$-braided Lie algebra
isomorphism, where $U_H \subset U$ is the triangular Hopf algebra dual to $H$.
\end{thm}

\begin{proof}
Recall from Example \ref{exAtr} that ${\rm{Der}}^{\mathfrak{R}\!\!\:}(A)=A\otimes
{\rm{Der}}^{\mathfrak{R}\!\!\:}(A){}_{\rm{inv}}$. 
Here $
{\rm{Der}}^{\mathfrak{R}\!\!\:}(A){}_{\rm{inv}}$ are the right-invariant vector fields for the
$A$-coaction (and thus for the $H$-coaction); they are $H$-equivariant. 
Then, 
$${\rm{Der}}^{\mathfrak{R}\!\!\:}_{{\!\mathcal{M}}^H\:\!\!}(A)=B\otimes {\rm{Der}}^{\mathfrak{R}\!\!\:}(A){}_{\rm{inv}}~.
$$ 
Recall the linear space equalities 
${\rm{Der}}^{\mathfrak{R}\!\!\:}(A)={\rm{Der}}^{\mathfrak{R}_H\!\!\:}(A)$
and the analogous one in \eqref{bd=}. 
From the proof of Proposition \ref{prop:autRB},  
${\rm{Der}}^{\r\!\!\:}_{{\!\mathcal{M}}^H\:\!\!}(A)={\rm{Der}}^\r(A)\cap
\Hom_{\M^H}(A,A)$
is a  $U^{op}$-braided Lie algebra and  thus 
can be seen as the  $U_H^{op}$-braided Lie algebra
${\rm{Der}}^{\r_H\!\!\:}_{{\!\mathcal{M}}^H\:\!\!}(A)={\rm{Der}}^{\r_H}(A)\cap
\Hom_{\M^H}(A,A)$. 
The equality of $U_H^{op}$-braided Lie algebras
$$
{\rm{Der}}^{\r_H\!\!\:}_{{\!\mathcal{M}}^H\:\!\!}(A)={\rm{Der}}^{\mathfrak{R}_H\!\!\:}_{{\!\mathcal{M}}^H\:\!\!}(A)
$$ is immediate since the $U_H$-action on
${\rm{Der}}^{\mathfrak{R}_H\!\!\:}_{{\!\mathcal{M}}^H\:\!\!}(A)$ is
  trivial so that the braiding ${\mathfrak{R}_H\!\!\:}$ acts
as $\r_H$. We thus have the $U^{op}_H$-braided Lie algebra isomorphism
$${\rm{Der}}^{\r_H\!\!\:}_{{\!\mathcal{M}}^H\:\!\!}(A)=B\otimes
{\rm{Der}}^{\mathfrak{R}\!\!\:}(A){}_{\rm{inv}}~.$$
The theorem is proven restricting this isomorphism to vertical vector fields.
\end{proof}
Given a cotriangular Hopf algebra $A$ the Galois object
$\bbK=A^{co A}\subseteq A$ is a quantum principal
bundle  over the quantum homogeneous space $B=\bbK$.
Here $(A,R)=(H,R_H)$ so that,
\begin{cor}\label{corgc}
Let $(A, R)$ be a cotriangular Hopf algebra with dual triangular Hopf
algebra $(U,\mathscr{R})$.
The $U^{op}$-braided Lie algebra of infinitesimal gauge trasformations ${\rm{aut}}^\r_\bbK(A)$  of the  Galois object
$\bbK=A^{co A}\subseteq A$ is
isomorphic to the $U^{op}$-braided Lie algebra ${\rm{Der}}^{\mathfrak{R}\!\!\:}(A){}_{\rm{inv}}$ of right-invariant vector fields which 
define the bicovariant differential calculus on $(A, R)$.
\end{cor}

This theorem and the corollary show that 
 Definition \ref{defBLieGG} of braided infinitesimal
gauge transformations $\aut{A}$ captures a much richer structure
--$\,$related to the bicovariant differential calculus$\,$-- than
that of the Lie algebra ${\rm{aut}}_B(A)$ in \eqref{LieGG}.

\begin{ex}\label{Uq2}
{\it Cotriangular quantum group $\O(U_q(2))$}
(the multiparametric quantum group $\O(U_{q,r}(2))$ with $r=1$, see
e.g. \cite{resh}). 
We recall that $\O(U_q(2))$ is generated by the entries of the matrix
{$T_{\!}=\!\!\:(t_{ij})\!:=$\tiny{ $
\begin{pmatrix} a & b \\ c & d\end{pmatrix},
 $}}
and by the inverse $D^{-1}$ of the quantum determinant $D=a d-q^{-1} b
c$. They satisfy the commutation relations
\begin{eqnarray*}
a b = q^{-1} b a \,\, ; \quad 
a c = q c a \,\, ; \quad  
b d = q d b \,\, ; \quad 
c d = q^{-1} d c\,\, ; \quad 
b c = q^2 c b \,\, ; \quad 
a d = d a\, \nonumber\\[.2em]
a {D}^{-1}= {D}^{-1} a \; ; \quad b {D}^{-1} = q^{-2} {D}^{-1} b
\; ; \quad c {D}^{-1} = {q}^2 {D}^{-1} c
  \; ; \quad d {D}^{-1} =  {D}^{-1} d \, .~~~
\end{eqnarray*}
The costructures are 
$\Delta(t_{ij})=\sum_{k=1,2}t_{ik} {\otimes} t_{kj}$ , $\Delta(D^{-1})=D^{-1}\otimes D^{-1}$, and 
$\varepsilon(t_{ij})=  \delta_{ij}$, $\varepsilon(D^{-1})= 1$,
while the antipode is
$ S(T)= {D}^{-1}
 {\mbox{\tiny{$\begin{pmatrix}
d & -q^{-1} b
\\ -q c & a
\end{pmatrix}$}}}
\,,~S(D^{-1}) = D \, . 
$

The $*$-structure defining the real form $\O(U_q(2))$ requires the
deformation parameter to be a phase;
the $*$-structure is then given
by 
${\mbox{\tiny{$\begin{pmatrix} a^* & b^* \\{} c^* & d^*\end{pmatrix}$}}}
=
{D}^{-1} {\mbox{\tiny{$\begin{pmatrix}
d & -q c
\\ -q^{-1} b & a
\end{pmatrix}$}}}\,, ~(D^{-1})^*=D
$.

From these relations it is easy to see that
${\rm{Aut}}_\bbK(\O(U_q(2)))={\rm{Char}}(\O(U_q(2)))=U(1)\times U(1)$
with abelian gauge Lie algebra
${\rm{aut}}_\bbK{{\O(U_q(2))}}=\mathbb{R}^2$. This is half the vector space dimension of the
commutative case ($q=1$). Definition \ref{defBLieGG}
gives on the other hand a braided gauge Lie algebra
${\rm{aut}}_\bbK^\r(\O(U_q(2)))$ that is non-abelian and 4-dimensional since
the bicovariant differential calculus on $\O(U_q(2))$ is
4-dimensional. Indeed, since $\O(U_q(2))=\O(U(2))_{\F}$ is a twist deformation of the coordinate ring on
$U(2)$ \cite{resh}, the calculus can be obtained as twist deformation of the calculus on $U(2)$,
hence from the linear space equalities
$\Der{\O(U(2))}=\Der{\O(U(2))}_\F=\Der{\O(U(2))_\F}$, cf. Theorem
\ref{thm:Dalg}, 
the linear space of vector fields has classical dimension.
The four dimensional quantum Lie algebra is explicitly
given in \cite{AC-IGLN}, Table 1 with conjugation
iii) in (4.69) there. Notice that this is not a twist deformation of
a principal bundle as in Example \eqref{btp}, rather $\O(U_q(2))$ it is a deformation of $\O(U(2))$
as a Hopf algebra.
 \end{ex}

\begin{rem}
  Braided derivations for a quantum principal bundle $A$ over a quantum
homogoneous space $B=A^{co H}$ can be more intrinsically defined in
terms of the cotriangular Hopf algebra $(A, R)$, without referring to a
dual triangular Hopf algebra $(U,\mathscr{R})$ giving $(K,\r)=(U^{op},
\overline{\mathscr{R}})$. Indeed, from  Definition \ref{Der} of braided
derivation, we have,  \begin{equation*}
  (\r_\al\trl a)\!\;(\r^\al\trl\psi)(a')
    =\bar{R}_\al(\one{a}) {R}_\beta(\two{a})\three{a}
    \!\:R^\beta(\one{a'})\bar{R}^\al(\one{\psi(\two{a'})})\,\two{\psi(\two{a'})}
\end{equation*}
for all $a, a'\in A$,
$\psi\in \Hom(A,A)$. Here
  we used the adjoint $U^{op}$-action on $\Hom(A,A)$, the properties of the universal matrix $\r$ 
and that $\bar{R}(a\otimes
a')=\langle \overline{\mathscr{R}},a\otimes a'\rangle=\langle \r,a\otimes a'\rangle$, for all $a,a'\in A$.
Therefore, for $(A,R)$ cotriangular we have the $R$-braided derivations
\begin{eqnarray*}
  {\rm{Der}}^R(A)= \{ \psi :A\to A   \,|\;\psi &\!\!\!\!\!\!\!\!\!\!\!\!\!\!\!\!\!\!\!\!\!\mbox{is  $\bbK$-linear},
                                  ~~~~~~~~~~~~~~~~~~~~~~~~~~~~~~~~~~~~~~~~~~~~~\;~~~~~~~~~~\\
  \psi(aa')=
  &\!\!\!\psi(a) a'  +\bar{R}_\al(\one{a}) {R}_\beta(\two{a})\three{a}
    R^\beta(\one{a'})\bar{R}^\al(\one{\psi(\two{a'})})\,\two{\psi(\two{a'})}\}\nonumber
\end{eqnarray*}
with infinitesimal gauge transformations 
\begin{equation*}
{\rm{aut}}_B^{R\!\!\:}(A)= \{ u\in   {\rm{Der}}^R(A)\,|\,
   \delta (u(a))=u(\zero{a})\otimes \one{a}, u(b)=0,
   \mbox{ for all } a\in A, b\in B\} .
   \end{equation*}
These are expressed only in terms of 
the cotriangular Hopf algebra $(A,R)$ and the Hopf--Galois extension $B:=A^{co
  H}\subseteq A$ with coaction given in \eqref{coactAH}.

When $A$ is finite dimensional its cotriangularity is equivalent
  to the triangularity of its dual Hopf algebra $U$ so that $
  {\rm{Der}}^R(A)= {\rm{Der}}^\r(A)$ and
  ${\rm{aut}}_B^{R\!\!\:}(A)=\aut{A}$ as defined in Definition  \ref{defBLieGG} 
with  $(K,\r)=(U^{op},{\overline{\mathscr{R}}})$.
\end{rem}

\section{Braided Hopf algebra gauge symmetry from twist deformation}
\label{BHagsftd}Given a twist $\F$ of the  Hopf algebra $K$, the $K$-equivariant $H$-comodule algebra
$A$ is twisted to the $K_\F$-equivariant $H$-comodule algebra $A_\F$
with multiplication defined in \eqref{rmod-twist}. The $K$-submodule
algebra $B=A^{co H}$ of $A$ is twisted to the $K_\F$-submodule
algebra $B_\F=A_\F^{co H}$. Furthermore  $B=A^{co H}\subseteq A$ is a
Hopf--Galois extension (principal comodule algebra)  if and only if  $B_\F=A_\F^{co H}\subseteq
A_\F$ is a Hopf--Galois extension  (principal comodule algebra), see
\cite[Corol. 3.16 (Corol. 3.19)]{ppca}.

We study
the twist deformation of infinitesimal gauge transformations when $K$
has triangular structure $\r$ using the results in \S \ref{sec:tbla}.
In this case $K_\F$ has triangular structure $\r_\F$ and the $K$-braided Lie algebra $(\aut{A}, [~,~])$
 is  twisted to the
$K_\F$-braided Lie algebra $(\aut{A}^{}_\F,[~,~]_\F)$ with bracket
$[~,~]_\F=[~,~]$ as in Proposition \ref{prop:gf}. This is a braided Lie subalgebra of
$(\rm{Der}^\r(A)_\F, [~,~]_\F)$. We can also consider the $K_\F$-braided Lie
algebra $(\autF{A_\F} ,[~,~]_{\r_\F})$:
\begin{prop}\label{autautF}
The $(K_\F,\r_\F)$-braided Lie algebras $(\aut{A}^{}_\F,[~,~]_\F)$ and
$(\autF{A_\F},$ $[~,~]_{\r_\F})$ are isomorphic via the map 
$$
\dd:
\aut{A}^{}_\F{\, \longrightarrow\,}\autF{A_\F} ,
$$
 the restriction of the isomorphism $\dd: (\rm{Der}^\r{(A)}_\F,
[~,~]_\F){\, \to\,} (\rm{Der}^{\r_\F}{(A_\F)}, [~,~]_{\r_\F})$ of braided
Lie algebras, in Theorem \ref{thm:Dalg}, to $\aut{A}^{}_\F\subseteq \rm{Der}^\r(A)_\F $.
\end{prop}
\begin{proof}
Since $B_\F$ is a $K_\F$-module (and a $K$-module since
the $K_\F$ and $K$ algebra structures coincide), the isomorphism $\dd$ maps vertical derivations into vertical
derivations: $\dd(u)(b)=0$ for any $u\in  \aut{A}^{}_\F$ and $b\in B_\F$.  Equivariance of $D(u)$ under the
$H$-coaction is straighforward from $H$-equivariance of $u$ and the
compatibility of the $K$-action (or the $K_\F$-action) with  the
$H$-coaction (cf. \eqref{compatib}). Verticality and $H$-equivariance
of $\dd^{-1}(u_\F)$ for any $u_\F\in\autF{A_\F}$ are similarly proven.
As a consequence $\dd:\aut{A}^{}_\F{\, \to \,}\autF{A_\F}\subseteq\rm{Der}^{\r_\F}{(A_\F)}$ is
injective and surjective.
\end{proof}

We can also consider 
the $(K,\r)$-braided Hopf algebra 
 $\U(\aut{A})$. Under a twist $\F$ of $K$ we have the $(K_\F,\r_\F)$-braided Hopf algebra
$\U(\aut{A})_\F$: 
\begin{cor}\label{UautUautF}
The twist $\U(\aut{A})_\F$ of the $(K,\r)$-braided Hopf algebra $\U(\aut{A})$   of the Hopf--Galois extension $B=A^{co H}\subseteq
A$ is isomorphic to the $(K_\F,\r_\F)$-braided Hopf algebra
$\U(\autF{A_\F})$  of the Hopf--Galois extension 
$B_\F=A_\F^{co H}\subseteq
A_\F$.
  \end{cor}\begin{proof}
  From Proposition \ref{prop:ugf} the braided Hopf algebra $\U(\aut{A})_\F$ is
 isomorphic to $\U(\aut{A}_\F)$. As a corollary of Proposition
 \ref{autautF} this latter is isomorphic to the  $(K_\F,\r_\F)$-braided Hopf
 algebra  $\U(\autF{A_\F})$. 
\end{proof}

\subsection{Gauge transformations of twisted principal bundles}
  
Braided Lie algebra and Hopf algebra gauge symmetries 
 arise in particular  by considering twist deformations of
commutative principal bundles of affine varieties. 
In this case $K=\U({\rm{Lie}}(L))$ is a cocommutative Hopf algebra and $H=\O(G)$ a commutative
Hopf algebra with $L$ and $G$ affine algebraic groups and $\O(G)$ the
coordinate ring of $G$. The $L$-equivariant principal
bundle of affine varieties $P\to M=P/G$ gives the $(K,\r=1\ot 1)$-equivariant Hopf--Galois extension $B=A^{co H}\subseteq
A$, where $A=\O(P)$, $B=\O(M)$ are the coordinate rings of  
$P$ and $M$. The Lie algebra $\auto{A}$ and the Hopf algebra
$\U(\auto{A})$ of infinitesimal gauge transformations  are
braided Lie and Hopf algebras with trivial braiding $\r=1\otimes 1$. Under a twist $\F$ of $K$   
we obtain the $(K_\F,\r_\F=\F_{21}\F^{-1})$-braided Lie and Hopf algebras
$\autF{A_\F}$ and $\U(\autF{A_\F})$  of the Hopf--Galois extension 
$B_\F=A_\F^{co H}\subseteq A_\F$.

Due to the isomorphisms
$\aut{A}_\F\simeq \autF{A_\F}$ and $\U(\aut{A})_\F\simeq\U(\autF{A_\F})$  
 of
Proposition \ref{autautF} and Corollary \ref{UautUautF}  
these braided Lie and Hopf algebras coincide as linear spaces
with  the initial Lie and Hopf algebras  $\auto{A}$ and
$\U(\auto{A})$.
This has to be compared with the (unbraided) gauge Lie
and Hopf algebras
$\auto{A_\F}$ and $\U(\auto{A_\F})$ as defined in \eqref{LieGG} which
are usually much smaller as vector spaces (cf. Example
\ref{Uq2}).

In the following we illustrate this constructions with a few explicit examples considering
twists arising from actions of tori, $L=\mathbb{T}^2$ (abelian twists),
and from actions of $L=\mathbb{R}_{>0}\ltimes \mathbb{R}$,
the affine group $ax+b$ of  the real line
(Jordanian twists).

Let us start with a $\mathbb{T}^2$-equivariant principal bundle $P\to M=P/G$ 
as before. The universal enveloping algebra $K$ of
$\mathbb{T}^2$ is  
generated by two commuting elements $H_1,H_2$. Due to the torus
action, these act as derivations on $A$ and on $B$. The twist
\begin{equation}\label{twistT2}
  \F=e^{{\pi}i
  \theta(H_1\ot H_2-H_2\ot H_1)} 
\end{equation}
of $K$, with $\theta\in \mathbb{R}$,  then defines the algebras $A_\theta:=A_\F$ and $B_\theta:=B_\F$
with multiplication given in \eqref{rmod-twist}. The element $\F$
actually belongs to a topological completion of the algebraic tensor
product $K\otimes K$ which does not play a role here, see Remark \ref{rem:topological}.
The algebras $A_\theta$ and $B_\theta$ are equivalently obtained considering the $2$-cocycle
$\gamma:  \mathcal{O}(\mathbb{T}^2)\otimes
\mathcal{O}(\mathbb{T}^2)\to \mathbb{C}$, $t\otimes t' \mapsto\gamma(t\otimes t'):=\langle \F, t\otimes t' \rangle$ and the $2$-cocycle
deformation of $B$ via the coaction of
$\O(\mathbb{T}^2)$ dual to the action of $K$.

\begin{ex}{\it{Trivial  bundle.}}
Consider
the trivial Hopf--Galois extension $B\subseteq B\otimes H$, with
$B=\O(M)$ and $H=\O(G)$ the coordinate algebras of an algebraic affine
variety $M$ and  a group $G$. The gauge
group of this  bundle consists of affine variety maps $\mathbb{T}^2\to G$, that
is, algebra maps $H\to B$. 
The Lie algebra of infinitesimal gauge transformations is the free $B$-module 
\beq\label{autbh}
{\rm{aut}}_B(B\otimes H)=B\otimes \rm{Lie}(G)
\eeq
with
$\rm{Lie}(G)$ the  right-invariant vector fields on $G$. Equivalently \eqref{autbh} is the Lie algebra of vertical derivations of $A=B\otimes H$.
If $\{\chi^i\}$ ($i=1,..., \rm{dim } (G)$) is a basis of the tangent space $T_eG$ at the unit
of $G$, then $\{X^i \}$, with $X^i(h) = \chi^i\trl
h=\chi^i(\one{h})\two{h}$, $h\in H$, is a basis of right-invariant
vector fields on $G$.
The Lie bracket is $[b_iX^i,b'_jX^j]=b_ib'_j[X^i,X^j]$, for all $b_i,b_j'\in B$.

Let   $L=\mathbb{T}^2$ act on $M$. With the twist  \eqref{twistT2} 
the  Hopf--Galois extension $B\subseteq B\otimes H$, is 
deformed to  $B_\theta\subseteq A_\theta=B_\theta\otimes H$. 
Here we use $\dott$ to denote the multiplications $\dotF$ in $B_\F=B_\theta$
and $A_\F=A_\theta$.
The corresponding infinitesimal gauge transformations form the braided Lie algebra and free $B_\theta$-module
\begin{equation}\label{autBLieG}
  {\rm{aut}}^{\r_\F}_{B_\theta}(B_\theta\otimes H)=
  \dd({\rm{aut}}_{B}(B\otimes H)^{}_\theta)=
  B_\theta\otimes \rm{Lie}(G)~,
  \end{equation}
where $\r_\F=\F_{21}\bF=\bF^2$. 
For the last equality we first observe that $\dd(X^i)=X^i$ for $X^i\in \rm{Lie}(G)\subseteq
\Der{A}$. Indeed, the twist acts only on $B$ while the right-invariant vector field
 $X^i$  acts only on $H$ and so, for all $b\otimes h\in A_\theta=B_\theta\otimes H$,
\begin{equation*}\begin{split}
    \dd(X^i)(b\otimes h)
&=\one{\bF^\alpha} \trl (X^i(S(\two{\bF^\alpha}){\bF_\alpha}\trl(b\otimes h)))\\
&=\one{\bF^\alpha} \trl
(S(\two{\bF^\alpha}){\bF_\alpha}\trl b\,\otimes X^i(h))\\ &=b\otimes
X^i(h) = X^i(b\otimes h)~.
\end{split}
\end{equation*}
Then $[X^i,X^j]_{\r_\F}=\dd([X^i,X^j]_\F)=[X^i,X^j]$ since $[X^i,X^j]_\F=[X^i,X^j]$.
From
\eqref{LieAmodF++}  the braided Lie bracket (defined in
\eqref{restrictedbracket}) explicitly reads, for all $b_i, b'_i\in B$,
$$
[b_i\dott X^i,b'_j\dott X^j]_{\r_\F}=b_i\dott b'_j\dott[X^i,X^j] ,
$$
being $(b_i\dott X^i)(a)=b_i\dott X^i(a)$
for all $a\in (B_\theta\otimes H)$, cf. \eqref{buldpsi}.

From Proposition \ref{prop:uea}  
  the universal enveloping algebra
  ${{\U}}(B_\theta\otimes {\rm{Lie}}(G))$ of this
braided Lie algebra is a $K_\F$-braided Hopf algebra (since $K$ is
abelian $K_\F$ concides with  $K$ as algebra but has triangular structure $\r_\F$).
The coproduct, counit and antipode are uniquely determined on the generators
$b\dott {X}\in B_\theta\otimes {\rm{Lie}}(G)$ as
$$\Delta(b\dott{X})=b\dott{X}\,\bot\, 1+ 1\,\bot\, b\dott{X}~,~~\varepsilon(b\dott{X})=0~,~~S(b\dott{X})=-\;\!b\dott{X}~.
$$

We briefly compare this result with the approach to the gauge group
defined in equation \eqref{def1GG} and based on algebra maps (not considering twist deformations of the structure Hopf algebras as
in Proposition \ref{isoclass}). An algebra map $\FP: {B_\theta\ot H}\to
{B_\theta\ot H}$ which is the identity on
$B_\theta\ot 1$ is determined by
its restriction to $\FP_{1\ot H}:1\ot H\to B_\theta\ot H$.  
Using the algebra map property of $\FP$ 
we further have, for all $b\in B_\theta, h\in H$, 
$$\FP(1\ot h)\:(b\ot 1) =\FP(1\ot h)\FP(b\ot 1)=
\FP((1\ot h)(b\ot 1))=\FP((b\ot 1)(1\ot h))=(b\ot 1)\FP(1\ot h)$$
that shows that $\FP_{1\ot H}:1\ot H\to Z(B_\theta)\ot H$, where $Z(B_\theta)$
is the center of $B_\theta$. This is drastically different from the
commutative case. For example if $M=\mathbb{T}^2$ then
$B_\theta=\O(\mathbb{T}^2_\theta)$ and, for $\theta$ irrational,
$Z(B_\theta)=\mathbb{C}$ and the gauge group is just the structure
group $G$. 
  \end{ex}

 \begin{ex}{\it{The instanton bundle on the sphere $S^{4}_\theta$}.}
The spheres $S^7$ and $S^4$ are the homogeneous spaces
$S^7=Spin(5)/SU(2)$ and
$S^4=Spin(5)/Spin(4)\simeq$ $ Spin(5)/SU(2)\times SU(2)$, hence the Hopf fibration 
$S^7\to S^4$ is a $Spin(5)$-equivariant $SU(2)$-principal bundle. Then,  
the right-invariant vector fields $X\in so(5)\simeq
spin(5)$ on  $Spin(5)$  project to the
right cosets $S^7$ and $S^4$ and generate the $\mathcal{O}(S^7)$-module of
vector fields on $S^7$ and the $\mathcal{O}(S^4)$-module of vector
fields on $S^4$.   This latter is the submodule of
right $SU(2)$-invariant vector fields on $S^7$.
A convenient generating set for the $\mathcal{O}(S^7)$-module 
 is given by the following right  SU(2)-invariant vector fields (cf. \cite{L06}):
\begin{align}\label{der-sopra1}
&H_1= \tfrac{1}{2}( z_1 \partial_1 - z_1^* \partial^*_1 - z_2 \partial_2  + z_2^* \partial^*_2 - z_3 \partial_3 + z_3^* \partial^*_3 + z_4 \partial_4  - z_4^* \partial^*_4) 
\nn
\\
&H_2 =  \tfrac{1}{2}(- z_1 \partial_1 + z_1^* \partial^*_1 + z_2 \partial_2  - z_2^* \partial^*_2 - z_3 \partial_3 + z_3^* \partial^*_3 + z_4 \partial_4  - z_4^* \partial^*_4)
\end{align}
\begin{align}\label{der-sopra2}
& E_{10} =  \stwo (z_1 \partial_3  - z_3^* \partial^*_1 - z_4 \partial_2 + z_2^* \partial^*_4)
&&  
E_{-10} = \stwo (z_3 \partial_1 - z_1^* \partial^*_3 - z_2 \partial_4  + z_4^* \partial^*_2)
\nn
\\
& E_{01} = \stwo ( z_2 \partial_3  - z_3^* \partial^*_2 + z_4 \partial_1 - z_1^* \partial^*_4) 
&&  
E_{0 -1} = \stwo (z_1 \partial_4  - z_4^* \partial^*_1 + z_3 \partial_2  - z_2^* \partial^*_3)
\nn
\\
& E_{11}= - z_4 \partial_3  + z_3^* \partial^*_4 
&&  
E_{-1-1} = z_4^* \partial^*_3  - z_3 \partial_4 
\nn
\\
& E_{1-1}= - z_1 \partial_2  + z_2^* \partial^*_1  
&&  
E_{-1 1}= - z_2 \partial_1  + z_1^* \partial^*_2 . 
\end{align}
Here $z_1,z_2,z_3,z_4$ generate the coordinate ring of $S^7$ with
$z_1z^*_1+z_2z^*_2+z_3z^*_3+z_4z^*_4=1$, the partial derivatives
$\partial_\mu, \partial_\mu^*$,  are defined by  $\partial_\mu(z_\nu) 
= \delta_{\mu \nu}$ and $\partial_\mu(z_\nu^*)=0$  and 
similarly  for  $\partial_\mu^*$, $\mu, \nu=1,2,3,4$. The above vector fields are chosen so that their commutators close the Lie algebra 
$so(5)$ in the form
\begin{align}\label{so5}
[H_1,H_2] &=0 \; ; \quad
[H_j,E_\textsf{r}] = r_j E_\textsf{r} \; ; \nn \\  
[E_\textsf{r},E_{-\textsf{r}}] &= r_1 H_1 + r_2 H_2  \; ; \quad
[E_\textsf{r}, E_\textsf{s}]= N_{rs } E_\textsf{r+s} \; .
\end{align}
Then, the elements 
$H_j$, $j=1,2$, are the generators of the Cartan subalgebra, 
and $E_\textsf{r}$ is labelled by  
$\textsf{r}=(r_1,r_2)  \in \Gamma= \{(\pm 1 , 0), (0, \pm 1), (\pm1,\pm
1)\}$, one of the eight roots. Also,  $N_\textsf{rs}=0$ if
$\textsf{r+s}$ is not a root  and $N_\textsf{rs} \in \{1,-1\}$
otherwise.
The $*$-structure is given by $H_j^*=H_j$ and $E_\textsf{r}
^*=E_{-\textsf{r}}$. Since the $*$-structure on 
vector fields $X$ is defined by $X^*(f)=-(X(f^*))^*$ for any function $f$,
one accordingly checks that for the vector fields in \eqref{der-sopra1} and \eqref{der-sopra2},  $E_{-\textsf{r}} (z_\mu) 
=-(E_{\textsf{r} }(z^*_\mu))^*$ and $H_j (z_\mu)=-(H_j(z^*_\mu))^*$.
The general 
right $SU(2)$-invariant vector field on $S^7$, that is  
$H$-equivariant derivation of $\O(S^7)$, is 
\beq\label{Xreal}
X= b_1 H_1 + b_2 H_2  + {\mbox{$\sum_\textsf{r}$}}  b_\textsf{r}  E_\textsf{r}   \; 
\eeq
with $b_j, b_\textsf{r}  \in \mathcal{O}(S^4)$.
The derivations $X$ are real,  $X^*=X$,
if and only if $b_j^*=b^{}_j$ and
$b_\textsf{r}^*=b^{}_{-\textsf{r}}$.

Infinitesimal gauge transformations are $H$-equivariant derivations $X$ 
which are  vertical: $X (b)=0$ for  $b \in \O(S^4)$. 
As shown in \cite[\S 4]{pgcex}, they form the $\O(S^4 )$-module generated by 
 \begin{align}
K_1&= 2x H_2 + \beta^* \sqrt{2} E_{01} + \beta \sqrt{2} E_{0-1}  &
K_2&= 2x H_1 + \alpha^* \sqrt{2} E_{10}  + \alpha \sqrt{2} E_{-10} 
\nn \nonumber\\
W_{01}&=   \sqrt{2} \big( \beta H_1 + \alpha^*  E_{11} +  \alpha
          E_{-1 1} \big) &
  W_{0-1}&=   \sqrt{2} \big( \beta^* H_1 +  \alpha^* E_{1-1} + \alpha E_{-1-1} \big)
\nn \nonumber\\
 W_{10}&=  \sqrt{2} \big( \alpha H_2 - \beta^*   E_{11}  + \beta E_{1-1} \big)
&
  W_{-10}&=  \sqrt{2} \big( \alpha^*  H_2 + \beta^* E_{-11}  - \beta E_{-1-1} \big)
\nn \nonumber\\
 W_{11}&=   2x  E_{11} + \alpha \sqrt{2}E_{01}  - \beta \sqrt{2}E_{10}  
& W_{-1-1}&= 2x E_{-1-1}  + \alpha^* \sqrt{2}E_{0-1}  - \beta^* \sqrt{2}E_{-10}
\nn \nonumber\\
W_{1 -1}&=  -2x E_{1-1}  + \beta^* \sqrt{2}E_{10} + \alpha \sqrt{2}E_{0-1}
 & W_{-1 1}&=   -2x E_{-11} + \beta \sqrt{2}E_{-10} + \alpha^* \sqrt{2}E_{01}  
 \label{KW}\end{align}
where 
$\alpha= 2(z_1 z_3^* + z^*_2 z_4)$,
$\beta= 2(z_2 z_3^* - z^*_1 z_4)$,
$x= z_1 z_1^* + z_2 z_2^* - z_3 z_3^* -z_4 z_4^*$ 
satisfy the relation
$
\alpha^* \alpha + \beta^* \beta + x^2=1 
$ 
and  are the coordinates of $\O(S^4)\subset \O(S^7)$.
The commutators of the generators in \eqref{KW} define the gauge Lie algebra since  $[bX,b'X']=bb'[X,X']$ for any $b,b' \in \O(S^4)$ and $X,X' \in \mathrm{aut}_{\O(S^4)}(\O(S^7))$. 
These generators
satisfy  $K_j^*=K_j$, $W^*_\textsf{r}  =W_{-\textsf{r}}$.
They transform under the adjoint representation 
of $so(5)$,
in particular 
\begin{equation}\label{adjSO}
 H_j \trl K_l = [H_j, K_l]=0 \, , \quad H_j \trl W_\textsf{r}= [H_j, W_\textsf{r}] = r_j W_\textsf{r}  \, .
\end{equation}
Due to $Spin(5)$-equivariance we have the decomposition, see
\cite[\S 4]{pgcex},
\begin{equation}\label{dsd}
\mathrm{aut}_{\O(S^4) }(\O(S^7))=\mbox{$\bigoplus_{n\in \mathbb{N}_0} $}\, [d(2,n)]
\end{equation}
where $[d(2,n)]$ is the representation of $so(5)$ as derivations on
$\O(S^7)$
of highest weight vector $\alpha^n W_{11}$ of weight $(n+1,1)$
 and dimension 
$
d(2,n)=\tfrac{1}{2}(n+1)(n+4)(2n+5)
$. These derivations
are combinations of the derivations
 in \eqref{KW} with spherical harmonics of degree $n$ on $S^4$
 (harmonic polynomials on $\mathbb{R}^5$ of homogenous degree $n$).
\\

The right-invariant vector fields
$H_1$ and $H_2$ of $Spin(5)$  are the vector fields of
a maximal torus subgroup $\mathbb{T}^2 \subset Spin(5)$. They define the universal enveloping algebra
$K$ of the abelian Lie algebra $[H_1,H_2]=0$. Their action \eqref{der-sopra1} on  $\O(S^7)$
commutes with the $\O(SU(2))$ right
coaction on $\O(S^7)$. 
The associated twist in \eqref{twistT2}
corresponds to the torus 2-cocycle of \cite[Ex. 3.21]{ppca}, hence
it gives the $\mathcal{O}(SU(2))$-Hopf--Galois extension $\mathcal{O}(S^4_{\theta})
=\mathcal{O}(S^7_\theta)^{co\O(SU(2))} \subset
\mathcal{O}(S^7_\theta)$ introduced in \cite{LS0}.
 In the following we use the subscript $\theta$
  instead of $\F$ for twisted algebras and their multplications to
  conform with the literature.
  The twisted algebra $\mathcal{O}(S^4_\theta)$ is generated by
  $\alpha, \alpha^*, \beta, \beta^*, x$ with 
  $\alpha\dott \alpha^*+\beta\dott \beta^*+x\dott  x=1$ where the only nontrivial
  commutation relations are
  \begin{equation}\label{t4sphererel}
    \alpha\dott \beta=e^{2\pi  i\theta} \beta\dott\!\: \alpha~,~~\alpha\dott \beta^*=e^{-2\pi  i\theta} \beta^*\!\!\:\dott \!\:\alpha 
 \end{equation}
 and their complex conjugates. They are 
obtained from the general definition of twisted multiplication
\eqref{rmod-twist} using that the coordinates
$\alpha, \beta, x$ are eigen-functions of $H_1$ and $H_2$.

The twist deformation
of the gauge Lie algebra
$\mathrm{aut}_{\mathcal{O}(S^4)}(\mathcal{O}(S^7))$ via the
action in \eqref{adjSO}
is the
$\mathcal{O}(S^4_\theta)$-module and braided Lie algebra
$\big(\mathrm{aut}_{\mathcal{O}(S^4)}(\mathcal{O}(S^7) )_\F, [~,~]_\F,\cdot_\F\big)$
associated with $(K_\F, \r_\F= \bF^2)$.
It has braided Lie bracket  on generators 
 (see Proposition \ref{prop:gf}):
\begin{align}\label{aut-twist}
[K_1,K_2]_\F &=[K_1,K_2] \; ; \quad [K_j,W_\textsf{r}]_\F= [K_j,W_\textsf{r}] \; ; \nn \\
[W_\textsf{r}, W_\textsf{s}]_\F &= e^{- i \pi \theta \wdg{r}{s} } [W_\textsf{r}, W_\textsf{s}] \; .
\end{align}
  From Proposition \ref{autautF} we have
  $\mathrm{aut}_{\mathcal{O}(S^4_\theta) }(\mathcal{O}(S^7_\theta)
  )=\dd(\mathrm{aut}_{\mathcal{O}(S^4) }(\mathcal{O}(S^7) )_\F)$ and
from Corollary \ref{DAmodiso} and equation \eqref{LieAmodF++}
we can conclude that (we denote the multiplications $\dotF$ in $B_\F=B_\theta$
and $A_\F=A_\theta$, and the module structures by $\dott$): 
\begin{prop}\label{gaugeinstanton}
The braided Lie algebra $\mathrm{aut}_{\mathcal{O}(S^4_\theta) }(\mathcal{O}(S^7_\theta) )$ of infinitesimal gauge transformations of the  $\mathcal{O}(SU(2))$-Hopf--Galois extension 
$\mathcal{O}(S^4_\theta) \subseteq \mathcal{O}(S^7_\theta)$
is generated, as an  $\mathcal{O}(S^4_\theta)$-module, by the elements
\begin{equation}\label{Dgenerators}
\widetilde{K}_j :=\dd (K_j)=K_j \, , \quad  \widetilde{W}_\mathsf{r}
:= \dd(W_\mathsf{r})=W_\mathsf{r} \,e^{\pi i\theta({r_1}H_2-r_2H_1)}\;
, \quad j=1,2 \, , \quad \mathsf{r} \in \Gamma
\end{equation}
with bracket 
\begin{align*}
[\widetilde{K}_1,\widetilde{K}_2]_\rF &=\dd([{K}_1,{K}_2]) \; ; \quad
[\widetilde{K}_j,\widetilde{W}_\mathsf{r}]_\rF = \dd([{K}_j,{W}_\mathsf{r}]) \; ;\nonumber \\
[\widetilde{W}_\mathsf{r}, \widetilde{W}_\mathsf{s}]_\rF &= e^{- i \pi \theta \mathsf{r}\wedge\mathsf{s}  } \, \dd([{W}_\mathsf{r}, {W}_\mathsf{s}]) \; .
\end{align*}
The braided Lie bracket of generic elements $\widetilde{X},\widetilde{X}'$ in $\mathrm{aut}_{\mathcal{O}(S^4_\theta)}(\mathcal{O}(S^7_\theta) )$ 
and $b,b'\in \mathcal{O}(S^4_\theta)$ 
is given by
\begin{equation}\label{Lierelations}
    [b \dott \widetilde{X}, b'\dott \widetilde{X}']_{\r_\F}\,=~b\dott (\r_\alpha\trl b')  \dott[\r^\alpha\trl
    \widetilde{X},\widetilde{X}']_{\r_\F}~   
  \end{equation}
with $\mathcal{O}(S^4_\theta)$-module structure as in \eqref{buldpsi}. 
\end{prop}
In \cite{pgcex} the Lie brackets $[K_j,W_\textsf{r}]$ and
$[W_\textsf{r} ,W_\textsf{s}]$ are found.
We now compute the
 brackets of the generators $\widetilde K_j, \widetilde W_\textsf{r}$
 of the braided gauge Lie algebra
 $\mathrm{aut}_{\mathcal{O}(S^4_\theta) }(\mathcal{O}(S^7_\theta) )$. 
For example, the bracket
$[W_{-1-1},W_{01}]=     \sqrt{2} \beta W_{-1-1} - \sqrt{2}
\alpha^*(K_1+K_2 )$ is the same as 
$e^{-\pi i\theta}[W_{-1-1},W_{01}]_\F=     e^{\pi i\theta}\sqrt{2}
\beta \cdot_\F W_{-1-1} - \sqrt{2} \alpha^*\cdot_\F(K_1+K_2 )$, where we used the module structure of
$\mathrm{aut}_{\mathcal{O}(S^4) }(\mathcal{O}(S^7) )_\F$ given by
 \eqref{acdotFD} and that the coordinates of the sphere $S^4$
are eigen-functions of $H_1$ and $H_2$. Then, 
by applying the map $\dd$ we have that $[\widetilde W_{-1-1},\widetilde
W_{01}]_{\r_\F}=  e^{2\pi i\theta}\sqrt{2} \beta \dott\widetilde
W_{-1-1} -  e^{\pi i\theta}\sqrt{2}
\alpha^*\!\!\:\dott(\widetilde K_1+\widetilde K_2 )$. 
 We list the
independent braided Lie algebra brackets in the following table:
 \begin{align*}
& [\widetilde K_1, \widetilde K_2]_{\r_\F} = \sqrt{2} ( \alpha^*  \!\dott\widetilde W_{10} - \alpha \dott\widetilde W_{-10} ) 
 \\
&[\widetilde K_1, \widetilde W_{01}]_{\r_\F} = - \sqrt{2} \beta\dott\widetilde K_2 + 2 x    \dott\widetilde W_{01} 
 \\
&[\widetilde K_1, \widetilde W_{1-1}]_{\r_\F} =- 2 x \dott\widetilde W_{1-1} + \sqrt{2}e^{\pi i\theta} \beta^* \! \dott\widetilde W_{10} \\
&[\widetilde K_1, \widetilde W_{10}]_{\r_\F} =\sqrt{2}e^{-\pi i\theta}\beta \dott\widetilde W_{1-1}-\sqrt{2}e^{\pi i\theta} \beta^* \!  \dott\widetilde W_{11}  \\
&[\widetilde K_1, \widetilde W_{11}]_{\r_\F} = 2 x
                                                                                                                                                                                  \dott\widetilde W_{11} - \sqrt{2} e^{-\pi i\theta}\beta \dott \widetilde W_{10}\\
&
[\widetilde K_2, \widetilde W_{01}]_{\r_\F} = \sqrt{2}e^{-\pi i\theta} \alpha^* \!  \dott\widetilde W_{11} + \sqrt{2}e^{\pi i\theta}\alpha \dott\widetilde W_{-11} \\
 &  [\widetilde K_2, \widetilde W_{1-1}]_{\r_\F}=2 x
   \dott\widetilde W_{1-1} - \sqrt{2}e^{-\pi i\theta}  \alpha
   \dott\widetilde W_{0-1}\\  
&[\widetilde K_2, \widetilde W_{10}]_{\r_\F} = {2}  x
    \dott\widetilde W_{10} -\sqrt{2} \alpha\dott    \widetilde K_1\\
   & [\widetilde K_2, \widetilde W_{11}]_{\r_\F}=2 x
   \dott\widetilde W_{11} + \sqrt{2} e^{\pi i\theta}\alpha
   \dott\widetilde W_{01}
\end{align*}
\begin{align*}
  &
[\widetilde W_{01}, \widetilde W_{1-1}]_{\r_\F}= \sqrt{2} \beta \dott\widetilde W_{1-1} + \sqrt{2} e^{\pi i\theta}\alpha \dott(\widetilde K_2 -\widetilde K_1 )
                               \\
&[\widetilde W_{01}, \widetilde W_{10}]_{\r_\F}=  \sqrt{2} \beta                            \dott\widetilde W_{10} - \sqrt{2}e^{\pi i\theta}\alpha \dott\widetilde W_{01}      \\
&[\widetilde W_{01}, \widetilde W_{11}]_{\r_\F}=  \sqrt{2} \beta \dott\widetilde W_{11} \\
&[\widetilde W_{1-1}, \widetilde W_{10} ]_{\r_\F}=     \sqrt{2}
      \alpha \dott\widetilde W_{1-1} \\
  & [\widetilde W_{1-1}, \widetilde W_{11}]_{\r_\F}=  -\sqrt{2}
     e^{-\pi i\theta}  \alpha \dott\widetilde W_{10} \\
 &[\widetilde W_{10}, \widetilde W_{11}]_{\r_\F}=   \sqrt{2}   \alpha
    \dott\widetilde W_{11}
\end{align*}
 \begin{align*}
&[ \widetilde W_{-1-1},\widetilde W_{01}]_{\r_\F}=     \sqrt{2} e^{2\pi i\theta}\beta \dott\widetilde W_{-1-1} - \sqrt{2}e^{\pi i\theta} \alpha^* \!\dott(\widetilde K_1 +\widetilde K_2 )\\
&[\widetilde W_{-1-1}, \widetilde W_{1-1}]_{\r_\F} =  \sqrt{2}   e^{-2\pi i\theta}\beta^* \!\dott\widetilde W_{0-1} \\
&[\widetilde W_{-1-1}, \widetilde W_{10}]_{\r_\F}=  \sqrt{2}   \alpha
                                                                                                                         \dott\widetilde W_{-1-1}  + \sqrt{2} e^{\pi i\theta}\beta^* \! \dott (\widetilde K_1 +\widetilde K_2 )\\
&[ \widetilde W_{-1-1}, \widetilde W_{11}]_{\r_\F}= -2x \dott (\widetilde K_1 +\widetilde K_2 ) - \sqrt{2}   \alpha \dott\widetilde W_{-10}- \sqrt{2}   \beta \dott\widetilde W_{0-1} \\
& [\widetilde W_{-10},\widetilde W_{01}]_{\r_\F}=  \sqrt{2} e^{2\pi i\theta}\beta \dott\widetilde W_{-10} - \sqrt{2}\alpha^* \! \dott\widetilde W_{01} \\
&[\widetilde W_{-10}, \widetilde W_{1-1}]_{\r_\F}=  -  \sqrt{2}
                                                                                                                                                            \alpha^* \! \dott\widetilde W_{1-1}  +  \sqrt{2} e^{-\pi i\theta}\beta^* \! \dott (\widetilde K_2 -\widetilde K_1 )\\
& [\widetilde W_{10}, \widetilde W_{-10}]_{\r_\F}= \sqrt{2} (\beta^* \! \dott\widetilde W_{01} + \beta \dott\widetilde W_{0-1} )\\
&[ \widetilde W_{-11},\widetilde W_{01}]_{\r_\F}=  \sqrt{2} e^{2\pi i\theta}\beta \dott\widetilde W_{-11}\\
 &[\widetilde W_{0-1}, \widetilde W_{01}]_{\r_\F}= \sqrt{2} ( \alpha^* \!
   \dott\widetilde W_{10} - \alpha \dott\widetilde W_{-1 0})\\
& [\widetilde W_{-11},\widetilde W_{1-1}]_{\r_\F} = 2x\dott (
                                                                               \widetilde K_1 -\widetilde K_2 ) - \sqrt{2}   \beta^* \! \dott\widetilde W_{01}  +  \sqrt{2}   \alpha \dott\widetilde W_{-10} 
 \end{align*}
\centerline{\it Table 1}
\\[1em]
The remaining brackets are obtained via  $*$-conjugation using
$([\widetilde X, \widetilde X']_{\r_\F})^*=[{{\widetilde X}'}{}^*,
{\widetilde X}^*]_{\r_\F}$ and $(b\dott\widetilde X)^*=
{\r_\F}_\alpha
\trl b^*\dott \!\:{\r_\F}^\alpha\trl\widetilde X^*$, where $\r_\F=\F_{21}\bF=\bF^2$.
These expressions hold since $\F$ is compatible with $*$-conjugation:
$\F^{*\otimes *}=(S\otimes S)\F_{21}$, (for details see
e.g. \cite[Sect. 8]{GR2}).
They can also readily be obtained by computing the brackets of the
missing generators starting from the undeformed brackets.

From Proposition \ref{prop:uea} we have that
  the universal enveloping algebra ${{\U}}(
\mathrm{aut}_{\mathcal{O}(S^4_\theta) }(\mathcal{O}(S^7_\theta) ))$ of this
$K_\F$-braided Lie algebra is a $K_\F$-braided Hopf algebra. The
algebra is the quotient of the tensor algebra of $\mathrm{aut}_{\mathcal{O}(S^4_\theta) }(\mathcal{O}(S^7_\theta) )$ modulo the
ideal generated by the braided Lie algebra relations \eqref{Lierelations}.
The coproduct, counit and antipode are uniquely determined on the generators
$b\dott \widetilde{X}\in
\mathrm{aut}_{\mathcal{O}(S^4_\theta) }(\mathcal{O}(S^7_\theta) )$ as
$$\Delta(b\dott\widetilde{X})=b\dott\widetilde{X}\,\bot\, 1+ 1\,\bot\, b\dott\widetilde{X}~,~~\varepsilon(b\dott\widetilde{X})=0~,~~S(b\dott\widetilde{X})=-\;\!b\dott\widetilde{X}~.
$$

Let us explicitly check the coproduct for the derivations $\widetilde{W}_\textsf{r}$.
For $a_\textsf{s} \in \mathcal{O}(S^7_\theta)$ an eigen-function of $H_j$ of eigenvalue $\textsf{s}_j$, the derivation  $\widetilde{W}_\textsf{r}$
acts as
\beq
\widetilde{W}_\textsf{r} (a_\textsf{s}) 
= (\bF^\alpha \trl {W}_\textsf{r} )(\bF_\alpha \trl a_\textsf{s}) 
= e^{- i \pi \theta \wdg{r}{s}} {W}_\textsf{r} (a_\textsf{s}) \, .
\eeq
On the product of two such eigen-functions  $a_\textsf{s}, {a}_\textsf{m}  \in \mathcal{O}(S^7_\theta)$, we can explicitly see that $\widetilde{W}_\textsf{r} $ acts as a braided derivation, with respect to the braiding  $\rF = \F_{21} \, \bF = \bF^2$:
\begin{align*}
\widetilde{W}_\textsf{r} (a_\textsf{s} \dott {a}_\textsf{m} ) 
& = e^{- i \pi \theta \wdg{r}{(s+m)} } e^{- i \pi \theta \wdg{s}{m}}  {W}_\textsf{r} (a_\textsf{s}   {a}_\textsf{m} ) 
\\
& = e^{- i \pi \theta (\wdg{r}{(s+m)}+ \wdg{s}{m}) } 
\big[ {W}_\textsf{r} (a_\textsf{s})   {a}_\textsf{m}  + a_\textsf{s} {W}_\textsf{r} ({a}_\textsf{m} ) \big]
\\
& = e^{- i \pi \theta (\wdg{r}{(s+m)}+ \wdg{s}{m}) } 
\big[ e^{i \pi \theta \wdg{(r+s)}{m} }  {W}_\textsf{r} (a_\textsf{s}) \dott {a}_\textsf{m}
+
e^{ i \pi \theta \wdg{s}{(r+m)}   } 
 a_\textsf{s}) \dott {W}_\textsf{r} ({a}_\textsf{m} ) \big]
\\
& = e^{- i \pi \theta \wdg{r}{s} }  {W}_\textsf{r} (a_\textsf{s}) \dott {a}_\textsf{m}
+
e^{- 2 i \pi \theta  (\wdg{r}{s}+ \wdg{r}{m})       } 
 a_\textsf{s} \dott {W}_\textsf{r} ({a}_\textsf{m} )  
\\
& =  \widetilde{W}_\textsf{r} (a_\textsf{s}) \dott {a}_\textsf{m}
+
e^{-2  i \pi \theta  \wdg{r}{s}      } 
 a_\textsf{s} \dott \widetilde{W}_\textsf{r} ({a}_\textsf{m} ) 
\\
&=\Delta(\widetilde{W}_\textsf{r} )(a_\textsf{s} \dott {a}_\textsf{m} ) \,  . \qedhere
\end{align*}\\[-2.1em]\phantom{3}  
\end{ex}

\begin{ex}{\it{The quantum orthogonal bundle on the homogeneous space
  $S^{4}_\theta$.}}
We describe the principal bundle
$\mathcal{O}(SO(5,\mathbb{R}))\to
\mathcal{O}(SO(5,\mathbb{R}))/\mathcal{O}(SO(4,\mathbb{R}))=S^4$ as
the Hopf--Galois extension $\O(S^4)\subseteq
\mathcal{O}(SO(5,\mathbb{R}))$. Here $\mathcal{O}(SO(5,\mathbb{R}))$
is the Hopf algebra generated by the commuting entries of the matrix
$N=(n_{JK})$, $J,K=1,2\ldots 5$, modulo the relations 
$$N^t Q N =Q~,~~ N Q N^t=Q~,~~
\det(N)=1~,~\mbox{ where }~ Q:=\tiny{\begin{pmatrix}
0 & \II_2 &0
\\
\II_2 & 0 &0
\\
0 & 0 & 1
\end{pmatrix} }~.
$$
The costructures on the generators are $\Delta(N)=N \overset{.}{\otimes} N$,
$\varepsilon(N)=\II_5$, $S(N)=N^{-1}$ and $*$-structure
$N^*=(n^*_{JK})= Q N Q^t$, that is, $n_{JK}^*=n_{J'K'}$ with
$1'\!=\!3$, $2'\!=\!4$, $3'\!=\!1$, $4'\!=\!2$, $5'\!=\!5$. 
The Hopf algebra $\mathcal{O}(SO(4,\mathbb{R}))$ is the quotient via the Hopf ideal
generated by $n_{55}-1, n_{5\nu}, n_{\mu 5}$ for $\mu,\nu=1,2,3,4$ so
that we have the Hopf algebra projection
\beq\label{pi}
\pi:\mathcal{O}(SO(5, \IR)) \longrightarrow \mathcal{O}(SO(4, \IR)) \; , \quad 
N
\longmapsto
{\tiny{\begin{pmatrix}
M &0 
\\
0 &  1
\end{pmatrix}~. }}
\eeq
This gives the coaction $\delta=(\id\otimes
\pi)\Delta: \mathcal{O}(SO(5, \IR))\to  \mathcal{O}(SO(5,\IR))\otimes
\mathcal{O}(SO(4, \IR))$. The base space algebra $\O(S^4)\subset \mathcal{O}(SO(5, \IR))$ is then generated by the coinvariant elements
\begin{equation}\label{4sphere}\alpha=\sqrt{2}\!\:n_{15}\,,~ \beta=\sqrt{2}\!\:n_{25}\,,~
\alpha^*=\sqrt{2}\!\:n_{35}\,,~ \beta^*=\sqrt{2}\!\:n_{45}\,,~
x=n_{55}\,.
\end{equation}
The $\sqrt{2}$ rescaling is chosen so that
$\alpha,\beta, x$ satisfy the 4-sphere relation
$\alpha\alpha^*+\beta\beta^*+x^2=1$ (coming from  the reality and orthogonality
conditions $N^tN^*=\II_5$).

The $\mathcal{O}(S^{4})$-module freely
generated by the right-invariant vector fields on the group manifold
$SO(5,\mathbb{R})$ is the $\mathcal{O}(S^{4})$-module of $SO(4,\mathbb{R})$-equivariant derivations of the
principal bundle $SO(5,\mathbb{R}))\to S^4$. 
A convenient basis is 
\begin{align}\label{der-sopra-SO}
&H_1:=  n_{1K} \parn{1K}- n_{3K} \parn{3K}
\quad  
&&H_2 := n_{2K} \parn{2K}- n_{4K} \parn{4K}
\nn
\\
& E_{10}   :=   n_{5K} \parn{3K}  - n_{1K} \parn{5K}
\quad  
&& E_{-10 } := n_{3K} \parn{5K} -n_{5K} \parn{1K}  
\nn
\\
& E_{01}  := n_{5K} \parn{4K} - n_{2K} \parn{5K}  
\quad  
&&E_{0 -1}  :=    n_{4K} \parn{5K} - n_{5K} \parn{2K}
\nn
\\
& E_{11}:=    n_{2K} \parn{3K} - n_{1K} \parn{4K} 
\quad  
&& E_{-1-1}   :=  n_{3K} \parn{2K} - n_{4K} \parn{1K}
\nn
\\
& E_{1-1}:= n_{4K} \parn{3K} - n_{1K} \parn{2K} 
\quad  
&& E_{-1 1}:=   n_{3K} \parn{4K} - n_{2K} \parn{1K}
\end{align}
with summation on  $K=1, \dots, 5$ understood, and 
$\parn{IJ}(n_{KL}) = \delta_{IK} \delta_{JL}$, for $I,J,K,L=1,2,\dots, 5$. 
These ten generators close the Lie algebra \eqref{so5} of $so(5)$
and satisfy the reality conditions  $H_j^*=H_j$ and $E_\textsf{r}
^*=E_{-\textsf{r}}$.
As in \eqref{Xreal}, the generic real equivariant derivation is of the form 
$
X= b_1 H_1 + b_2 H_2 + \sum_\textsf{r}  b_\textsf{r}  E_\textsf{r} 
$, 
with 
$b_j^*=b_j$,
$b_\textsf{r}^*=b_{-\textsf{r}}$ in $\mathcal{O}(S^4)$ 
and now $H_j,
E_{\mathsf{r}}$ in \eqref{der-sopra-SO}. 

Using \eqref{4sphere}, these derivations restricted to $O(S^4)$
coincide with  those in \eqref{der-sopra1}, \eqref{der-sopra2} restricted to $O(S^4)$.
This implies that the verticality condition $X(b) = 0$ for each $b \in
O(S^4)$, is the same as that imposed on
the derivations of the istanton example. Therefore, 
infinitesimal gauge transformations, that is,
$SO(4,\mathbb{R})$-equivariant derivations $X$ which are also
vertical, are generated, as an $O(S^4)$-module, by the
vector fields  $K_j, W_\textsf{r}$ defined as in \eqref{KW} but now
with $H_j, E_\textsf{r}$ as in \eqref{der-sopra-SO}.

Due to $SO(5)$-equivariance of the bundle
$SO(5)\to S^4$ we have the direct sum decomposition \cite[Prop. 4.3]{pgcex}, 
\begin{equation}\label{dsdd}
\mathrm{aut}_{\mathcal{O}(S^4) }(\mathcal{O}(SO(5, \IR)))=\mbox{$\bigoplus_{n\in \mathbb{N}_0}$} \, [d(2,n)] \oplus \widehat{[d(2,n-1)]}
\end{equation}
of the gauge Lie algebra. Here 
 $[d(2,n)]$,  $\widehat{[d(2,n-1)]}$ are the
 representations of $so(5)$ with highest weight vectors $\alpha^n
 W_{11}$, $\alpha^{n-1}
 (\sqrt{2} x W_{11} + \alpha W_{01} - \beta W_{10})$ of weights
 $(n,1)$,  $(n+1,1)$, respectively. They consist of  derivations on
 $\mathcal{O}(SO(5, \IR))$ which are combinations of the derivations
 in \eqref{KW} with spherical harmonics of degree $n$ on $S^4$.
\\

The twist deformation of the Hopf--Galois extension
$\O(S^4)\subset \mathcal{O}(SO(5,\mathbb{R}))$ 
gives the $\mathcal{O}(SO_\theta(4,\mathbb{R}))$-Hopf--Galois extension
$\O(S_\theta^4)\subset \mathcal{O}(SO_\theta(5,\mathbb{R}))$. 
It can be obtained in two equivalent ways
(cf. \cite[Sect.~4.1 and 4.1.1]{ppca} where, dually, 2-cocycles
are used):
\item{- }{On the one hand, using the Cartan subalgebra of $so(5)$ given by
    $H_1$ and $H_2$ 
we  twist with
\begin{equation}\label{tso5}
\F:= e^{2 \pi i\theta (H_1 \ot H_2 -H_2 \ot H_1)}~,~~~\theta \in \mathbb{R}~,
\end{equation}
the Hopf algebra
$\mathcal{O}(SO(5,\mathbb{R}))$ to $\mathcal{O}(SO_\theta(5,\mathbb{R}))$
and then obtain $\mathcal{O}(SO_\theta(4,\mathbb{R}))$ as the quotient
via the Hopf ideal generated by $n_{55}-1, n_{5\nu}, n_{\mu 5}$ for
$\mu,\nu=1,2,3,4$. Defining the right
$\mathcal{O}(SO_\theta(4,\mathbb{R}))$-coaction
$\delta=(\id\otimes\pi)\Delta$, with $\pi
:\mathcal{O}(SO_\theta(5,\mathbb{R}))\to
\mathcal{O}(SO_\theta(4,\mathbb{R}))$, $N\mapsto \big(
{}^{M_{\,}}_{\:0^{}} 
{}^{0_{}}_{1^{}}\big)
$, cf. \eqref{pi}, we
obtain the quantum homogeneous space  $\mathcal{O}(S^4_\theta)\subset
\mathcal{O}(SO_\theta(5,\mathbb{R}))$ as the coinvariant subalgebra.}
The generators of $\mathcal{O}(S_\theta^{4})$
are in \eqref{4sphere} and satisfy the
unit radius and the commutation relations  \eqref{t4sphererel}.

\item{- }
On the other hand we can look at the inclusion 
$\O(S^4)\subset \mathcal{O}(SO(5,\mathbb{R}))$
as a Hopf–Galois extension, and twist it. In this case we forget the
Hopf algebra structure of $\mathcal{O}(SO(5,\mathbb{R}))$ and just use that we have a left and a
right action of $U_H:=\U(so(4))$ on the total space, this is the left
$U_H^{op}\otimes U_H$ action of Section \eqref{QPBQHS}.
We first consider the twist in
\eqref{tso5}, with $H_1$ and $H_2$ that are the generators of the  Cartan subalgebra of
$so(4)\subset so(5)$, and twist the Hopf--Galois extension using the
$U_H$-action.
The structure Hopf algebra
$\mathcal{O}(SO(4,\mathbb{R}))$ is twisted to
the Hopf algebra $\mathcal{O}(SO_\theta(4,\mathbb{R}))$ and the total
space $\mathcal{O}(SO(5,\mathbb{R}))$ is twisted  
as a  
left $U_H$-module algebra. The resulting Hopf--Galois extension
is still equivariant with respect to the $U_H^{op}$-action. 
We then consider a second  twist deformation using 
\eqref{tso5} as a
twist of $K=U^{op}$ (as in the first paragraph of Section \ref{BHagsftd}). This gives the Hopf--Galois extension
$\O(S_\theta^4)\subset
\mathcal{O}(SO_\theta(5,\mathbb{R}))$.

This second approach
is well adapted to describe the gauge Lie algebra of the Hopf--Galois extension
$\O(S_\theta^4)\subset \mathcal{O}(SO_\theta(5,\mathbb{R}))$.
The first twist deformation uses the $U_H$-action which acts trivially on 
$SO(4,\mathbb{R})$-equivariant derivations and hence on infinitesimal gauge transformations. The gauge Lie
algebra is therefore undeformed. This same result follows from
Proposition \ref{isoclass}, where  the
2-cocycle $\cot$ is associated with the maximal torus  $\mathbb{T}^2\subset SO(4,\mathbb{R})$ of the Cartan
subalgebra and defined by the twist \eqref{tso5}.
Under the second twisting the Hopf algebra $K = U^{op}$ becomes the
Hopf algebra $K_\F$ with nontrivial braiding $\r_\F={\bF}^{2}$ and the gauge Lie algebra
correspondingly becomes a $K_\F$-braided Lie algebra. According to
Proposition \ref{prop:gf}  the braided Lie bracket on the generators
reads as in \eqref{aut-twist} where $K_j, W_{\textsf{r}}$ are
defined as in \eqref{KW} but now with $H_j, E_{\textsf{r}}$ in
\eqref{der-sopra-SO}. We next apply the isomorphism $\dd$ of
Proposition \ref{autautF} and obtain the braided gauge Lie algebra
$\mathrm{aut}_{\mathcal{O}(S^4_\theta)
}(\mathcal{O}(SO_\theta(5,\mathbb{R})))$ of  the  $\mathcal{O}(SO_\theta(4,\mathbb{R}))$-Hopf--Galois extension 
$\O(S_\theta^4)\subset \mathcal{O}(SO_\theta(5,\mathbb{R}))$.
\begin{prop}The braided gauge Lie algebra 
$\mathrm{aut}_{\mathcal{O}(S^4_\theta)
}(\mathcal{O}(SO_\theta(5,\mathbb{R}))$ is described in Proposition
\ref{gaugeinstanton} by replacing
 $\mathcal{O}(S^7_\theta)$ with 
$\mathcal{O}(SO_\theta(5,\mathbb{R}))$. 
\end{prop}
In particular, the Lie algebra bracket of 
$\mathrm{aut}_{\mathcal{O}(S^4_\theta)}(\mathcal{O}(SO_\theta(5,\mathbb{R}))$ is determined by the brackets
  in {\it Table 1} of the previous Example
  \ref{gaugeinstanton}. The difference between the gauge
  transformations of the noncommutative instanton and orthogonal
  bundles on $S^4_\theta$ is in their actions on the total
  spaces $\mathcal{O}(S^7_\theta)$ and
  $\mathcal{O}(SO_\theta(5,\mathbb{R}))$, respectively.  This is clear
  when comparing the two
  different direct sums \eqref{dsd} and \eqref{dsdd}
  describing the linear space structure of the  braided  Lie algebras 
$\mathrm{aut}_{\mathcal{O}(S^4_\theta)}(\mathcal{O}(S^7_\theta)$
and
$\mathrm{aut}_{\mathcal{O}(S^4_\theta)}(\mathcal{O}(SO_\theta(5,\mathbb{R}))$.
\end{ex}

\begin{ex}{\it{A principal line bundle over the
      $\kappa$-Minkowski spacetime (Jordanian twist deformation).}} 
Let $P=\mathbb{R}_{>0}\ltimes \mathbb{R}^{n+1}$ be the affine
group of dilatations and translations of the  
$\mathbb{R}^{n+1}$. It is the group of invertible
matrices $({}^u_x{\,}^0_\II)$ with $u\in \mathbb{R}_{>0}, x\in
\mathbb{R}^{n+1}$. We consider the  principal bundle $P\to P/\mathbb{R}_{>0}$, where  $\mathbb{R}_{>0}$ 
acts by right multiplication of 
$({}^u_0{\,}^0_\II)$. From $({}^u_x{\,}^0_\II)=
({}^{\;1}_{xu^{-1}}{\,}^0_\II)({}^u_0{\,}^0_\II)$ the base space is
$\mathbb{R}^{n+1}$ and the bundle is trivial.
Dually
we have 
the Hopf--Galois extension  
$B=A^{co
  H}\subseteq A$ of coordinate rings on these affine algebraic varieties. The 
  coordinate ring  $A$ is generated by the elements $u$, $x^I$ while  the coordinate
ring  $B=\O(\mathbb{R}^{n+1})$ is generated by $\x^I:=x^Iu^{-1}$, $I=0,1,\ldots n$.
The infinitesimal gauge transformations form the abelian Lie
algebra  
$\O(\mathbb{R}^{n+1})\otimes{\rm{Lie}}(\mathbb{R}_{>0})\simeq
\O(\mathbb{R}^{n+1})$.
The Lie algebra of right-invariant vector fields of the group $P$ is
spanned by
$$u\frac{\partial}{\partial u},~u\frac{\partial}{\partial x^I}~,~~~I=0,1,\ldots n~.$$ 
The  vector field $u\frac{\partial}{\partial u}+ \sum_I x^I\frac{\partial}{\partial x^I}$ is vertical
and   span, with coefficients  in $B=\O(\mathbb{R}^{n+1})$, the Lie algebra 
${\rm{Lie}}(\mathbb{R}_{>0})$ of infinitesimal gauge group transformations. 

In order to deform the principal bundle $P\to P/\mathbb{R}_{>0}$, 
let
$K$ 
be the universal enveloping algebra of
the right-invariant vector fields of the group $P$ and, for $\kappa\in
\mathbb{R}$,  define the
Jordanian twist (see e.g. \cite{Borowiec:2008uj})
\begin{equation}
\F=\exp \left( u\mbox{$\frac{\partial}{\partial u}$}\otimes \sigma \right) ~,\label{nsymJ}
\end{equation}%
where $\sigma =\ln
\left( 1+\frac{1}{\kappa }P_0\right)$, $P_0=iu\frac{\partial}{\partial
  x^0}$ and $[u\mbox{$\frac{\partial}{\partial u}$},P_0]=iP_0$.
This twist is a formal power series in $\frac{1}{\kappa }P_0$.
However, due to the
affine algebraic nature of $A$, only a finite number of powers of $\frac{1}{\kappa
}P_0$ are relevant and we can   avoid defining the
topological completion of the tensor product $K\otimes K$.
We deform the $K$-equivariant Hopf--Galois extension 
$B=A^{co
  H}\subseteq A$
 to the $K_\F$-equivariant  one
$B_\F=A_\F^{co H}\subseteq A_\F$.
The total space algebra $A_\F$ is the 
 polynomial ring generated by the coordinates $u$, $u^{-1}$ with relations
$uu^{-1}=u^{-1}u=1$,  and  $n+1$ coordinates $x^I$, $I=0,1,\ldots , n$. The only nontrivial commutation
relation is
\begin{equation}\label{xumeno}
  x^0\dotF u^{-1}-u^{-1}\dotF
x^0=-\frac{i}{\kappa}~.
\end{equation}
The base algebra
$B_\F$
is generated by the coordinates $\x^I=x^Iu^{-1}=x^I\dotF u^{-1}$, $I=0,1,\ldots , n$, with  commutation
relations 
\begin{equation} \label{kppa-mink} 
  \x^0\dotF \x^j- \x^j \dotF
  \x^0=-\frac{i}{\kappa}\x^j~, \quad\x^l \dotF \x^j - \x^j\dotF \x^l=0~
\end{equation}
for $j , l = 1, \dots, n$.  
This is the algebra of the  $(n+1)$-dimensional $\kappa$-Minkowski quantum space.
A $*$-conjugation on $A_\F$ is
defined by $u^*=u, {(x^I)}^*=x^I$. It follows that $(\x^j)^*=\x^j$ and
$(\x^0)^*=\x^0+\frac{i}{\kappa}$, so that
$(\x^0+\frac{i}{2\kappa})^*=\x^0+\frac{i}{2\kappa}$ and this real generator
has the same commutation relations as $\x^0$ in \eqref{kppa-mink}. 

Since the 
vector field $u\frac{\partial}{\partial u}$ entering the twist
commutes with the vertical derivation $X:=u\frac{\partial}{\partial u}+
\sum_I x^I\frac{\partial}{\partial x^I}$, we have the braided vertical
$\mathbb{R}_{>0}$-equivariant derivation
$\dd(X)=X$ satisfying $[X,X]_{\r_\F}=\dd([X,X]_\F)=\dd([X,X])=0$.
Infinitesimal gauge transformations are given by $b\dotF X$ with $b\in
B_\F=\O(\mathbb{R}^{n+1})^{}_\F$.
From \eqref{LieAmodF++} we see that the braided gauge Lie algebra is abelian:
$[b\dotF X,b'\dotF X]_{\r_\F}=b\dotF b'\dotF [X,X]_{\r_\F}=0$.
 \end{ex}
\begin{ex} {\it{A non-abelian principal  bundle over the
      $\kappa$-Minkowski spacetime.}} 
  The previous example generalises to a non-abelian
  setting. For example let the total space  be $P=\mathbb{R}_{>0}\ltimes\mathbb{R}^{n+1}\rtimes SO(1,n)$,
the $n+1$-dimensional Poincar\'e-Weyl group, the semidirect product of the Poincar\'e group with the group
$\mathbb{R}_{>0}$ of dilatations. This  is  the group of invertible
matrices $({}^u_x{\,}^0_T)$ with $u\in \mathbb{R}_{>0}, x\in
\mathbb{R}^{n+1}$ and $T=(t_{IJ})_{I,J=0,1,...,n}\in  SO(1,n)$. 

Associated with the  principal bundle $P\to \mathbb{R}^{n+1}=P/G$, where $G=\mathbb{R}_{>0}\times
SO(1,n)$ we have the Hopf--Galois
extension  
$B=A^{co
  H}\subseteq A$ with $A$ generated by the coordinate functions
$u$, $x^I$, $t_{IJ}$. As in the previous example, with  the twist in \eqref{nsymJ}  the
only nontrivial commutation relation in the deformed algebra $A_\F$ is the one in \eqref{xumeno}; the base space
algebra $B_\F$ is the $\kappa$-Minkowski  in \eqref{kppa-mink}. The infinitesimal  
 gauge transformations form the braided Lie algebra   
\begin{equation*}
  \begin{split}
    {\rm{aut}}^{\r_\F}_{B^{}_\F}(A_\F)= \dd({\rm{aut}}_{B}(A)_\F)&=
\O(\mathbb{R}^{n+1})_\F\otimes\dd({\rm{Lie}}(\mathbb{R}_{>0}\times
SO(1,n))^{}_\F)\\
&=\O(\mathbb{R}^{n+1})_\F\otimes{\rm{Lie}}(\mathbb{R}_{>0}\times
SO(1,n))~.
\end{split}\end{equation*}
The last equality follows from  the commutativity of 
 the vector field $u\frac{\partial}{\partial u}$,
 entering the twist, with the right $G$-invariant ($H$-equivariant)
 vertical vector fields $X\in {\rm{Lie}}(\mathbb{R}_{>0}\times
SO(1,n))$ generating the gauge
transformations
$$u\frac{\partial}{\partial
  u}+x^I\frac{\partial}{\partial x^I}~,~~
t_{0K}\frac{\partial}{\partial t_{jK}}+t_{jK}\frac{\partial}{\partial
  t_{0K}}~,~~t_{iK}\frac{\partial}{\partial t_{jK}}-t_{jK}\frac{\partial}{\partial
  t_{iK}}~.
  $$ 
Here $i,j=1,2,\ldots , n$ and sum over $I,K=0,1,\ldots , n$
is understood. Indeed this commutativity
implies $\dd(X)=X$ and  $[X,X']_\F=[X,X']$, so that $[X,X']_{\r_\F}=\dd([X,X']_\F)=[X,X']$ for all
$X,X'\in {\rm{Lie}}(\mathbb{R}_{>0}\times
SO(1,n))$. Moreover, 
$[b\dotF X,b'\dotF X']_{\r_\F}=b\dotF
b'\dotF [X,X']$,
for all $b,b'\in \O(\mathbb{R}^{n+1})_\F$.
 \end{ex}

\begin{ex}\label{formalex}{\it{Formal deformation quantization of smooth principal bundles
and their gauge groups.}}
The twist deformations presented in this section for 
principal bundles that are affine algebraic varieties can be also
considered for smooth $L$-equivariant principal bundles $P\to P/G$.  In this case, see
\cite[Ex. 3.24]{ppca}, we have the
 gical $K[[\hbar]]$-equivariant
$H[[\hbar]]$-Hopf–Galois extension
$B[[\hbar]] \simeq A[[\hbar]]^{co H[[\hbar]]} \subseteq A[[\hbar]]$, 
where $K[[\hbar]]$, $B[[\hbar]]$, $H[[\hbar]]$ and
$A[[\hbar]]$ are the formal power
series extension of the $\mathbb{C}$-modules $K=C^\infty(L)$, $B=C^\infty(P/G)$, $H=C^\infty(G)$ and $A=C^\infty(P)$.
A twist $F\in K[[\hbar]]\hat\otimes
K[[\hbar]]\simeq C^\infty(L\times L)[[\hbar]]$ then leads to 
the noncommutative topological $H[[\hbar]]_\F$-Hopf–Galois
extension $B[[\hbar]]_\F \simeq A[[\hbar]]_\F^{co H[[\hbar]]_\F} \subseteq
A[[\hbar]]_\F$.  
 The associated braided gauge
  Lie algebra is ${\rm{aut}}^{\r_\F}_{B[[\hbar]]_\F}(A[[\hbar]]_\F)$,
  where ${\rm{aut}}^{\r_\F}_{B[[\hbar]]_\F}(A[[\hbar]]_\F)\subseteq
{\rm{Der}}^{\r_\F}(A[[\hbar]]_\F)\subseteq\Hom(A[[\hbar]]_\F,A[[\hbar]]_\F)$, this latter being the
linear space  of continuous algebra automorphisms.
 \end{ex}

\medskip
\noindent
\textbf{Acknowledgments.}~\\[.5em]
We thank Alessandro Ardizzoni for useful discussions.
PA acknowledges partial support from INFN, CSN4, Iniziativa
Specifica GSS and from INdAM-GNFM. 
This research has a financial support from Universit\`a del Piemonte Orientale. GL acknowledges partial support from INFN, Iniziativa
Specifica GAST and from INdAM-GNSAGA.
CP acknowledges partial support from Fondazione Cassa di Risparmio di Torino,  from Universit\`a di Trieste (Dipartimenti di Eccellenza,  legge n. 232 del 2016) and from INdAM-GNSAGA.

\appendix

\section{Proof of Proposition  \ref{DKF}}\label{app:DKF}
From the expression $(\Delta \ot \id) \bF$ obtained from the twist condition \eqref{twist-bF} for $\bF$, it is immediate to see that the two expressions of the map $\dd$ in \eqref{mapD} coincide. Then, 
$\dd^{-1}(k) = \bF^\alpha k S_\F(\bF_\alpha) \uf$ is the inverse of $\dd$. 
An equivalent expression for $\dd^{-1}$ is then given by
$\dd^{-1}(k)=(\F^\alpha \tadF k)\F_\alpha$. Indeed, this follows by recalling that $\F$ is a twist for $\K$ if and only if $\bF$ is a twist for $\KF$ and observing that the element $\bar{\textsf{u}}_{\bF}$ in the Hopf algebra $\KF$ (analogous to the element $\buf$ in $\K$) is 
\beq\label{u=u}
\bar{\textsf{u}}_{\bF}=S_\F(\F^\beta) \F_\beta = \uf S(\F^\beta) \buf \F_\beta
= \uf S(\F^\beta)  S(\bF^\alpha)  \bF_\alpha \F_\beta =
\uf S(\bF^\alpha \F^\beta)  \bF_\alpha \F_\beta = \uf \; .
\eeq
Next we show that $\dd$, or equivalently its inverse $\dd^{-1}$, is a morphism of $\KF$-modules:
$
\dd^{-1} (h \tadF k) = h \tad \dd^{-1}(k)
$.
On the one hand we compute 
\begin{align*}
\dd^{-1} (h \tadF k) &=  ((\F^\alpha h) \tadF k) \F_\alpha 
= \fone{(\F^\alpha h)}k  \uf S(\ftwo{(\F^\alpha h)}) \buf \F_\alpha 
\\
&= \F^\beta \one{\F^\alpha}\one{ h} \bF^\gamma ~k ~   \uf  S( \F_\beta \two{\F^\alpha}\two{ h}
\bF_\gamma) \buf \F_\alpha 
\end{align*}
which, by using the twist condition  \eqref{twist-F} gives
\begin{align*}
\dd^{-1} (h \tadF k)  
&= \F^\alpha \one{ h} \bF^\gamma ~k ~   \uf  S( \F^\beta \one{\F_\alpha}\two{ h}
\bF_\gamma) \buf \F_\beta \two{\F_\alpha}
\\
&= \F^\alpha \one{ h} \bF^\gamma ~k ~   \uf  S(\bF_\gamma) S(\two{ h}) S( \one{\F_\alpha}) S(\F^\beta
) \buf \F_\beta \two{\F_\alpha}
\\
&= \F^\alpha \one{ h} \bF^\gamma ~k ~   \uf  S(\bF_\gamma) S(\two{ h}) S( \one{\F_\alpha}) \two{\F_\alpha}
\end{align*}
when using that $S(\F^\beta) \buf \F_\beta=1$. Finally, since  $\F$ is unital,
\begin{align*}
\dd^{-1} (h \tadF k)  = \one{ h} \bF^\gamma ~k ~   \uf  S(\bF_\gamma) S(\two{ h}). 
\end{align*}
On the other hand, since $S_\F (\cdot) \uf=\uf S(\cdot)$, we have
$$
h \tad \dd^{-1}(k) = \one{h}  \bF^\alpha k S_\F(\bF_\alpha) \uf S(\two{h})
=
\one{h}  \bF^\alpha k \uf S(\bF_\alpha)  S(\two{h})
$$
and the two expressions coincide.

It is easy to prove that the map $\dd$ is an algebra map: 
\begin{align*}
\dd(h \dotF k) & = \dd ((\bF^{\alpha} \trl  h)\, (\bF_{\alpha} \trl k))
=\left( \bF^\beta \trl ((\bF^{\alpha} \trl  h)\, (\bF_{\alpha} \trl k))\right) \bF_\beta
\\
& =\left( (\one{\bF^\beta} \bF^{\alpha}) \trl  h \right) \left(  (\two{\bF^\beta}\bF_{\alpha}) \trl k)\right) \bF_\beta ,
\end{align*}
and using the twist condition, this simplifies to
 \begin{align*}
\dd(h \dotF k) &
=( \bF^\beta \trl  h ) \left(  (\one{\bF_\beta}\bF^{\alpha}) \trl k)\right) \two{\bF_\beta}\bF_{\alpha}
=( \bF^\beta \trl  h )  \one{\bF_\beta} (\bF^{\alpha}\trl k) S(\two{\bF_\beta}) \three{\bF_\beta}\bF_{\alpha}
\\
&= ( \bF^\beta \trl  h )  {\bF_\beta} (\bF^{\alpha}\trl k) \bF_{\alpha}
=
\dd(h) \dd( k)  \; .
 \end{align*}

The proof that $\dd^{-1}$ is also a coalgebra morphism, 
$(\dd^{-1} \ot \dd^{-1}) \circ \underline{\Delta_\F}=  \Delta_{\tK_F}  \circ \dd^{-1}$, 
 is more involved. We first observe that   $\bF^\alpha k S(\bF_\alpha)= \dd^{-1}(k ~\buf)$ for each $k \in \underline{\KF}$. Then, using the twist condition \eqref{twist-F},
$$
(\F^\alpha \tad k)\ot \F_\alpha  = 
\bF^\gamma \F^\nu k S(  \bF_\gamma \F^\mu \one{\F_{\nu}}) \ot \F_\mu \two{\F_\nu} 
=
\dd^{-1} \left( \F^\nu k  S(\F^\mu \one{\F_\nu}) \buf \right)  \ot \F_\mu   \two{\F_\nu} ,
$$
and the coproduct in ${\tK_F}$ can be rewritten as
\begin{align*}
 \Delta_{\tK_F} (k) & =
\F^{\beta} \tad (\one{k} S(\r_\alpha) ) \ot (\F_\beta \r^\alpha) \tad \two{k} 
\\
& =
\dd^{-1} \left( \F^\nu  (\one{k} S(\r_\alpha)) S(\F^\mu \one{\F_\nu}) \buf \right)  
\ot (\F_\mu   \two{\F_\nu}  \r^\alpha) \tad \two{k} 
\\
& =
\dd^{-1} \left( \F^\nu  \one{k} S(\one{\F_\nu} \r_\alpha) S(\F^\mu) \buf \right)  
\ot (\F_\mu   \two{\F_\nu}  \r^\alpha) \tad \two{k} 
\\
& =
\big( \dd^{-1}  \ot \id \big) \left( \F^\nu  \one{k} S(\r_\alpha \two{\F_\nu} ) S(\F^\mu)  \buf
\ot (\F_\mu    \r^\alpha \one{\F_\nu} ) \tad \two{k} \right) .
\end{align*}
In the last equality we have used the quasitriangularity condition. 

For the coproduct in $\underline{\KF}$ we have
\begin{align*}
\underline{\Delta_\F} (k) &= \fone{k}  S_\F (\rF_\alpha) \ot (\rF^\alpha \tadF \ftwo{k}) 
=
\F^{\mu} \one{k} \bF^{\nu} \uf S (\rF_\alpha) \buf \ot \left(\rF^\alpha \tadF (\F_\mu \two{k} \bF_{\nu}) \right)
\end{align*}
with
$\rF = \F_{21} \r ~\bF = \F_\alpha \r^\beta \bF^{\gamma} \ot \F^\alpha \r_\beta \bF_{\gamma} $. 
Then, since $\dd^{-1}$ is a $\KF$-module morphism, 
\begin{align*}
(\id \ot \dd^{-1})\underline{\Delta_\F} (k) &  =\F^{\mu} \one{k} \bF^{\nu} \uf S (\rF_\alpha) \buf \ot \left(\rF^\alpha \tad \dd^{-1}(\F_\mu \two{k} \bF_{\nu}) \right)
\\
& =\F^{\mu} \one{k} \bF^{\nu} \uf S (\F^\alpha \r_\beta \bF_{\gamma}) \buf \ot \left(
(\F_\alpha \r^\beta \bF^{\gamma})
 \tad \dd^{-1}(\F_\mu \two{k} \bF_{\nu}) \right) .
\end{align*}
To prove that $\dd^{-1}$ is also a coalgebra morphism, we have hence to show that
\begin{multline*}
\F^{\mu} \one{k} \bF^{\nu} \uf S (\F^\alpha \r_\beta \bF_{\gamma}) \buf \ot \left(
(\F_\alpha \r^\beta \bF^{\gamma})
 \tad \dd^{-1}(\F_\mu \two{k} \bF_{\nu}) \right) \\ 
=
\F^\nu  \one{(\dd^{-1}(k))} S(\r_\alpha \two{\F_\nu} ) S(\F^\mu)  \buf
\ot (\F_\mu    \r^\alpha \one{\F_\nu} ) \tad \two{(\dd^{-1}(k))} 
\end{multline*}
where
$$\Delta(\dd^{-1}(k))= \one{\bF^\gamma} \one{k}  \bF^\rho \uf S(\bF_\sigma) S(\two{\bF_\gamma})
\ot \two{\bF^\gamma} \two{k} \bF_\rho \uf S(\bF^\sigma) S(\one{\bF_\gamma})~,
$$
as deduced from
$ \Delta(\uf)= \bF (\uf \ot \uf) (S \ot S)(\bF_{21})= \bF^\rho \uf S(\bF_\sigma) \ot
\bF_\rho \uf S(\bF^\sigma)$, obtained with some computations. 
Then the right hand side of the identity we wish to show becomes
\begin{align*}
& \F^\nu   
\one{\bF^\gamma} \one{k}  \bF^\rho \uf S(\bF_\sigma) S(\two{\bF_\gamma})
S(\r_\alpha \two{\F_\nu} ) S(\F^\mu)  \buf
\\ & \hspace{5cm}
\ot (\F_\mu    \r^\alpha \one{\F_\nu} ) \tad 
\left(\two{\bF^\gamma} \two{k} \bF_\rho \uf S(\bF^\sigma) S(\one{\bF_\gamma})\right)
\\
&=
\F^\nu  
\one{\bF^\gamma} \one{k}  \bF^\rho \uf S(\bF_\sigma) S(\two{\bF_\gamma})
S(\r_\alpha \two{\F_\nu} ) S(\F^\mu)  \buf 
\\ & \hspace{5cm}
\ot (\F_\mu    \r^\alpha \one{\F_\nu} ) \tad 
\left(\two{\bF^\gamma} \two{k} \bF_\rho \uf S(\bF^\sigma) S(\one{\bF_\gamma})\right)
\\
&= \F^\nu  
\one{\bF^\gamma} \one{k}  \bF^\rho \uf S( \F^\mu \r_\alpha \two{\F_\nu}\two{\bF_\gamma} \bF_\sigma)   \buf
\ot (\F_\mu    \r^\alpha \one{\F_\nu} ) \tad 
\left(\two{\bF^\gamma} \two{k} \bF_\rho \uf S(\one{\bF_\gamma} \bF^\sigma) \right)
\\
&= \F^\nu  
\bF^\delta \F^\gamma \one{k}  \bF^\rho \uf S( \F^\mu \r_\alpha \two{\F_\nu}
\three{\bF_\delta} \two{\bF_\beta} 
 \bF_\sigma)   \buf
\\
& \hspace{5cm}
\ot (\F_\mu    \r^\alpha \one{\F_\nu} ) \tad 
\left(\one{\bF_{\delta}}\bF^\beta \F_{\gamma}  \two{k} \bF_\rho \uf S(\two{\bF_\delta} \one{\bF_\beta}   \bF^\sigma) \right)
\\
&= \F^\nu  
\bF^\delta \F^\gamma \one{k}  \bF^\rho \uf S( \F^\mu \r_\alpha \two{\F_\nu}
\two{\bF_\delta} \two{\bF_\beta} 
 \bF_\sigma)   \buf 
\\ & \hspace{5cm}
\ot (\F_\mu    \r^\alpha \one{\F_\nu} \one{\bF_{\delta}} ) \tad \left(
\bF^\beta \F_{\gamma}  \two{k} \bF_\rho \uf S( \one{\bF_\beta}   \bF^\sigma) \right)
\end{align*}
where we have used the twist condition \eqref{twist-bF} to get the second last identity.
Next, being $\bF$ the inverse of $\F$, using again the twist condition \eqref{twist-bF}, this expression simplifies to 
\begin{align*}
&
\F^\gamma \one{k}  \bF^\rho \uf S( \F^\mu \r_\alpha  \two{\bF_\beta} 
 \bF_\sigma)   \buf
\ot (\F_\mu    \r^\alpha  ) \tad \left(
\bF^\beta \F_{\gamma}  \two{k} \bF_\rho \uf S( \one{\bF_\beta}   \bF^\sigma) \right)
\\
& \qquad =\F^\gamma \one{k}  \bF^\rho \uf S( \F^\mu \r_\alpha  
{\bF_{\beta}})   \buf
\ot (\F_\mu    \r^\alpha  ) \tad \left(
\one{\bF^{\beta}}  \bF^\sigma
\F_{\gamma}  \two{k} \bF_\rho \uf S( 
\two{\bF^\beta}\bF_{\sigma} 
) \right)
\\
& \qquad =\F^\gamma \one{k}  \bF^\rho \uf S( \F^\mu \r_\alpha  
{\bF_{\beta}})   \buf
\ot (\F_\mu    \r^\alpha  \bF^{\beta}) \tad \left(
\bF^\sigma
\F_{\gamma}  \two{k} \bF_\rho \uf S( 
\bF_{\sigma} 
) \right) .
\end{align*}
From the definition of $\dd^{-1}$, it is easy to see that the above can be rewritten as
\begin{align*}
\F^\gamma \one{k}  \bF^\rho \uf S( \F^\mu \r_\alpha  
{\bF_{\beta}})   \buf
\ot (\F_\mu    \r^\alpha  \bF^{\beta}) \tad \left(
\dd^{-1}(
\F_{\gamma}  \two{k} \bF_\rho 
) \right),
\end{align*}
and thus coincides with the left hand side, $(\id \ot \dd^{-1})\underline{\Delta_\F} (k)$. This concludes the proof.

\end{document}